\documentclass[11pt]{article}
\usepackage[cm]{fullpage}

\usepackage{amsmath,amscd, float,rotating}
\usepackage{pb-diagram}

\usepackage{latexsym}
\usepackage{amsfonts}
\usepackage{amsmath}
\usepackage{amssymb}

\numberwithin{equation}{section}
\newcommand{\qed}{\hfill \ensuremath{\Box}}

\def\XXint#1#2#3{{\setbox0=\hbox{$#1{#2#3}{\int}$}
\vcenter{\hbox{$#2#3$}}\kern-.5\wd0}}

\newcommand{\kod}{\textnormal{kod}}

\newcommand{\dbar}{\overline{\partial}}

\newcommand{\ddt}[1]{\frac{\partial #1}{\partial t}}

\newcommand{\ddbar}{\sqrt{-1}\partial\dbar}

\begin{document}
\newcounter{remark}
\newcounter{theor}
\setcounter{remark}{0} \setcounter{theor}{1}
\newtheorem{claim}{Claim}
\newtheorem{theorem}{Theorem}[section]
\newtheorem{proposition}{Proposition}[section]
\newtheorem{lemma}{Lemma}[section]
\newtheorem{definition}{Definition}[section]
\newtheorem{conjecture}{Conjecture}[section]
\newtheorem{corollary}{Corollary}[section]
\newenvironment{proof}[1][Proof]{\begin{trivlist}
\item[\hskip \labelsep {\bfseries #1}]}{\end{trivlist}}
\newenvironment{remark}[1][Remark]{\addtocounter{remark}{1} \begin{trivlist}
\item[\hskip \labelsep {\bfseries #1
\thesection.\theremark}]}{\end{trivlist}}
~

\bigskip

\centerline{\bf\Large THE K\"AHLER-RICCI FLOW THROUGH SINGULARITIES
 \footnote{Research supported in
part by National Science Foundation grants DMS-0847524 and DMS-0804095. The first named author is also supported in part by a Sloan Foundation Fellowship.}}

\bigskip
\begin{center}{\large Jian Song$^{*}$   ~  and  ~ Gang Tian$^{\dagger}$}

\end{center}
\bigskip
\bigskip
\noindent
{\bf Abstract} \  We prove the existence and uniqueness of the weak K\"ahler-Ricci flow on projective varieties with log terminal singularities. It is also shown that the weak K\"ahler-Ricci flow can be uniquely continued through divisorial contractions and flips if they exist. We then propose an analytic version of the Minimal Model Program with  Ricci flow.

\bigskip
\bigskip

{\footnotesize  \tableofcontents}

\bigskip
\bigskip


\section{Introduction}\label{intro}

It has been the subject of intensive study over the last few
decades to understand the existence of canonical K\"ahler metrics of Einstein type on a compact K\"ahler
manifold, following Yau's solution to the Calabi conjecture (cf.
\cite{Y2}, \cite{Au}, \cite{Y3}, \cite{T1}, \cite{T2}). The Ricci flow (cf. \cite{Ha}, \cite{Ch}) provides a canonical deformation of K\"ahler metrics toward such canonical metrics. Cao \cite{C} gives an alternative proof
of the existence of K\"ahler-Einstein metrics on a compact
K\"ahler manifold with numerically trivial or ample canonical bundle by the K\"ahler-Ricci flow.
However, most projective manifolds do not have a numerically definite or
trivial canonical bundle.
It is a natural question to ask if
there exist any well-defined canonical metrics on these manifolds
or on varieties canonically associated to them. A projective variety is minimal if its canonical bundle is nef (numerically effective) and many results have been obtained on the K\"ahler-Ricci flow on minimal varieties.
Tsuji \cite{Ts} applies the K\"ahler-Ricci flow and proves the existence of a canonical singular K\"ahler-Einstein metric on a minimal projective manifold of general type. It is the first attempt to relate the K\"ahler-Ricci flow and canonical metrics to the Minimal Model Program. Since then, many interesting results have been achieved in this direction. The long time existence of the K\"ahler-Ricci flow on a minimal projective manifold with any initial K\"ahler metric is established in \cite{TiZha}. The regularity problem of the canonical singular K\"ahler-Einstein metrics on minimal projective manifolds of general type is intensively studied in \cite{EGZ1} and \cite{Z1} independently.
If the minimal projective manifold has positive Kodaira dimension and it is not of general type, it admits an Iitaka fibration over its canonical model.  The authors define on the canonical model a new family of generalized K\"ahler-Einstein metrics twisted by a canonical form of Weil-Petersson type from the fibration structure (\cite{SoT1}, \cite{SoT2}). It is also proved by the authors that the normalized K\"ahler-Ricci flow converges to such a canonical metric if the canonical bundle is semi-ample (\cite{SoT1}, \cite{SoT2}).

Let $X$ be an $n$-dimensional projective manifold. We consider the following unnormalized K\"ahler-Ricci flow starting with a K\"ahler metric $\omega_0 \in H^{1,1}(X, \mathbf{R}) \cap H^2(X, \mathbf{Z})$.
\begin{equation}\label{ricciflow}
\left\{
\begin{array}{l}
{ \displaystyle \ddt{}\omega =-Ric(\omega) }\\
\\
 \omega|_{t=0}= \omega_0 .
\end{array} \right.
\end{equation}

The unnormalized K\"ahler-Ricci flow (\ref{ricciflow}) has long time existence if $X$ is a minimal model, i.e., the canonical bundle $K_X$ is nef (\cite{TiZha}).

%
If $K_X$ is not nef, the unnormalized K\"ahler-Ricci flow (\ref{ricciflow}) must become singular at certain time $T_0>0$. At time $T_0$,  either the flow develops singularities on a subvariety of $X$ or $X$ admits a Fano fibration and the flow is expected to collapse along the fibres. For the first case, the subvariety where the singularities appear is exactly where $K_X$ is negative. The flow then might perform an analytic/geometric surgery equivalent to an algebraic surgery such as a divisorial contraction or a flip, and replace $X$ by a new projective variety $X'$. Hopefully, the flow can be continued on $X'$, which usually has mild singularities. The main goal of the paper is to define the K\"ahler-Ricci flow on singular varieties such as $X'$ and to construct analytic surgeries for the K\"ahler-Ricci flow.
If the second case occurs and the flow (\ref{ricciflow}) collapses onto a new projective variety $X''$, we expect the flow can be continued on the base $X''$.  Heuristically, we can repeat the above procedures until either the flow exists for all time or it collapses to a point. If the flow exists for all time, it should converge to a generalized K\"ahler-Einstein metric on its canonical model or a Ricci-flat metric on its minimal model after normalization if we assume the abundance conjecture. Eventually, we arrive at the final case when $X$ is Fano and it is conjectured by the second named author that the unnormalized K\"ahler-Ricci flow becomes extinct in finite time after surgery if and only if it is birationally equivalent to a Fano variety \cite{T3}. The conjecture is proved by the first named author for smooth solutions of the flow \cite{So}.

In general, the varieties obtained from divisorial contractions and flips have mild singularities. Also we expect that the analytic surgeries performed by the K\"ahler-Ricci flow coincide with the algebraic surgeries as divisorial contractions and flips.  Therefore we can not avoid singularities if the K\"ahler-Ricci flow can  be indeed continued through surgeries.  We then must define the K\"ahler-Ricci flow on projective varieties with singularities. 

We confine ourselves in the category of singularities considered in the Minimal Model Program because such singularities are rather mild and they do not get worse after divisorial contractions or flips are performed. The precise definition is given in Section \ref{2.3} for a $\mathbf{Q}$-factorial projective variety with log terminal singularities. Roughly speaking, let $X$ be such a projective normal variety and $\pi: \tilde{X} \rightarrow X$ be a resolution of singularity, then the pullback of any smooth volume form on $X$ is integrable on the nonsingular model $\tilde{X}$.

Our first theorem proves the existence and uniqueness for the K\"ahler-Ricci flow on projective varieties with log terminal singularities. Furthermore, it establishes a smoothing property for the K\"ahler-Ricci flow if the initial data is not smooth. 

\bigskip

\noindent{\bf Theorem A.1~} {\it Let $X$ be a $\mathbf{Q}$-factorial projective variety  with log terminal singularities and $H$ be an ample $\mathbf{Q}$-divisor on $X$. Let $$T_0 = \sup\{ t>0 ~|~ H + t K_X ~is~ nef~\}.$$
If $\omega_0 \in \mathcal{K}_{H,p}(X)$ for some $p>1$, then there exists a unique solution $\omega$ of the unnormalized weak K\"ahler-Ricci flow (\ref{ricciflow}) starting with $\omega_0$ for $t\in [0, T_0)$.

Furthermore, if $\Omega$ is a smooth volume form on $X$, then for any $T\in (0, T_0)$, there exists $C>0$ such that on $[0, T]\times X$,

\begin{equation}
e^{-\frac{C}{t} } \Omega \leq \omega^n \leq e^{\frac{C}{t} }\Omega.
\end{equation}

}
\bigskip

The definitions are given in Section \ref{4.1} for $\mathcal{K}_{H,p}(X)$ (Definition \ref{khp}) and the weak K\"ahler-Ricci flow (Definition \ref{weakdef}). Theorem A.1 shows that the K\"ahler-Ricci flow can start with a K\"ahler current which admits bounded local potential and an $L^p$ Monge-Amp\`ere mass for some $p>1$. It gives the short time existence for the weak K\"ahler-Ricci flow. 
%
%
%
Furthermore, it smoothes out the initial current in the sense that the flow becomes smooth on the nonsingular part of $X$ once $t>0$ and the evolving metrics always admit bounded local potentials for any $t\in (0, T_0)$. In particular, if $X$ is nonsingular, the flow becomes the usual K\"ahler-Ricci flow with smooth solutions on $(0, T_0)\times X$. In particular, $T_0$ is exactly the first singular time for the unnormalized weak K\"ahler-Ricci flow.

It is not clear how to define metrics on a singular variety $X$ with reasonable regularity and curvature conditions in general. One natural choice is the restriction of the Fubini-Study metric $\omega_{FS}$ for some projective embedding of $X$, if $X$ is normal and projective. It is indeed a smooth metric on $X$, however, even the scalar curvature of $\omega_{FS}$ might blow up near the singularities of $X$. More seriously, $(\omega_{FS})^n$ might not be a smooth volume form on $X$ in general, although it is a smooth non-negative $(n,n)$-form on $X$ (see Section \ref{4.3} for more detailed discussions). Theorem A.1 shows that the volume form of the corresponding solutions of the weak K\"ahler-Ricci flow becomes equivalent to a smooth volume form immediately for $t>0$. We speculate that the weak K\"ahler-Ricci flow produces metrics on $X$ with reasonably good geometric conditions. For example, given a normal projective orbifold $X$ embedded in some projective space $\mathbf{CP}^N$, the Fubini-Study metric $\omega_{FS}$ is in general not a smooth orbifold K\"ahler metric on $X$. If we start the K\"ahler-Ricci flow with $\omega_{FS}$, the evolving metrics immediately become smooth orbifold K\"ahler metrics on $X$.

We also make a remark that the assumption of $\mathbf{Q}$-factoriality can be weakened. In fact, the theorems still hold with slight modification if both the initial divisor $H$ and $K_X$ are $\mathbf{Q}$-Cartier.

The following theorem shows that the K\"ahler-Ricci flow can be defined on smooth projective varieties if the initial class is not K\"ahler.

\bigskip
\noindent{\bf Theorem A.2~} {\it
Let $X$ be a non-singular projective variety and $H$ be a big and semi-ample $\mathbf{Q}$-divisor on $X$. Suppose that $$T_0 = \sup\{ t>0 ~|~ H + t K_X ~{is~ nef~}\}>0.$$
If $\omega_0 \in \mathcal{K}_{H,p}(X)$ for some $p>1$, then there exists a unique solution $\omega$ of the unnormalized weak K\"ahler-Ricci flow (\ref{ricciflow}) for $t\in [0, T_0)$.

Furthermore, for any $t\in(0, T_0)$, there exists $C(t)>0$ such that
\begin{equation}
||S(\omega(t, \cdot))||_{L^\infty(X)}\leq C(t),
\end{equation}
where $S(\omega(t, \cdot))$ is the scalar curvature of $\omega(t, \cdot)$.
}
\bigskip

The scalar curvature $S$ is defined on a Zariski open set of $X$ away from the exceptional locus of $H$. Theorem A.2 shows that for each $t\in (0,T_0)$, the scalar curvature is uniformly bounded.

Theorem A.2 immediately implies the following corollary. It turns out that at each $t\in(0, T_0)$, the evolving metric has  bounded scalar curvature on a projective variety which admits a crepant resolution.

\bigskip

\noindent{\bf Corollary A.3~} {\it
Let $X$ be a $\mathbf{Q}$-factorial projective variety  with crepant singularities, $H$ be an ample $\mathbf{Q}$-divisor on $X$ and $$T_0 = \sup\{ t>0 ~|~ H + t K_X ~is~ nef~\}.$$
If $\omega_0 \in \mathcal{K}_{H,p}(X)$ for some $p>1$, then there exists a unique solution $\omega$ of the unnormalized weak K\"ahler-Ricci flow (\ref{ricciflow}) for $t\in [0, T_0)$.

Furthermore, for any $t\in(0, T_0)$, there exists $C(t)>0$ such that
\begin{equation}
||S(\omega(t, \cdot))||_{L^\infty(X)}\leq C(t),
\end{equation}
where $S(\omega(t, \cdot))$ is the scalar curvature of $\omega(t, \cdot)$.
}
\bigskip

From now on, we always assume $X$ is a $\mathbf{Q}$-factorial projective variety  with log terminal singularities and $H$ be an ample $\mathbf{Q}$-divisor on $X$. Let $$T_0 = \sup\{ t>0 ~|~ H + t K_X ~is~ nef~\}$$ be the first singular time the unnormalized weak K\"ahler-Ricci flow (\ref{ricciflow}) for $t\in [0, T_0)$ starting with $\omega_0 \in \mathcal{K}_{H, p}(X)$ for some $p>1$.

Theorem A.1 gives the short time existence of the weak unnormalized K\"ahler-Ricci flow and the first singular time $T_0$ is exactly when the K\"ahler class of the evolving metrics stops being nef. If $T_0<\infty$ and the limiting K\"ahler class is big, there is a contraction morphism $\pi: X \rightarrow Y$ uniquely associated to the limiting divisor $H+T_0K_X$. Let $\overline{NE}(X)$ be the closure of the convex cone that consists of the classes of effective curves on $X$. If the morphism $\pi$ contracts exactly one extremal ray of $\overline{NE}(X)$, the recent result of \cite{BCHM} and \cite{HM} shows that either $\pi$ contracts a divisor or there exists a unique flip associated to $\pi$ (see Definition \ref{flip} for a flip).

Since the weak unnormalized K\"ahler-Ricci flow cannot be continued on $X$ at the singular time $T_0$,  we have to replace $X$ by another variety $X'$ and continue the flow on $X'$. Our next main result is to relate the finite time singularities of the unnormalized K\"ahler-Ricci flow (\ref{ricciflow}) to divisorial contractions and flips in the Minimal Model Program.

\bigskip

\noindent{\bf Theorem B.1~} {\it Let $\omega$ be the unique solution  of the unnormalized weak K\"ahler-Ricci flow (\ref{ricciflow}) for $t\in [0, T_0)$ starting with $\omega_0 \in \mathcal{K}_{H, p}(X)$ for some $p>1$.
 Suppose that $H+T_0K_X$ is big and the morphism $\pi: X \rightarrow Y$ induced by the semi-ample divisor $H_{T_0}= H+T_0K_X$ contracts exactly one extremal ray of  $\overline{NE}(X)$.

\begin{enumerate}

\item If $\pi$ is a divisorial contraction, then there exists $\omega_{Y,0} \in \mathcal{K}_{H_Y, p'}(Y) \cap C^\infty(Y_{reg} \setminus \pi(Exc(\pi)) )$ for some $p'>1$ such that,  $\omega(t, \cdot)$ converges to $\pi^* \omega_{Y,0}$ in $C^\infty(X_{reg}\setminus Exc(\pi))$-topology as $t\rightarrow T_0$, where $H_Y= \pi_* H_{T_0}$.

Furthermore, the unnormalized weak K\"ahler-Ricci flow (\ref{ricciflow})  can be continued on $Y$ with the initial K\"ahler current $\omega_{Y,0}$.

\item  If $\pi$ is a small contraction and there exists a flip \begin{equation}
\begin{diagram}\label{diag1}
\node{X} \arrow{se,b,}{\pi}  \arrow[2]{e,t,..}{\check{\pi}^{-1} }     \node[2]{X^+} \arrow{sw,r}{\pi^+} \\
\node[2]{Y}
\end{diagram},
\end{equation}

then there exists $\omega_{X^+,0} \in \mathcal{K}_{H_{X^+}, p'}(X^+)$ for some $p'>1$, such that $\omega(t, \cdot)$ converges to $(\check{\pi}^{-1})^* \omega_{X^+,0}$ in $C^\infty(X_{reg}\setminus Exc(\pi))$-topology, where $H_{X^+}$ is the strict transformation of $H_{T_0}$ by $\check{\pi}$.

Furthermore, $\omega_{X^+, 0}$ is smooth outside the singularities of $X^+$ and where the flip is performed, and the unnormalized weak K\"ahler-Ricci flow (\ref{ricciflow}) can be continued on $X^+$ with the initial K\"ahler current $\omega_{X^+,0}$.

\end{enumerate}

}

\bigskip
Here $Exc(\pi)$ denotes  the exceptional locus of the morphism $\pi$ and $X_{reg}$ denotes the nonsingular part of $X$. In summary, we have the following corollary. $(\check{\pi}^{-1} )^* \omega_{X^+, 0}$ is defined by pulling back the local potentials of $\omega_{X^+, 0}$. It is well-defined because the local potential of $\omega_{X^+, 0}$ can be chosen to be constant along each connected fibre of $\pi^+$ and in particular, $(\check{\pi}^{-1} )^* \omega_{X^+, 0} \in \mathcal{K}_{H_{T_0}, p}(X)$.

\bigskip

\noindent{\bf Corollary B.2~} {\it The unnormalized K\"ahler-Ricci flow can be continued through divisorial contractions and flips.

}

\bigskip

The Minimal Model Program is successful in dimension three by Mori's work and the recent works have (c.f. \cite{BCHM}, \cite{Si}) led to proving  the finite generation of canonical rings. The deformation of the K\"ahler classes along the unnormalized K\"ahler-Ricci flow is in line with the Minimal Model Program with Scaling (MMP with scaling) proposed in \cite{BCHM}. It is also proved in \cite{BCHM} that MMP with scaling terminates after finitely many divisorial contractions and flips if the variety $X$ is of general type. 

A good initial divisor $H$ means that there are finitely many singular times and the contraction morphism at each singular time only contracts one extremal ray if the unnormalized K\"ahler-Ricci flow (\ref{ricciflow}) starts with $H$. We refer the readers to Section \ref{5.1} for the precise definition (Definition \ref{goodinitial}). In particular, good initial divisors always exist if $\dim X =2$ and $\kod(X)\geq 0$,  then the normalized K\"ahler-Ricci flow with a good initial divisor converges to the canonical model or the minimal model of $X$ coupled with a generalized K\"ahler-Einstein metric. It is possible that a general ample $\mathbf{Q}$-divisor $H$ on $X$ is a good initial divisor since MMP with scaling terminates for $X$.



\bigskip
\noindent{\bf Theorem C.1~} {\it Let $X$ be a projective $\mathbf{Q}$-factorial variety of general type with log terminal singularities. If  $H$ is a good initial divisor on $X$, then the normalized weak K\"ahler-Ricci flow \begin{equation}\label{nlflow}
\ddt{\omega} = -Ric(\omega) - \omega
\end{equation}
 starting with any initial K\"ahler current in $\mathcal{K}_{H, p}(X)$ for some $p>1$ exists for $t\in [0,\infty)$ and  it replaces $X$ by its minimal model $X_{min}$ after finitely many surgeries. Furthermore, the normalized K\"ahler-Ricci  flow 
converges in distribution to the unique K\"ahler-Eintein metric $\omega_{KE}$ on its canonical model $X_{can}$.

}
\bigskip

Theorem C.1 gives the general philosophy of the analytic Minimal Model Program with Ricci Flow. The K\"ahler-Ricci flow deforms a given projective variety $X$ to its minimal model $X_{min}$ in finite time after finitely many metric surgeries. Then $X_{min}$ is deformed to the canonical model $X_{can}$ coupled with a generalized K\"ahler-Einstein metric by the flow after normalization. We also remark that the flow converges in the sense of distribution globally and in the $C^\infty$-topology away from the singularities of $X_{min}$ and the exceptional locus of the pluricanonical system. Certainly, it is desired that the convergence should be in the sense of Gromov-Hausdorff. We also remark that when $X$ is a nonsingular minimal model of general type, the convergence of the normalized K\"ahler-Ricci flow is proved in \cite{Ts} and \cite{TiZha}.
\bigskip

The organization of the paper is the following. In Section \ref{2}, we set up the basic notations for degenerate complex Monge-Amp\`ere equations and algebraic singularities in the minimal model theory. In Section \ref{3}, we solve a special family of degenerate parabolic Monge-Amp\`ere equations on projective manifolds. In Section \ref{4}, we apply the results in Section \ref{3} to prove Theorem A.1, Theorem A.2 and Corollary A.3 for the short time existence of the weak K\"ahler-Ricci flow. In Section \ref{5}, Theorem B.1 and Corollary B.2 are proved for the weak K\"ahler-Ricci flow through singularities. We also prove Theorem C.1 for long time existence and convergence. Finally in Section \ref{6}, we propose an analytic Minimal Model Program with Ricci Flow.


\section{Preliminaries } \label{2}


\subsection{Kodaira dimension and canonical measures} \label{2.1}

Let $X$ be an $n$-dimensional compact complex projective manifold and $ L \rightarrow X $ a holomorphic line bundle over $X$. Let $N( L) $ be the semi-group defined by $$N(L) = \{ m\in \mathbf{N} ~ | ~ H^0(X , L^m) \neq  0 \}.$$

Given any $m\in N(L)$,  the linear system $|L^m|= \mathbf{P} H^0(X, L^m)$ induces a rational map $\Phi_m$
$$\Phi_m ~ : ~ X  \dashrightarrow \mathbf{CP}^{d_m}$$
by any basis $\{ \sigma_{m,0},~ \sigma_{m,1}, ~... ~, \sigma_{m, d_m} \}$ of  $H^0(X, L^m )$, where 

$$\Phi_m (z) = \left[ \sigma_{m,0},~ \sigma_{m,1}, ~... ~, \sigma_{m, d_m} (z) \right], $$
and $d_m + 1 = \dim H^0(X, L^m) $.
Let $Y_m = \overline{\Phi_m (X)} \subset \mathbf{CP}^{d_m}$ be  the closure of the image of $\Phi_m$.

\begin{definition} The Iitaka dimension of $L$ is defined to be

$$\kappa(X, L) = \max _{ m\in N(L)} \{ \dim Y_m \}$$ if $N(L) \neq \phi$,  and
$\kappa(X, L) = -\infty $ if $N(L)= \phi$.

\end{definition}

\begin{definition}
Let $X$ be a projective manifold and $K_X$ be the canonical line bundle over $X$. Then the Kodaira dimension $\kod(X)$ of $X$ is defined to be

$$\kod(X) = \kappa (X, K_X).$$

\end{definition}

The Kodaira dimension  is a birational invariant of a projective variety and the Kodaira dimension of a singular variety is  equal to that of its smooth model.

\begin{definition}

Let $L \rightarrow X$ be a holomorphic line bundle  over a compact projective manifold $X$. $L$ is called nef if $L\cdot C\geq 0$ for any curve $C$ on $X$ and $L$ is called semi-ample if $L^m$ is globally generated for some $m>0$.

\end{definition}

For any $m\in \mathbf{N}$ such that $L^m$ is globally generated,  the linear system $|L^m|$ induces a holomorphic map $\Phi_m$
$$\Phi_m ~ : ~ X  \rightarrow \mathbf{CP}^{d_m}$$
by any basis  of  $H^0(X, L^m )$.
Let $Y_m = \Phi_m (X)$ and so $\Phi_m$ can be  considered as

$$ \Phi_m ~ : ~ X  \rightarrow Y_m .$$

The following theorem is well-known (cf. \cite{La, U}).

\begin{theorem}\label{safibration}

Let $L \rightarrow X$ be a semi-ample line bundle over an algebraic manifold $X$. Then there is an algebraic fibre space
$$\Phi_{\infty} : X \rightarrow Y $$
such that for any sufficiently large integer $ m $ with $L^m$ being globally generated,
$$Y_m = Y ~~~~ and ~~~~ \Phi_m = \Phi_{\infty}, $$
where $Y$ is a normal projective variety.
Furthermore, there exists an ample line bundle $A$ on $Y$ such that $ L^m = (\Phi_{\infty})^* A$.

\end{theorem}

If $L$ is semi-ample, the graded ring $R(X, L) = \oplus_{m\geq 0} H^0( X, L^m)$ is finitely generated and so it is the coordinate ring of $Y$.

Let $X$ be an $n$-dimensional projective manifold. It is recently proved in \cite{BCHM} and \cite{Si}  independently that the canonical ring $R(X, K_X)$ is finitely generated if $X$ is of general type. 
Then the canonical ring induces a rational map from $X$ to its unique canonical model $X_{can}$.  The following theorem is proved in  \cite{EGZ1} when $X$ is of general type (also see \cite{Ts}, \cite{TiZha} for minimal models of general type) and in \cite{SoT2} when $X$ admits an Iitaka fibration over $X_{can}$.

\begin{theorem} \label{canmetric} Let $X$ be an $n$-dimensional projective manifold with $R(X, K_X)$ being finitely generated.

\begin{enumerate}

\item $\kod(X)=n$, then there there exists a unique K\"ahler current $\omega_{KE} \in [K_{X_{can}}]$ with bounded local potential  satisfying the K\"ahler-Einstein equation

\begin{equation}
Ric(\omega_{KE}) = - \omega_{KE}.
\end{equation}

\item $0<\kod(X)<n$, $X$ admits a rational fibration over $X_{can}$ whose general fibre has Kodaira dimension $0$. There exists a unique K\"ahler current $\omega_{can} \in [ K_{X_{can}} +L_{X/X_{can}}]$ such that

\begin{equation}
Ric(\omega_{can}) = -\omega_{can} + \omega_{WP},
\end{equation}
where $L_{X/X_{can}}$ is relative dualizing sheaf and $\omega_{WP}$ is a canonical current of Weil-Peterson type induced from the Calabi-Yau fibration.

\end{enumerate}

\end{theorem}
The closed current $\omega_{WP}$ is exactly the pullback of the Weil-Peterson metric on the moduli space of Calabi-Yau varieties associated to the fibres if $X$ is a smooth minimal model and $K_X$ is semi-ample. When $X$ is not minimal, the general fibre is not necessary a Calabi-Yau variety. One can still define $\omega_{WP}$ as a $L^2$-metric on the deformation space for varieties of $0$ Kodaira dimension. We refer the readers to the precise definition in \cite{SoT2}.

If $\kod(X)=0$, then so there exists a holomorphic volume form $\Omega= (\eta \otimes \overline{\eta})^{1/m}$ for some holomorphic section $\eta\in H^0(X, mK_X)$. It is proved in \cite{SoT2} that   for any ample divisor $H$ on $X$, there exists a K\"ahler current $\omega_{CY}\in H$ with bounded local potential such that
\begin{equation}
(\omega_{CY} )^n = c \Omega
\end{equation}
for some positive constant $c>0$.
Therefore $Ric(\omega_{CY})=0$ outside the stable base locus of the pluricanonical system of $X$. The existence of singular Ricci-flat K\"ahler metrics is proved in \cite{EGZ1} on singular Calabi-Yau varieties.

Such metrics are the unique canonical metrics on projective varieties of non-negative Kodaira dimension and the generalized K\"ahler-Einstein equations can be viewed as an analytic version of the adjunction formula. They are candidates for the limiting metrics of the K\"ahler-Ricci flow.


\subsection{Complex Monge-Amp\`ere equations}  \label{2.2}

In this section, we review some of the important results in degenerate complex Monge-Amp\`ere equations developed by Kolodziej \cite{Ko1} and many others (\cite{Z1}, \cite{EGZ1}, \cite{DP}, \cite{EGZ2}). We start with some basic notations.

\begin{definition} Let $X$ be an $n$-dimensional K\"ahler manifold and $\omega$ be a closed semi-positive $(1,1)$-current on $X$.

\begin{enumerate}

\item $\omega$ is K\"ahler if it is positive.

\item $\omega$ is called big if $[\omega]^n = \int_X \omega^n >0$.

\item $\omega$ is called a K\"ahler current if it is big.

\end{enumerate}

\end{definition}

If $\omega$ is a K\"ahler current with bounded local potential on $X$, the corresponding volume current $\omega^n$ is uniquely well-defined by the standard pluripotential theory. 

\begin{definition} Let $\omega$ be a K\"ahler current with bounded local potential on $X$. A quasi-plurisubharmonic function associated to $\omega$ is an upper semi-continous function $\varphi: X \rightarrow [-\infty, \infty)$ such that $\omega+ \ddbar\varphi \geq 0$.  We denote by $PSH(X, \omega)$ the set of all quasi-plurisubharmonic functions associated to $\omega$ on $X$.

\end{definition}

In \cite{Ko1}, Kolodziej proves the fundamental theorem on the existence of continuous solutions to the Monge-Amp\`ere equation $ (\omega + \ddbar \varphi)^n = F \omega^n $, where $\omega$ is a K\"ahler form  and $F\in L^p(X, \omega^n)$ for some $p>1$ is non-negative.
Its generalization was independently carried out in [Zh] and [EyGuZe1].
They prove  that there is a bounded solution when $\omega$ is
semi-positive and big. A detailed proof for the
continuity of the solution
was given in [DZ] (also see [Zh] for an earlier sketch of proof).
These generalizations are summarized in the
following.
\begin{theorem}\label{zhang} Let $X$ be an $n$-dimensional K\"ahler manifold and let $\omega$ be a K\"ahler current with bounded local potential. Then there exists a unique solution $\varphi\in PSH(X, \omega)\cap L^\infty(X)$ solving the following Monge-Amp\`ere equation

$$ (\omega + \ddbar \varphi)^n =  F\Omega, $$
where $\Omega>0$ is a smooth volume form on $X$, $F\in L^{p}(X, \Omega)$ for some $p>1$ and $\int_X F\Omega = \int_X \omega^n$.
\end{theorem}

In \cite{Ko2}, Kolodziej proves the stability result for solutions of the complex Monge-Ampere equations for K\"ahler classes. It is later improved by Dinew and Zhang \cite{DZ} (also see \cite{DP} more general cases) for big and semi-ample classes. The following is a version of their result.

\begin{theorem}\label{stability} Let $X$ be an $n$-dimensional compact K\"ahler manifold. Suppose $L\rightarrow X$ is a semi-ample line bundle and $\omega \in c_1(L)$ is a smooth K\"ahler current. Let $\Omega$ be a smooth volume form on $X$. For any non-negative functions $f$ and $g \in L^p(X, \Omega)$ for some $p>1$ with $\int_X f \Omega = \int_X g \Omega$, there exist $\varphi$ and $\psi \in PSH(X, \omega) \cap L^\infty(X)$ solving $$ (\omega+ \ddbar\varphi)^n =f \Omega, ~~~~~ (\omega + \ddbar \psi)^n = g \Omega$$
with $$\max_X (\varphi - \psi) = \max_X (\psi - \varphi).$$

Then for any $\epsilon>0$, there exists $C>0$ depending on $\epsilon$ and $p$, $|| f||_{L^p(X, \Omega)}$ and $||g||_{L^p(X, \Omega)}$ such that

\begin{equation}\label{zhangstable}
||\varphi - \psi||_{L^\infty(X)} \leq C || f - g||_{L^1(X, \Omega)}^{\frac{1}{n+3+\epsilon}}.
\end{equation}

\end{theorem}

Theorem \ref{stability} can be generalized for the case where the right hand side of the Monge-Ampere equations contains terms such $e^{\varphi}$ by Kolodziej's argument in \cite{Ko2}. Theorem \ref{stability} also holds uniformly for certain family of $\omega$, such as $\omega + \epsilon \chi$ with a fixed K\"ahler metric $\chi$ and $\epsilon \in [0,1]$. Also the sharper exponents are obtained in \cite{DZ} and \cite{DP}.


\subsection{Singularities}  \label{2.3}

We will have to study the behavior of the K\"ahler-Ricci flow on normal projective varieties with singularities because the original smooth manifold might be replaced by varieties with mild singularities through surgery along the flow.

The pluripotential theory on normal varieties has been extensively studied (cf \cite{FN}) . Let $X$ be a normal  variety. A function $f$ on $X$ is continuous (or smooth) if  $f$ can be extended to a continuous (or smooth) function in a local embedding from $X$ to $\mathbf{C}^N$. A plurisubharmonic function  is an upper semi-continuous function $\varphi: V \rightarrow [-\infty, \infty)$ which locally extends to a plurisubharmonic function in a local embedding from $X$ to $\mathbf{C}^N$.  By the work of \cite
, any bounded plurisubharmonic function on $X_{reg}$, the nonsingular part of $X$,  can be uniquely extended to a plurisubharmonic function on $X$.
Let $X$ be a normal projective variety and $\omega$ be a semi-positive closed $(1,1)$-current  on $X$. We let $PSH(X, \omega) $ be the set of all upper semi-continuous functions $\varphi: X \rightarrow [-\infty, \infty)$ such that $\omega+ \ddbar \varphi\geq 0$.

In this paper, we confine our discussions to  projective varieties with mild singularities which are allowed in the Minimal Model Program in algebraic geometry.

\begin{definition} \cite{KMM} Let $X$ be a normal projective variety such that $K_X$ is a $\mathbf{Q}$-Cartier divisor. Let $ \pi : \tilde{X} \rightarrow X$ be a resolution and $\{E_i\}_{i=1}^p$ the irreducible components of the exceptional locus $Exc(\pi)$ of $\pi$. There there exists a unique collection $a_i\in \mathbf{Q}$ such that
$$K_{\tilde{X}} = \pi^* K_X + \sum_{i=1}^{ p } a_i E_i .$$ Then $X$ is said to have

\begin{enumerate}

\item[$\bullet$] terminal singularities if  $a_i >0$, for all $i$.

\item[$\bullet$] canonical singularities if $a_i \geq 0$, for all $i $.

\item[$\bullet$]  log terminal singularities if $a_i > -1$, for all $i $.

\item[$\bullet$]  log canonical singularities if $a_i \geq -1$, for all $i$.

\end{enumerate}

\end{definition}

Terminal, canonical and log terminal singularities are always rational, while log canonical singularities are not necessarily rational. We can always assume that the resolution $\pi$ is good enough such that the exceptional locus is a simple normal crossing divisor.

\begin{definition}  A variety $X$ is $\mathbf{Q}$-factorial if any $\mathbf{Q}$-Weil divisor on $X$ is $\mathbf{Q}$-Cartier.

\end{definition}

Kodaira's lemma states that for any big and nef line divisor $H$ on $X$, there always exists an effective divisor $E$ such that $H-\epsilon E$ is ample for any sufficiently small $\epsilon>0$. Let $\pi: \tilde{X} \rightarrow X$ be a birational morphism between two projective varieties and $Exc(\pi)$ be the exceptional locus of $\pi$, where $\pi$ is not isomorphic.
The following proposition is a special case of Kodaira's lemma and the support of $E$ exactly coincides with $Exc(\pi)$ (see \cite{D}).
\begin{proposition}\label{qfac}
If $X$ is normal and $\mathbf{Q}$-factorial, then for any ample $\mathbf{Q}$-divisor $H$ on $X$, there exists an effective divisor $E$ on $\tilde{X}$ whose support is $Exc(\pi)$ and $\pi^*H - \epsilon E$ is ample for sufficiently small  $\epsilon>0$.
\end{proposition}

It is also well-known that $\mathbf{Q}$-factoriality is preserved after divisorial contractions and flips in the Minimal Model Program. $\mathbf{Q}$-factoriality is a necessary condition in our discussion because we need the canonical divisor to be a Cartier $\mathbf{Q}$-divisor in order to define a volume form appropriately.



\section{Monge-Amp\`ere flows} \label{3}


\subsection{Monge-Amp\`ere flows with rough initial data} \label{3.1}

In this section, we will prove the smoothing property of the K\"ahler-Ricci flow with rough initial data. We will assume that $X$ is an $n$-dimensional K\"ahler manifold.

\begin{definition}\label{lpmass}
 Suppose $\omega$ is a K\"ahler form and $\Omega$ is a smooth volume form on $X$. Then we define for $p\in (0, \infty]$,

$$PSH_p(X, \omega, \Omega)=\{\varphi\in PSH(X, \omega)\cap L^\infty(X)~|~ \frac{(\omega+\ddbar\varphi)^n }{\Omega} \in L^p(X)\}.$$



\end{definition}

Note that $(\omega+\ddbar\varphi)^n$ is a well-defined Monge-Amp\`ere mass for bounded $\omega$-psh function $\varphi$. In Definition \ref{lpmass}, the Monge-Amp\`ere mass $(\omega+\ddbar\varphi)^n$ must be absolutely continuous with respect to $\Omega$ in order to define $\frac{(\omega+\ddbar\varphi)^n }{\Omega} \in L^p(X)$. Suppose $\varphi_0\in PSH_p(X, \omega_0, \Omega)$ for some $p>1$. Let $$F= \frac{(\omega_0 + \ddbar\varphi_0)^n}{\Omega} \in L^p(X).$$ By Kolodziej's result \cite{Ko1}, $$\varphi_0\in C^0(X).$$

The following proposition shows that any element in $PSH_p(X, \omega, \Omega)$ for $p>1$ can be uniformly approximated by smooth quasi-plurisubharmonic functions.

\begin{proposition} \label{app2} There exist a sequence  $\{\varphi_{0,j}\}_{j=1}^\infty \subset PSH(X, \omega_0) \cap C^\infty(X)$ such that

\begin{equation}
\lim_{j\rightarrow \infty} ||\varphi_{0, j} - \varphi_0 ||_{L^\infty(X)} = 0.
\end{equation}

\end{proposition}

\begin{proof}

Recall that $C^\infty(X)$ is dense in $L^p(X)$. Therefore there exists a sequence of positive functions $\{F_j\}\in C^\infty(X)$ such that $\int_X F_j \Omega = \int_X F \Omega$ and

$$\lim_{j\rightarrow \infty} ||F_j -F||_{L^p(X)} = 0.$$

We then consider the solutions of the following Monge-Amp\`ere equations

\begin{equation}
(\omega_0 + \ddbar\varphi_{0, j})^n = F_j \Omega.
\end{equation}

Since $F_j \in C^\infty(X)$ and $F_j>0$, $\varphi_{0, j} \in PSH(X, \omega_0)\cap C^\infty(X)$. Without loss of generality, we can assume $$ \sup_X (\varphi_0 - \varphi_{0,j}) = \sup_X (\varphi_{0,j}- \varphi_0).$$

By the stability theorem of Kolodziej \cite{Ko2} (Theorem \ref{stability}), we have

$$||\varphi_{0, j} - \varphi_0||_{L^\infty(X)}\leq C ||F_j - F||_{L^1(X)}^{\frac{1}{n+4}}$$
where $C$ only depends on $||F_j||_{L^p(X)}$ and $ ||F||_{L^p(X)}$. The proposition follows easily.

\qed\end{proof}







Now we have a sequence of smooth K\"ahler forms $$\omega_{0, j} = \omega_0 + \ddbar \varphi_{0,j}.$$

Let $\chi= \ddbar\log \Omega \in [K_X]$ and $\omega_t = \omega_0 + t\chi$. By simple calculation,  one can show that the unnormalized K\"ahler-Ricci flow with the initial K\"ahler metric $\omega_{0, j}$ is equivalent to the following Monge-Amp\`ere flow

\begin{equation}\label{maflow_deg1}
\left\{
\begin{array}{rcl}
&&{ \displaystyle \ddt{\varphi_j} = \log
 \frac{ ( \omega_t + \ddbar \varphi_j ) ^n } {\Omega} }\\
&&\\
&& \varphi_j(0,\cdot)=\varphi_{0,j} .
\end{array} \right.
\end{equation}

We define $$T_0= \sup\{ t\geq 0~|~ [\omega_0] + t [K_X] ~is ~nef~ \}$$ to be the first time when the K\"ahler class stops being positive along the K\"ahler-Ricci flow.  It is well-known that $T_0>0$ and by the result of \cite{TiZha}, the Monge-Amp\`ere flow exists for $[0, T_0)$.

The following lemma shows that the Monge-Amp\'ere flows starting with $\varphi_{0, j}$ approximate the same flow starting with $\varphi_0$.

\begin{lemma} \label{cauchy1}For any $0<T < T_0$, there exists  $C>0$ such that for $t\in [0, T]$,

\begin{equation}
||\varphi_{j}||_{L^\infty(X)} \leq C.
\end{equation}

Furthermore, $\{\varphi_j\}$ is a Cauchy sequence in $L^\infty([0, T]\times X)$, i.e.,

\begin{equation}
\lim_{j, k\rightarrow \infty} ||\varphi_j - \varphi_k||_{L^\infty([0,T] \times X )} = 0.
\end{equation}

\end{lemma}

\begin{proof}

Applying  the maximum principle to $\varphi_j$, we can show that there exists $C>0$ such that $$    \sup_{[0, T]\times X} |\varphi_j| \leq T \sup_X | \log \frac{  \omega_t^n}{\Omega}|   + \sup_X | \varphi_{0, j} |\leq C.$$

Let $\psi_{j,k} = \varphi_j -\varphi_k$. Then $\psi_{j, k}$ satisfies the following equation

\begin{equation}\label{maflow_deg2}
\left\{
\begin{array}{rcl}
&&{ \displaystyle \ddt{\psi_{j,k} } = \log
 \frac{ ( \omega_t + \ddbar \varphi_{k} + \ddbar\psi_{j,k} ) ^n } {(\omega_t + \ddbar \varphi_k)^n } }\\
&&\\
&& \psi_{j,k}(0,\cdot)= \varphi_{0,j}-\varphi_{0,k} .
\end{array} \right.
\end{equation}

By the maximum principle, $$ \sup_{[0. T]\times X} |\varphi_j - \varphi_k| = \sup_{[0. T]\times X} |\psi_{j,k} |\leq \sup_X |\varphi_{0, j} - \varphi_{0,k}|.$$

Then

$$\lim_{j, k\rightarrow \infty} ||\varphi_j - \varphi_k||_{L^\infty([0,T] \times X )} \leq \lim_{j, k\rightarrow \infty} ||\varphi_{0,j} - \varphi_{0,k} ||_{L^\infty( X )}= 0.$$

\qed\end{proof}

We also can bound the volume form along the Monge-Amp\`ere flow, even though the initial volume form is only in $L^p(X)$.

\begin{lemma}\label{volsm} For any $0<T< T_0$, there exists $C>0$, such that for $t \in [0, T]$.

\begin{equation}
\frac{t^n}{C}  \leq  \frac{ (\omega_t + \ddbar\varphi_j)^n}{\Omega}  \leq   e^{ \frac{ C } { t } }.
\end{equation}

\end{lemma}

\begin{proof} Let $\Delta_j$ be the Laplacian operator associated to the K\"ahler form $\omega_j= \omega_t + \ddbar\varphi_j$. Straightforward calculations show that $( \ddt{} - \Delta_j ) \dot{\varphi}_j = tr_{\omega_j} (\chi)$.

Let $H^+ =  t \dot{\varphi_j} -  \varphi_j$. Then $H^+(0, \cdot) = - \varphi_j$ is uniformly bounded and

$$(\ddt{} - \Delta_j ) H^+ = - tr_{\omega_j} (\omega_t - t \chi) + n = - tr_{\omega_j} (\omega_0) + n\leq n.$$
By the maximum principle, $H^+$ is uniformly bounded from above for $t\in[0, T]$.

Let $H^- = \dot{\varphi}_j+ A \varphi_j - n \log t$. Then $H^-(t, \cdot) $ tends to $\infty$ uniformly as $t\rightarrow 0^+$ and there exist constants $C_1$, $C_2$ and $C_3>0$ such that

\begin{eqnarray*}
(\ddt{} - \Delta_j ) H^- &=&  tr_{\omega_j} (A\omega_t + \chi) + A\dot{\varphi_j} - \frac{n}{t}- An\\
%
%
&\geq&C_1 (\frac{\omega_0^n}{\omega_j^n})^{\frac{1}{n}}+ A \log \frac{\omega_j^n}{\Omega} - \frac{n}{t}- An\\
&\geq& C_2  (\frac{\omega_0^n}{\omega_j^n})^{\frac{1}{n}} -\frac{C_3}{t},
\end{eqnarray*}
if $A$ is chosen sufficiently large such that $A\omega_t +  \chi \geq \omega_0$ for $t\in [0, T]$.  Then at the minimal point of $H^-$, the maximum principle gives $$\omega_j^n \geq C_4 t^n \Omega.$$
It easily follows that $H^-$ is uniformly bounded from below for $t\in [0,T]$.

Since $\varphi_j$ is uniformly bounded for $t\in [0,T]$, the lemma is proved.

\qed
\end{proof}

The following smoothing lemma shows that the approximating metrics become uniformly bounded immediately along the Monge-Amp\`ere flow.

\begin{lemma} \label{c2smoothing} For any $0<T< T_0$, there exists  $C>0$ such that for $t\in (0, T]$,

\begin{equation}
 tr_{\omega_0} (\omega_j) \leq e^{\frac{C}{t}}.
\end{equation}

\end{lemma}

\begin{proof} This is a parabolic Schwarz lemma similar to \cite{Y1}, \cite{LiYa}. Straightforward computation from \cite{SoT1}  shows that for any $t\in [0, T]$, there exist uniform constants $C_1$ and $C_2>0$ such that

$$ (\ddt{} - \Delta_j) \log tr_{\omega_0} (\omega_j) \leq C_1 tr_{\omega_j } (\omega_0)+ C_2.$$

Let $H=  t \log tr_{\omega_0}(\omega_j) - A  \varphi_j$. Then if $A$ is sufficiently large, there exist uniform constants $C_3$, $C_4$, ..., $C_{10}>0$ such that

\begin{eqnarray*}
(\ddt{} - \Delta_j) H &\leq& - tr_{\omega_j } (A \omega_t - C_1 t \omega_0) - A  \dot{\varphi_j} + \log tr_{\omega_0}(\omega_j) + C_3\\
&\leq&-  C_4  ~ tr_{\omega_j}(\omega_0) +  C_5 \log tr_{\omega_j}( \omega_0 ) - C_6   \log \frac{\omega_j^n}{\omega_0^n} + C_7\\
&\leq& - C_8 \left(  \frac{\omega_0^n}{\omega_j^n} \right)^{\frac{1}{n-1}}   \left( tr_{\omega_0} (\omega_j)  \right)^{\frac{1}{n-1}} -  C_9 \log t + C_{10}.
\end{eqnarray*}

Suppose $\max_{[0, T]\times X} H = H(t_0, z_0)$. Since $H(0, \cdot) = - \infty$, $t_0 >0$. Then by the maximum principle, at $(t_0, z_0)$,

$$\log tr_{\omega_0}(\omega_j) \leq \log \left( (\log\frac{1}{t})^{n-1}  \left( \frac{\omega_j^n}{\omega_0^n} \right) \right)  +C_{11}\leq   C_{12} \frac{1}{t} + C_{13}$$
and so $H$ is uniformly bounded from above at $(t_0, z_0)$.
Since $H(0, \cdot) = 0 $ and $\varphi_{0,j}$ is both uniformly bounded, $H$ is uniformly bounded for $t\in [0, T]$ and so we prove the lemma.

\qed
\end{proof}

  Let $g^{(j)}$ be the K\"ahler metric associated to $\omega_j$ and $ \nabla^{(j)}$ be the gradient operator associated to the K\"ahler form $\omega_j$. As in \cite{Y2}, set $$(\varphi_j)_ {p\bar{k}m} = \nabla^{(j)}_p \partial_{\bar{k}} \partial_m \varphi_j$$and $$S_j= (g^{(j)})^{p\bar{r}} (g^{(j)})^{s\bar{k}} (g^{(j)})^{m\bar{t}} (\varphi_j)_{p\bar{k}m} (\varphi_j)_{\bar{r}s\bar{t}}. $$

\begin{lemma} For any $T< T_0$, there exist  constants $\lambda>0$ and  $C >0$ such that for $t\in [0, T]$,

\begin{equation}
||  \varphi_j (t, \cdot) ||_{C^3(X)} \leq C e^{\frac{\lambda}{t}}.
\end{equation}

\end{lemma}

\begin{proof} Since all the second order terms are bounded by certain power of $e^{\frac{1}{t}}$. By the computation in \cite{PSS}, there exist sufficiently large $\alpha$ and $\beta >0$ such that

$$(\ddt{} - \Delta_j) e^{- \frac{2\alpha}{t}} tr_{\omega_0} (\omega_j)  \leq  - C_1  ~e^{- \frac{ \alpha}{t}}  S_j + C_2  $$
and
$$(\ddt{} - \Delta_j) e^{- \frac{2\beta}{t}}  S_j \leq  C_3 ~  e^{- \frac{\beta}{t}}  S_j + C_4. $$   By choosing sufficiently large $A$ and $\beta> \alpha$ , we have
 $$(\ddt{} - \Delta_j) ( e^{- \frac{2\beta}{t}}  S_j + A e^{- \frac{2\alpha}{t}}  tr_{\omega_0} (\omega_j) ) \leq  - C_5  ~e^{- \frac{ \alpha}{t}}  S_j - C_6$$
for sufficiently large $A>0$.

By the maximum principle and Lemma \ref{c2smoothing}, $ e^{- \frac{\beta}{t}}  S_j$ is uniformly bounded for $t\in[0, T]$.

\qed
\end{proof}

\begin{proposition} For any $0< \epsilon < T< T_0$ and $k \geq 0$, there exists   $C_{\epsilon, T, k} >0$ such that,

\begin{equation}
|| \varphi_j||_{C^k([\epsilon, T]\times X)}  \leq C_{\epsilon, T, k}.
\end{equation}

\end{proposition}

By Lemma \ref{cauchy1}, $\varphi_j$ is a Cauchy sequence in $L^{\infty}([0, T]\times X)$ and so $\varphi_j$ converges to  $\varphi \in L^\infty([0, T]\times X)$ uniformly in $L^\infty( [0, T]\times X)$. For any $0<\delta< T$,  $\varphi_j$ is uniformly bounded in $C^{\infty}( [\delta, T]\times X)$. Therefore $\varphi_j$ converges to $\varphi$ in $C^\infty ( (0, T]\times X)$. Hence $\varphi \in C^\infty( (0, T]\times X)$.

\begin{lemma}  $\varphi\in C^0([0, T]\times X)$ and

\begin{equation}
\lim_{t\rightarrow 0^+} ||\varphi(t, \cdot)- \varphi_0(\cdot) ||_{L^\infty (X)} = 0.
\end{equation}

\end{lemma}

\begin{proof}

For any $(t, z) \in [0, T]\times X$,
$$| \varphi(t, z) - \varphi_0(z) | \leq  |\varphi (t, z) - \varphi_{ j} (t,z)| + |\varphi_{ j} (t,z) - \varphi_{0,j}(z)|+  |\varphi_{0, j} (z) - \varphi_0(z)|.$$
Since $\{\varphi_{j}\}$ is a Cauchy sequence in $L^{\infty}([0,T]\times X)$, $$\lim_{j\rightarrow \infty} ||\varphi - \varphi_j||_{L^\infty([0,T]\times X)} = 0.$$
Also  $$\lim_{j\rightarrow \infty} ||\varphi_{0, j}  - \varphi_0 ||_{L^\infty (X)}=0$$ and for any $j$,
$$\lim_{t\rightarrow 0^+}  \sup_{z\in X} |\varphi_j(t, z) -\varphi_{0,j}(z)|=0.$$
For any $\epsilon >0$, there exists $J>0$ such for any $j>J$,
$$\sup_{(t,z)\in [0, T]\times X} |\varphi (t, z) - \varphi_{ j} (t,z)|< \frac{\epsilon}{3}$$ and $$ \sup_{z\in X} |\varphi_{0, j} (z) - \varphi_0(z)|< \frac{\epsilon}{3}.$$
Fix such $j$, there exists $0<\delta<T$ such that $$\sup_{t\in [0,\delta], z\in X} | \varphi_j (t, z) - \varphi_{0,j} (z)| < \frac{\epsilon}{3}.$$
Combining the above estimates together, for any $t\in [0, \delta]$ and  $z\in X$
$$| \varphi(t, z) - \varphi_0(z) | < \epsilon.$$
This completes the proof of the lemma.

\qed
\end{proof}

Now we are ready to show the existence and uniqueness for the Monge-Amp\`ere flow starting with $\varphi_0\in PSH_p(X, \omega_0, \Omega)$.
\begin{proposition}\label{propexistence}  $\varphi$  is the unique solution  of the following Monge-Amp\`ere equation
\begin{equation}\label{maflow_deglim}
\left\{
\begin{array}{rcl}
&&{ \displaystyle \ddt{\varphi} = \log
 \frac{ ( \omega_t + \ddbar \varphi ) ^n } {\Omega} },  ~~~~~(0, T_0)\times X\\
&&\\
&& \varphi(0,\cdot)=\varphi_{0}
\end{array} \right.
\end{equation}
in the space of $C^0([0, T_0)\times X) \cap C^{\infty}((0, T_0)\times X) $.

\end{proposition}

\begin{proof} It suffices to prove the uniqueness. Suppose there exists another solution $\varphi'\in C^0([0, T_0)\times X) \times C^{\infty}((0, T_0)\times X) $ of the Monge-Amp\`ere flow (\ref{maflow_deglim}).

Let $\psi = \varphi'- \varphi$. Then
\begin{equation}
\left\{
\begin{array}{rcl}
&&{ \displaystyle \ddt{\psi} = \log
 \frac{ ( \omega_t + \ddbar \varphi + \ddbar \psi ) ^n } { ( \omega_t + \ddbar \varphi)^n} },  ~~~~~(0, T_0)\times X\\
&&\\
&& \psi(0,\cdot)=0.
\end{array} \right.
\end{equation}

By the maximum principle,  $\max_X \psi(t, \cdot)$ is decreasing and $\min_X \psi (t, \cdot)$ is increasing on $(0, T_0)$. Since both of $\max_X  \psi(t, \cdot)$ and $\min_X \psi(t, \cdot)$ are continuous on $[0, T_0)$ with $\max_X  \psi(0, \cdot)=\min_X  \psi( 0, \cdot)=0$ , $$\psi (t, \cdot) =0$$ for $t \in [0, T_0)$.
The proposition follows easily.

\qed
\end{proof}

With the above preparations, we can show the smoothing property for the K\"ahler-Ricci flow with rough initial data.

\begin{theorem} \label{smoothing1} Let $X$ be an $n$-dimensional K\"ahler manifold. Let $\omega_0$ be a K\"ahler form and $\Omega$ be a smooth volume form on $X$.   Suppose that $\omega_0'= \omega_0 + \ddbar\varphi_0$ ~for some $\varphi_0 \in PSH_p(\omega_0, \Omega)$ for some $p>1$.  Then  there exists a unique family of smooth K\"ahler metrics $\omega(t, \cdot) \in C^\infty ((0,T_0) \times X)$ satisfying the following conditions.

\begin{enumerate}

\item $\displaystyle{
 \ddt{\omega} = -Ric(\omega),  ~~~~~(0, T_0)\times X}.$

\item There exists $\varphi\in C^0([0, T_0)\times X)\cap C^\infty((0, T_0) \times X)$ such that  $\omega = \omega_0 + \ddbar \varphi$ and $$\lim_{t\rightarrow 0^+} || \varphi(t, \cdot)- \varphi_0(\cdot)||_{L^{\infty}(X)} = 0. $$ In particular, $\omega(t, \cdot)$ converges in the sense of distribution to $\omega_0'$ as $t\rightarrow 0$.

\end{enumerate}

\end{theorem}

\begin{proof} The unnormalized K\"ahler-Ricci flow $\ddt{\omega} = -Ric(\omega)$ is equivalent to the Monge-Amp\`ere flow \begin{equation}\label{equiv1} \ddbar ( \ddt{\varphi} -  \log \frac{ (\omega_t + \ddbar\varphi)^n}{\Omega})=0.\end{equation} Then $ \ddt{\varphi} =  \log \frac{ (\omega_t + \ddbar\varphi)^n} {\Omega}+f(t)$ with $\lim_{t\rightarrow 0^+} \varphi(t, 0) = \varphi_0$ for a smooth function $f(t)$ on $(0, T_0)$. Proposition \ref{propexistence} gives the existence of such a $\varphi$ with $f(t)=0$.

Suppose there is another solution $\phi \in C^\infty((0, T_0)\times X)\cap C^0([0, T_0)\times X)$ to the equation (\ref{equiv1}). Then  $$\ddt{\phi} = \log \frac{(\omega_t + \ddbar \phi)^n}{\Omega} + f(t)$$ for some smooth function $f(t)$ on $(0, T_0)$. We can assume that $\phi(0, \cdot)=\varphi_0$ by subtracting a constant because $\phi(t, \cdot)$ converges to a continuous $\phi_0(\cdot)$ in $C^0(X)$ as $t\rightarrow 0$, and $\phi_0$ differs from $\varphi_0$ by a constant. Then we consider the function $\psi = \phi - \varphi$,   $$\ddt{\psi } = \log \frac{ (\omega_t+\ddbar\varphi + \ddbar \psi )^n}{(\omega_t + \ddbar \varphi)^n} + f(t)$$ with $\psi(0, \cdot)=0$. By the same argument as that in the proof of Proposition \ref{propexistence} we can show that  for $0<t_1< t_2<T_0$,
\begin{eqnarray*}
&&\max_X\psi(t_2, \cdot) \leq \max_X \psi(t_1, \cdot) + \int_{t_1}^{t_2} f(t) dt\\
&&\min_X\psi(t_2, \cdot) \geq \min_X \psi(t_1, \cdot) + \int_{t_1}^{t_2} f(t) dt\
\end{eqnarray*}
Therefore,
$$ \min_X \psi(t_2, \cdot) \geq  \max_X \psi (t_2, \cdot) - (\max_X \psi(t_1, \cdot) - \min_X \psi(t_1, \cdot)).$$
and by letting $t_1 \rightarrow 0^+$, we have $$ \min_X \psi(t_2, \cdot) \geq  \max_X \psi (t_2, \cdot).$$
So $\psi(t, \cdot)= \psi(t) = \int_0^t f(s) ds$.

\qed
\end{proof}

Theorem \ref{smoothing1} shows that the K\"ahler-Ricci flow smooths out the initial semi-positive closed $(1,1)$-current $\omega$ with bounded local potential and $\omega^n \in L^p(X)$ for some $p>1$. It improves a result of Chen-Tian-Zhang \cite{CTZ} (also see \cite{CT} and \cite{CD}), where $p>3$. We remark that the condition that $p>1$ is essential for later estimates and geometric applications.



\subsection{Monge-Amp\`ere flows with degenerate  initial data} \label{3.2}

In this section, we will investigate a family of Monge-Amp\`ere flows with singular data on a smooth projective variety. The existence and uniqueness for the solutions will be proved.

We start with two conditions prescribing the singularity and degeneracy of the data that will be considered along certain Monge-Amp\`ere flows. These conditions arise naturally in the geometric setting in later discussions.

\medskip

\noindent {\bf Condition A.}  Let $X$ be an $n$-dimensional projective  manifold.
 Let $L_1\rightarrow X$ be a big and semi-ample line bundle over $X$ and  $L_2\rightarrow X$ be a line bundle  such that $[ L_1 +  \epsilon L_2]$ is still semi-ample for $\epsilon>0$ sufficiently small. Let $\omega_0 \in c_1(L_1)$ be a smooth semi-positive closed $(1,1)$-form on $X$ and $\chi \in c_1(L_2)$  a smooth closed $(1,1)$-form. Let $\omega_t = \omega_0 + t \chi$. We assume that $\omega_0$ at worst vanishes along a projective subvariety of $X$ to a finite order, that is, there exists an effective divisor $E_0$ on $X$ such that for any fixed K\"ahler metric $\vartheta$, $$\omega_0 \geq C_{\vartheta} |S_{E_0}|^2_{h_{E_0}} \vartheta,$$
 where $C_{\vartheta}>0$ is a constant, $S_{E_0}$ is a defining section of $E_0$ and $h_{E_0}$ is a smooth hermitian metric on the line bundle associated to $E_0$.

\bigskip

Such an $\omega_0$ always exists. For example, let $m$ be sufficiently large such that $(L_1)^m$ is globally generated and let $\{ S_j^{(m)} \}_{j=0}^{d_m} $ be a basis of $H^0(X, (L_1)^m)$. We can then let $$\omega_0 = \frac{1}{m} \ddbar \log  \sum_{j=0}^{d_m} |S^{(m)}_j|^2.$$

\medskip
\noindent {\bf Condition B.} Let $\Theta$ be a smooth volume form on $X$. Let $ E = \sum_{i=1}^p a_i E_i$ and $F= \sum_{j=1}^q b_j F_j$ be  effective divisors on $X$, where $E_i$ and $F_j$ are irreducible components of $E$ with simple normal crossings. In addition,  we assume $a_i \geq 0$ and $0 < b_j <1$.  Let $\Omega$ be a semi-positive $(n,n)$-form on $X$ such that $\int_X \Omega >0$   and

 \begin{equation}
 \Omega =  |S_{E}|^2_{h_{E}} |S_F|^{-2}_{h_F} \Theta,
 \end{equation}
 where $S_{E}$ and $S_F$ are the multi-valued holomorphic defining sections of $E$ and $F$, $h_{E}$ and $h_F$ are smooth hermitian metrics on the line bundles associated to $E$ and $F$.

\bigskip

Note that the condition $b_j\in (0, 1)$ makes $\Omega$ an integrable $(n,n)$-form on $X$. Furthermore, $\displaystyle{\frac{\Omega}{\Theta}  } $ is in $L^p(X, \Theta)$ for some $p>1$.

Since $L_1$ is big and semi-ample, by Kodaira's lamma, there exists an effective $\mathbf{Q}$-divisor $\tilde{E}$ such that  $[L_1] -\epsilon [\tilde{E}] $ is ample for any sufficiently small rational $\epsilon>0$. Without loss of generality, we can always assume that the support of $\tilde{E}$ contains $E_0$, $E$ and $F$, i.e.,
$$ supp E \cup supp F \cup supp E_0 \subset supp \tilde{E} . $$
Let $S_{\tilde{E}}$ be the defining section of $\tilde{E}$ and $h_{\tilde{E}}$ a smooth hermitian metric on the line bundle associated to $[\tilde{E}]$ such that for sufficiently small $\epsilon>0$,

$$\omega_0 -  \epsilon Ric(h_{\tilde{E}}) > 0.$$
We can also scale $h_{\tilde{E}}$ and assume $ |S_{\tilde{E}}|_{h_{\tilde{E}}}^2 \leq 1$ on $X$.

Let $\omega_t = \omega_0 + t \chi$. We consider the following Monge-Amp\`ere flow with the initial data $\varphi_0 \in PSH(X, \omega_0)\cap C^\infty(X)$.

\begin{equation}\label{degmaflow1}
\left\{
\begin{array}{l}
{ \displaystyle \ddt{\varphi} = \log
 \frac{ ( \omega_0 + t\chi + \ddbar \varphi ) ^n } {\Omega} },\\
\\
 \varphi(0,\cdot)=\varphi_0.
 \end{array} \right.
\end{equation}

The equation (\ref{degmaflow1}) is not only degenerate in the sense that $[\omega_t]$ is not neccesarily K\"ahler but also that $\Omega$ has zeros and poles along $E$ and $F$. The goal of the following discussion is to prove the existence and uniqueness of the solution for the Monge-Amp\`ere flow (\ref{degmaflow1}) with appropriate assumptions.
Let $$T_0 = \sup \{ t\geq 0 ~|~ [L_1 + t L_2] ~is~semi{\textnormal -}ample~ \}. $$ If {\bf Condition A} is satisfied,  $T_0>0$ or $T_0= \infty$.   Furthermore, $L_1+ tL_2$ is big  for any $t\in[0, T_0)$.

The following theorem is the main result of this section.

\begin{theorem} \label{firstthm} Let $X$ be an $n$-dimensional projective manifold. Suppose {\bf Condition A} and {\bf Condition B} are satisfied. Then for any $\varphi_0 \in PSH(X, \omega_0)\cap C^\infty(X)$, there exists a unique $\varphi\in    C^{\infty}([0, T_0))\times (X\setminus \tilde{E}) )$ with $\varphi(t, \cdot) \in  PSH(X, \omega_{t})\cap L^{\infty}( X )$ for each $t\in [0, T_0)$, satisfying the following Monge-Amp\`ere flow. %
\begin{equation}\label{degmaflowlim}
\left\{
\begin{array}{cl}
&{ \displaystyle \ddt{\varphi} = \log
 \frac{ ( \omega_t  + \ddbar \varphi ) ^n } {  \Omega} },   ~~~~ {\textnormal on ~[0, T_0) \times X \setminus \tilde{E} }, \\
&\\
& \varphi(0,\cdot)=\varphi_0 ,~~~~~~~~~~~~~~~~~~~~~~~~~~~~~~~{\textnormal on }~X  .
\end{array} \right.
\end{equation}

\end{theorem}

\medskip

\begin{remark} For any fixed $T\in (0, T_0)$, we can assume that $\omega_t\geq \epsilon \omega_0$ for all $t\in [0, T]$, where $\epsilon>0$ is sufficiently small and it depends on $T$.   Recall that  $L_1 + t L_2$ is semi-ample and big for all $t\in [0, T_0)$.  We fix $T'\in (T, T_0)$. Then
$$ [\omega_t] = \frac{t}{T_0} [ \omega_0 + T_0 \chi] + \frac{ T_0-t }{T_0} [\omega_0]$$
and $[ \omega_0 + T_0 \chi ] $ is semi-positive and big. Then there exists  $\phi \in PSH( X, \omega_0 + T_0 \chi) \cap C^\infty(X)$, i.e., $ \omega_0 + T_0  \chi + \ddbar\phi \geq 0$. Then for all $t\in [0, T']$,
$$\omega_t + \frac{t}{T_0} \ddbar\phi = \frac{t}{T_0}( \omega_0 + T_0 \chi) + \frac{t}{T_0}\ddbar \phi +  \frac{ T_0-t }{T_0}  \omega_0 \geq \frac{ T_0-t }{T_0}  \omega_0 \geq 0 .$$
Then $$ \ddt{} ( \varphi - \frac{t}{T_0}\phi) = \log \frac{ ( \omega_0 + t (\chi + \frac{1}{T_0}\ddbar \phi) + \ddbar(\varphi - \frac{t}{T_0} \phi) )^n } {\Omega} - \frac{1}{T_0} \phi.$$
Let $\omega_0 '= \omega_0$, $ \chi' = \chi + \frac{1}{T_0} \ddbar \phi$, $\omega_t'= \omega_0'  + t \chi'$, $ \Omega' = \Omega e^{ \frac{\phi} {T_0} }  $ and $\varphi' = \varphi - \frac{t}{T_0} \phi$. We have
$$ \ddt{} \varphi' = \log \frac{ ( \omega_t' + \ddbar \varphi' )^n } {\Omega'}.$$
It is easy to check that $\omega'_t \geq \frac{T_0-t}{T_0} \omega_0$ and {\bf Condition A} and {\bf Condition B} are still satisfied for $\omega_0'$ and $\Omega'$.
\end{remark}

\bigskip

From now on, we fix $T\in (0, T_0)$ and assume without loss of generality from the previous remark that $ \omega_t $ is bounded from below by $\epsilon \omega_0$ for sufficiently small $\epsilon>0$. In order to prove Theorem \ref{firstthm}, we have to perturb equation (\ref{degmaflowlim}) in order to obtain smooth approximating solutions.  Let $$\omega_{t,s} = \omega_0 + t \chi + s \vartheta$$ and

$$\Omega_{w, r} = \frac  {r+ |S_{E}|^2_{h_{E}} } { w+ |S_F|^{2}_{h_F}} ~ \Theta$$ be the perturbations of $\omega_t$ and $\Omega$
for $s, w, r \in[0, 1]$. In particular, $\omega_{t,0} =\omega_t$ and $\Omega_{0, 0} = \Omega$. Since $\vartheta$ is K\"ahler, $\omega_{t,s}$ is K\"ahler for $s>0$ and $t\in [0, T_0)$.

Then we consider the following well-defined family of Monge-Amp\`ere flows with the fixed initial data $\varphi_0$.

\begin{equation}\label{degmaflow2}
\left\{
\begin{array}{cl}
&{ \displaystyle \ddt{\varphi_{s,w,r} } = \log
 \frac{ ( \omega_{s,t} + \ddbar \varphi_{s,w, r}  ) ^n } {\Omega_{w,r}} },\\
&\\
& \varphi_{s,w, r} (0,\cdot)=\varphi_0 .
\end{array} \right.
\end{equation}
The standard argument gives the following lemma as $[\omega_{s,t}]$ stays positive on $[0, T_0)$.

\begin{lemma} For any $s, w ,r \in (0, 1]$, there is a unique smooth solution $\varphi_{ s, w, r}$ of the Monge-Amp\`ere flow (\ref{degmaflow2}) on $[0, T_0) \times X$.
\end{lemma}

 \begin{lemma}\label{gugong} Let $\displaystyle{F_{w, r} = \frac{ \Omega_{w, r}} {\vartheta^n}}.$ Then there exist constants $p>1$ and $C>0$ such that for $w, r\in [0,1]$ and
\begin{equation}
||F_{w, r} ||_{L^p(X, \vartheta^n)} \leq C.
\end{equation}

\end{lemma}

\begin{proof}
There exists a constant $C>0$ such that
\begin{eqnarray*}
\int_X (F_{w,r})^p \vartheta^n &=& \int_X (F_{w, r})^{p-1} \Omega_{w, r}\\
&=& \int_X (F_{w,r})^{p-1}(r+ |S_{E}|)^2_{h_{E}} (w+ |S_F|)^{-2}_{h_F} \Theta\\
&\leq&
C \int_X (F_{w,r})^{p-1}  |S_F|^{-2}_{h_F} \Theta.
\end{eqnarray*}

Since $F_{w, r}$ has at worst poles along $\tilde{E}$ and the vanishing order of $|S_F|^2_{h_F}$ is strictly less than $2$, by choosing $p-1>0$ sufficiently small, $\int_X (F_{w,r})^p \vartheta^n $ is uniformly bounded from above.

\qed
\end{proof}

We can apply the results for degenerate complex Monge-Amp\`ere equations as $F_{w,r}$ is uniformly bounded in $L^p(X)$ for some $p>1$.

 \begin{lemma}\label{c0estimate1} For any $ 0<T<T_0$, there exists $C>0 $ such that  for all $ s, w, r \in (0, 1] $,
$$ || \varphi_{s, r, w } ||_{L^\infty( [0, T] \times X )} \leq C $$

 \end{lemma}

 \begin{proof} We first prove the uniform upper bound for $\varphi_{s,w,r}$. We define for $t\in [0, T]$

\begin{equation}
   \alpha_{s,w, r} (t)  = \frac{ \int_X \Omega_{w,r}} {[\omega_{t,s}]^n}.
\end{equation}
It is easy to see that $\alpha_{s,w,r}(t)$ is uniformly bounded for $ t\in [0, T]$.
 Notice that
\begin{eqnarray*}
\ddt{} \int_X \varphi_{s,w, r} \Omega_{w,r}  &=&  \int_X \log \frac{ (\omega_{t,s} + \ddbar \varphi_{s,w,r})^n }{  \Omega_{w,r}}  \Omega_{w,r}  \\
&\leq &  (\int_X \Omega_{w,r})  \left( \log \frac{ \int_X (\omega_{t,s} + \ddbar \varphi_{s,w,r})^n}{ \alpha_{s,w,r}(t)\int_X \Omega_{w,r}  } \right)   + \alpha_{s,w,r}(t) \int_X \Omega_{w,r} \\
&\leq& ( \int_X \Omega_{w,r} )  \left( \log \frac{ [\omega_{t,s}]^n}{ \alpha_{s,r}(t)\int_X \Omega_{w,r} } \right) + \alpha_{s,w,r}(t) \int_X \Omega_{w,r} \\
&=&\alpha_{s,w,r}(t) \int_X \Omega_{w,r}.
\end{eqnarray*}
By the maximum prinicple,  $\int_X  \varphi_{s,w,r} \Omega_{w,r} \leq C$ for a uniform constant $C$ that depends on $T$.

On the other hand, since $ \varphi_{s,w,r} \in PSH ( X, \omega_{s,w,r} )$, by H\"omander-Tian's estimate,  there exist $\alpha>0$ and $C_{\alpha} >0 $ such that for all $s, r\in (0,1]$ and $t\in [0, T]$,

$$ \int_X e^{- \alpha   (\varphi_{s,w,r} - \sup_X \varphi_{s,w,r})} \Omega_{w,r} \leq C_\alpha.$$

By Jensen's inequality, there exists a constant $C>0$ such that  for  $s, r\in (0,1]$ and $t\in [0, T]$,
$$ \sup_X \varphi_{s,w,r} - \int_{\tilde{X}} \varphi_{s,w,r} \Omega_{w,r} \leq C .$$
It follows that  $\sup_X \varphi_{s,w,r}$ is uniformly bounded above.

 Now it suffices to obtain a uniform lower bound for $\varphi_{s, w, r}$. Let $\theta = \delta \omega_0$, where $\delta>0$ is sufficiently small such that $$ 2 \theta \leq \omega_t$$ for all $t\in [0, T]$.  By the choice of $h_{\tilde{E}}$,  $$\theta - \epsilon Ric(h_{\tilde{E}})>0$$ for any sufficiently small $\epsilon>0$.

  We consider the following family of Monge-Amp\`ere equations for $w, r\in [0,1]$.

\begin{equation}\label{app1}
 ( \theta + \ddbar \phi_{w,r} ) ^n =  C_{w,r} \Omega_{w,r},
\end{equation}
with the normalization conditions $[\theta]^n = C_{w,r} \int_X \Omega_{w,r}$ and $ \sup_X \phi  = 0.$

Note that by Lemma \ref{gugong}, $\displaystyle{\frac{C_{w,r} \Omega_{w,r}}{\theta^n} }$ is uniformly bounded in $L^p(X, \theta^n)$ for $w, r\in [0,1]$.   By the resluts in \cite{EGZ1}, $\phi_{w,r}  \in C^{\infty} ( X \setminus \tilde{E}) $ and there exists a uniform constant $C>0$ such that for all $w, r\in [0, 1]$,

\begin{equation}
||\phi_{w,r}  ||_{L^\infty(X) }\leq C.
\end{equation}

Let $\psi_{s,w,r} (t, \cdot) = \varphi_{s,w,r} (t, \cdot)-   \phi_{w, r} $. The evolution equation for $\psi_{s,w.r}$ is given by the following formula for $t\in [0,T]$.

 $$ \ddt{} \psi_{s,w,r} = \log \frac{ ( \omega_{t,s} +  \ddbar \varphi_{s, r} )^n} { ( \theta + \ddbar \phi_{w,r} )^n} + \log C_{w,r} .$$

 Let $H= \psi_{s,w,r} - \epsilon \log |S_{\tilde{E}} |^2_{h_{\tilde{E}}} .$ Then the minimum of $H$ is always achieved in $X\setminus \tilde{E}$ as $H$ tends to $\infty$ near $\tilde{E}$. Suppose for fixed $t\in [0, T]$, $\min_X H(t, \cdot)= H(t, z_0)$. By choosing sufficiently small $\epsilon>0$, we have at $(t, z_0)$,
 \begin{eqnarray*}
 \ddt{}   H
&=&  \log \frac{ ( \theta + \ddbar \phi_{w,r} + ( \omega_t - \theta  + s \vartheta -\epsilon Ric(h_{\tilde{E}}) +  \ddbar H )^n} { ( \theta + \ddbar \phi_{w,r})^n} + \log C_{w,r}\\
&\geq &  \log \frac{ ( \theta + \ddbar \phi_{w,r} + (  \theta   -\epsilon Ric(h_{\tilde{E}}))  +  \ddbar H )^n} { ( \theta + \ddbar \phi_{w, r})^n} + \log C_{w,r}\\
 & \geq &  \log \frac{ ( \theta + \ddbar \phi_{w,r} +  \ddbar H )^n} { ( \theta + \ddbar \phi_{w, r} )^n} + \log C_{w,r}  \\
 &\geq &  \log C_{w,r} .\\
 \end{eqnarray*}
 By the maximum principle,  there exist a sufficiently small $\epsilon>0$ and  a uniform constant $C>0$ independent of $\epsilon$ such that for all $s, w, r \in (0, 1]$, $t\in [0, T]$, $$ H \geq -C.$$
It follows then that  for $s, w, r \in (0,1]$ and $t\in [0, T]$,  $$\varphi_{s,w, r} (t, \cdot) \geq   \phi_{w,r}(\cdot ) +  \epsilon \log |S_{\tilde{E}}|^2_{h_{\tilde{E}}}  - C.$$

 By letting $\epsilon \rightarrow 0$, we have

 $$\varphi_{s, w, r} (t, \cdot) \geq   \phi_{w,r}(\cdot ) - C .$$
The uniform bound for $\phi_{w,r}$ gives the uniform lower bound for $\varphi_{s,w,r}$. Combined with the upper bound for $\varphi_{s,w,r}$, we have completed the proof.

 \qed
 \end{proof}

 \begin{lemma}\label{timebound1} For any $T\in (0, T_0)$,
there exist   $C$, $\alpha>0$  such that for all $ t \in [0, T] $ and $ s,w, r \in (0, 1] $

 $$ \left|\ddt{}\varphi_{s,w,r}\right| \leq  C + \log  |S_{\tilde{E}}|_{h_{\tilde{E}}}^{-2\alpha}  .$$

\end{lemma}

\begin{proof}  Let $\Delta_{s,w,r}$ be the Laplace operator with respect to the K\"ahler metric $\omega_{s,w,r}$. Notice that $$ \ddt{} \dot{\varphi}_{s,w, r} = \Delta_{s,w, r} \dot{\varphi}_{s, w,r} + tr_{\omega_{s,w,r}}(\chi).$$

Let $H^+ =  \left( \dot{\varphi}_{s,w,r}  - A^2 \varphi_{s,w,r} + A \log|S_{\tilde{E}}|_{h_{\tilde{E}}}^2\right).$ $H^+(0, \cdot)$ is uniformly bounded from above for $A>0$ sufficiently large. Then for $A$ sufficiently large, we have
\begin{eqnarray*}
&&\ddt{} H^+  \\
&=& \Delta_{s,w,r} H^+
 - tr_{ \omega_{s,w,r} }( A^2 \omega_{t,s} - \chi - A ~Ric(h_{\tilde{E}})) - A^2 \dot{ \varphi}_{s,w,r} + nA^2 \\
&\leq& \Delta_{s,w,r} H^+
- A^2 \dot{\varphi}_{s,w,r}  + nA^2 \\
&=& \Delta_{s,w,r} H ^+
- A^2 H^+  + A^2 (  - A^2 \varphi_{s,w,r} + A \log|S_{\tilde{E}}|_{h_{\tilde{E}}}^2) + nA^2\\
&\leq&  \Delta_{s,w,r} H ^+
- A^2 H^+ + C.
\end{eqnarray*}
By the maximum principle, $H^+$ is uniformly bounded above and so there exist $C_1$ and $C_2$ such that $$\dot{\varphi}_{s,w,r} \leq C_1 + C_2 \log |S_{\tilde{E}}|^2_{h_{\tilde{E}}}.$$

To estimate the lower bound of $\dot{\varphi}_{s,w,r}$, we define $$H^- = \dot{\varphi}_{s,w,r} + A^2 \varphi_{s,w,r} - A \log |S_{\tilde{E}}|^2_{h_{\tilde{E}}}$$ for sufficiently large $A$. Then straightforward calculation shows that there exist constants $C_3$,  $C_4$, ..., $C_7$  such that
\begin{eqnarray*}
 \ddt{} H^- &\geq& \Delta_{s,w,r} H^- +  C_3 (\frac{\omega_{t,s}^n}{\omega_{s,w,r}^n})^{\frac{1}{n}}  + A^2 \dot{\varphi}_{s,w,r} - C_4 \\
 &=& \Delta_{s,w,r} H^- +  C_3 (\frac{\omega_{t,s}^n}{\omega_{s,w,r}^n})^{\frac{1}{n}}  + A^2 \log \frac{\omega_{s,w,r}^n}{\omega_{t,s}^n} + A^2\log \frac{\omega_{t,s}^n}{\Omega_{w,r}} - C_4\\
&\geq& \Delta_{s,w,r} H^-  -  A^2 \log \frac{\omega_{s,w,r}^n}{\omega_{t,s}^n} + A^2\log \frac{\omega_{t,s}^n}{\Omega_{w,r}} - C_5\\
 &=& \Delta_{s,w,r} H^-  -  A^2 H^- - A^3 \log |S_{\tilde{E}}|^2_{h_{\tilde{E}}} + 2 A^2\log \frac{\omega_{t,s}^n}{\Omega_{w,r}} - C_6\\
 &\geq& \Delta_{s,w,r} H^-  -  A^2 H^- - C_7.
 \end{eqnarray*}
Then a similar argument by the maximum principle gives the lower bound for $H^-$ and $\dot{\varphi}_{s,w,r}$.

\qed
\end{proof}

\begin{lemma} \label{c2estimate1}

For any $T\in (0, T_0)$, there exist  $C$, $\alpha>0$  such that for all $ t \in [0, T] $ and $ s, r, w \in (0, 1] $
 $$| tr_{\vartheta} (\omega_{s,w,r} )| \leq  C  |S_{\tilde{E}}|_{h_{\tilde{E}}}^{-2\alpha}  .$$

\end{lemma}

\begin{proof} Standard calculations show that for some constant $C>0$,

$$(\ddt{} - \Delta_{s,w,r}) \log tr_{\vartheta} (\omega_{s,w,r} ) \leq C tr_{\omega_{s,w,r} } (\vartheta)  + \frac{ tr_{\omega_{s,w,r} } (Ric( \Omega_{w,r})) } { tr_{\omega_{s,w,r}} (\vartheta)} + C.$$
Define  $$H = \log tr_{\vartheta}(\omega_{s,w,r}) - A^2\varphi_{s,w,r} + A \log |S_{\tilde{E}}|_{h_{\tilde{E}}}^2.$$
Then for sufficiently large $A>0$, there exist uniform constants $C_1$ and $C_2$ such that
\begin{eqnarray*}
&&(\ddt{} -\Delta_{s,w,r}) H \\
&\leq& -  tr_{\omega_{s,w,r} } (A^2 \omega_{t,s} - A Ric(h_{\tilde{E}}) - C\vartheta)  + \frac{ tr_{\omega_{s,w,r} } (Ric( \Omega_{w,r})) } { tr_{\omega_{s,w,r}} (\vartheta)} -  A^2 \log \frac{\omega_{s,w,r}^n } {\Omega_{w,r} }+ C_1 \\
&\leq& - A~tr_{\omega_{s,w,r}}(\theta)  + \frac{ tr_{\omega_{s,w,r} } (Ric( \Omega_{w,r})) } { tr_{\omega_{s,w,r}} (\vartheta)} -  A^2 \log \frac{\omega_{s,w,r}^n } {\Omega_{w,r} }  +  C_2.
\end{eqnarray*}

Suppose $\max_{[0, t]\times X} H = H(t_0, z_0)$. Then $z_0 \in X\setminus \tilde{E}$ and at $(t_0, z_0)$,
there exist $\alpha_1$, $C_3$ and $C_4$ such that %
$$ tr_{\omega_{s,w,r}}(\theta) \leq A^{-1} \frac{ tr_{\omega_{s,w,r} } (
A~ Ric( \Omega_{w,r})) } { tr_{\omega_{s,w,r}} (\vartheta)} -  A \log \frac{\omega_{s,w,r}^n } {\Omega_{w,r} } + C_3\leq C_4 ~|S_{\tilde{E}}|^{-2\alpha_1}_{h_{\tilde{E}}}.$$
Applying the mean value inequality and Lemma \ref{timebound1},  there exist $\alpha_2$ and $C_5$

$$ tr_{\vartheta} (\omega_{s,w,r}) \leq C_5  |S_{\tilde{E}}|^{-2\alpha_2}_{h_{\tilde{E}}} .$$
Therefore $H(t_0, z_0)$ is uniformly bounded from above. The lemma then follows easily.

\qed
\end{proof}

The following proposition gives a uniform bound for the approximating K\"ahler metrics $\omega_{s,w,r}$ away from $\tilde{E}$.

\begin{proposition}\label{highestimate1} For any $T\in (0, T_0)$, $K\subset\subset X \setminus {\tilde{E}}$ and $k>0$, there exists $C_{k, K, T}$ such that

$$|| \varphi_{s,w,r}||_{C^k([0, T] \times K)} \leq C_{k, K, T} .$$

\end{proposition}

\begin{proof} The proof follows from standard Schauder's estimates.

\qed\end{proof}


Our goal is to construct a solution by the approximating solutions $\varphi_{s,w,r}$.

\begin{lemma}The following monotonicity conditions hold for $\varphi_{s, w, r}$.

 \begin{enumerate}

\item For any $0< r_1 \leq r_2 \leq 1$ and $s, w\in (0,1]$ , $$\varphi_{s, w, r_1} \geq \varphi_{s, w, r_2}.$$

\item For any $0< w_1 \leq w_2 \leq 1$ and $s, r\in (0, 1]$, $$ \varphi_{s, w_1, r} \leq \varphi_{s, w_2. r}.$$

\item For any $0< s_1 \leq s_2 \leq 1$ and $w, r\in (0, 1]$, $$ \varphi_{s_1, w, r} \leq \varphi_{s_2, w, r}.$$

\end{enumerate}

\end{lemma}

\begin{proof} The proof is a straightforward application of the maximum principle.

\qed
\end{proof}

Fix $T\in (0, T_0)$, for each $t\in [0, T]$, let $$\varphi_{s, w } (t, \cdot)= (\limsup_{r\rightarrow 0}  \varphi_{s,w, r}(t, \cdot))^*$$
where $f^*(z) = \lim_{\delta\rightarrow 0 } \sup_{B_\delta(z) } f(\cdot).$ Then $\varphi_{s,w}\in PSH(X, \omega_{t,s})\cap L^{\infty}(X)\cap C^{\infty}(X \setminus {\tilde{E}})$ and $\varphi_{s, w, r}$ converges to $\varphi_{s, w}$ on $X\setminus {\tilde{E}}$ locally in $C^\infty$-topology by estimates from Lemma \ref{c0estimate1} and Proposition \ref{highestimate1}.  The following monotonicity also holds and follows easily from the above results.


\begin{lemma}   For any $0< s_1 \leq s_2 \leq 1$ and $w\in (0, 1]$, $$ \varphi_{s_1, w} \leq \varphi_{s_2, w}.$$ Also  for any $0< w_1 \leq w_2 \leq 1$ and $s\in (0, 1]$, $$ \varphi_{s, w_1} \leq \varphi_{s, w_2}.$$ Furthermore, for any $T\in (0, T_0)$, $K\subset\subset X\setminus {\tilde{E}}$ and $k>0$, there exists $C_{K,k,T}>0$ such that on $[0, T]$, 
$$||\varphi_{s, w} ||_{C^k(K)} \leq C_{ K, k, T}.$$

\end{lemma}


Let $$\varphi_{s} (t, \cdot)= \lim_{w \rightarrow 0}  \varphi_{s, w}(t, \cdot) .$$ Then $\varphi_{s}\in PSH(X, \omega_{t,s})\cap L^{\infty}(X)\cap C^{\infty}(X \setminus {\tilde{E}})$ and $\varphi_{s, w}$ converges to $\varphi_{s}$ on $X\setminus {\tilde{E}}$ locally in $C^\infty$-topology.



\begin{lemma}   For any $0< s_1 \leq s_2 \leq 1$, $$ \varphi_{s_1} \leq \varphi_{s_2}.$$ Furthermore, for any $T\in (0, T_0)$, $K\subset\subset X\setminus {\tilde{E}}$ and $k>0$, there exists $C_{K,k,T}>0$ such that  such that on $[0, T]$,
$$||\varphi_{s} ||_{C^k(K)} \leq C_{K. k, T}.$$

\end{lemma}


Let $\varphi = \lim_{s \rightarrow 0} \varphi_{s}$. Since $ \varphi_{s}$ is decreasing as $s \rightarrow  0$ is and $\varphi_{s}$ is bounded below uniformly, we have

$$\varphi \in PSH(X, \omega_{t,0}) \cap L^{\infty}( X )\cap C^{\infty}(X\setminus {\tilde{E}}).$$
Furthermore, for any $K\subset\subset X\setminus {\tilde{E}}$,  $$\varphi_{s} \rightarrow \varphi$$ in  $C^{\infty}( [0, T]\times K).$       The following corollary is then immediate.
\begin{corollary}  \label{corexist}
$\varphi$ satisfyies the following Monge-Amp\`ere flow

\begin{equation}\label{krflowlim}
\left\{
\begin{array}{l}
{ \displaystyle \ddt{\varphi} = \log
 \frac{ ( \omega_t  + \ddbar \varphi ) ^n } {  \Omega} },   ~~~~ {\rm on}~[0, T] \times( X\setminus \tilde{E} ) \\
\\
 \varphi(0,\cdot)=\varphi_0 ,~~~~~~~~~~~~~~~~~~~~~~~~~~~~~~~{\rm on }~X  .
\end{array} \right.
\end{equation}

\end{corollary}

In order to prove the uniqueness of the solution of the Monge-Amp\`ere flow (\ref{krflowlim}),   we consider a family of new Monge-Amp\`ere flows with one more parameter.

Let $\omega_{t, s}^{ (\delta)} = (1-\delta) \omega_0 + t\chi + s \vartheta = \omega_{t,s} -\delta \omega_0$. Then for fixed $T\in [0, T_0)$, there exists $\delta_0>0$, such that $(1-\delta) \omega_0 + t\chi  \geq 0 $ for all $t\in [0, T]$ and $\delta\in [-\delta_0, \delta_0]$.

 The following family of Monge-Amp\`ere flows admit smooth solutions in $C^{\infty}([0, T]\times X)$
 \begin{equation}\label{flowuniq}
\left\{
\begin{array}{l}
 { \displaystyle \ddt{}\varphi_{s,w, r}^{(\delta)} = \log \frac{( \omega_{t,s}^{(\delta)}  + \ddbar \varphi_{s,w, r}^{(\delta)} )^n }{ \Omega_{w,r}} }\\
\\
 \varphi_{s,w, r}^{(\delta)} (0, \cdot)= (1-\delta) \varphi_0,
\end{array} \right.
\end{equation}

Obviously for $s,w, r \in (0, 1]$, $\varphi_{s,w, r}^{(\delta)} $ is a smooth family in $\delta$ and   $$\lim_{\delta\rightarrow 0 } \varphi_{s,w, r}^{(\delta)} = \varphi_{s,w, r} $$ in $C^\infty$-topology.

\begin{lemma}\label{osc1} For any $T\in(0, T_0)$, there exist  $C$ and $\delta_0>0$ such that   for $s, w, r  \in (0,1]$, $t\in [0, T]$  and $\delta \in [-\delta_0,\delta_0]$, $$  C \log |S_{\tilde{E}}|_{h_{\tilde{E}}}^{2} - C \leq \frac{\partial}{\partial \delta } \varphi_{s,w, r}^{(\delta)}   \leq  C .$$

 \end{lemma}

\begin{proof} Let $\Delta_{s,w,r}^{(\delta)}$ be the Laplace operator with respect to $\omega_{s,w,r}^{(\delta)} = \omega_{t,s}^{(\delta)} + \ddbar \varphi_{s,w,r}^{(\delta)}$. Notice that $\frac{\partial}{\partial \delta } \varphi^{(\delta)}_{s, w, r} = -\varphi_0$ when $t=0$ and

$$( \ddt{} -\Delta_{s,w, r}^{(\delta)} )  (\frac{\partial}{\partial \delta } \varphi_{s,w, r}^{(\delta)}  ) =  - tr_{\omega_{s,w, r}^{(\delta)}  }(\omega_0) \leq 0.$$
It is easy to see that $\frac{\partial}{\partial \delta  } \varphi_{s,w, r}^{(\delta)} $ is uniformly bounded above by the maximum principle. By the similar argument as before, $\varphi_{s,w, r}^{(\delta)} $ is bounded in $L^{\infty}(X)$ uniformly  in $s,w, r, \delta $ and  $t\in [0, T]$.

Consider $H = e^{-A^2t} \frac{\partial}{\partial \delta  } \varphi_{s,w, r}^{(\delta)}  + A^2\varphi_{s,w, r}^{(\delta)}  - A \log|S_{\tilde{E}}|^2_{h_{\tilde{E}}} .$
$H$ is uniformly bounded when $t=0$. There exist constants $C_1$, $C_2$, $C_3$ and $C_4>0$ such that \begin{eqnarray*}
&& ( \ddt{} -\Delta_{s,w, r}^{(\delta)} )  H\\
 &=&  A^2 \ddt{} \varphi_{s,w, r}^{(\delta)}  + tr_{\omega_{s,w, r}^{(\delta)} }( A^2 \omega_{t, s}^{(\delta)}- A Ric(h_{\tilde{E}}) - e^{-A^2 t} \omega_0) - nA^2 - A^2 e^{-A^2t}\frac{\partial}{\partial \delta  } \varphi_{s,w, r}^{(\delta)}  \\
&\geq& - A^2 \log \frac{\Omega_{w,r}}{ (\omega_{s,w,r}^{(\delta)})^n} + C_1 ~tr_{\omega_{s,w, r}^{(\delta)} } (\omega_{t,s}^{(\delta)} )   - A^2 H - A^3  \log|S_{\tilde{E}}|^2_{h_{\tilde{E}}} - C_2\\
&\geq&  -A^2 \log \frac{\Omega_{w,r}}{ (\omega_{t,s} ^{(\delta)})^n} - A^3 \log|S_{\tilde{E}}|^2_{h_{\tilde{E}}} -  A^2 H - C_3\\
&\geq& -A^2H -C_4 .
\end{eqnarray*}

Therefore $H$ is uniformly bounded from below by the maximum principle and the lemma easily follows.

\qed
\end{proof}

By the same argument as for $\delta=0$,  for $s, \delta\in (0,1]$, $\lim_{w\rightarrow 0} ( \limsup_{r\rightarrow 0} \varphi^{(\delta)}_{s,w,r})^*$ exists on $X\setminus \tilde{E}$ and  there exists $\varphi^{(\delta)}_s \in L^{\infty}([0, T]\times X) \cap C^\infty([0, T]\times X\setminus \tilde{E}) $  such that  $$\varphi^{(\delta)}_s =\lim_{w\rightarrow 0} ( \limsup_{r\rightarrow 0} \varphi^{(\delta)}_{s,w,r})^*$$
and  it satisfies the following Monge-Amp\`ere equation%
 \begin{equation}\label{vaphidelta}
\left\{
\begin{array}{l}
 { \displaystyle \ddt{}\varphi_{s}^{(\delta)} = \log \frac{( \omega_{t,s}^{(\delta)}  + \ddbar \varphi_{s}^{(\delta)} )^n }{ \Omega} }, ~~~~on ~X\setminus \tilde{E}
\\
\\
 \varphi_{s}^{(\delta)} (0, \cdot)= (1-\delta) \varphi_0.
\end{array} \right.
\end{equation}


Let $$\varphi^{(\delta)} = \lim_{s\rightarrow 0} \varphi^{(\delta)}_s$$ as $\varphi_s^{(\delta)}$ is decreasing as $s\rightarrow 0$.  Then $\varphi^{(\delta)}\in L^\infty([0, T]\times X)\cap C^{\infty}((0,T]\times X\setminus \tilde{E})$ solves the following Monge-Amp\`ere equation.
 \begin{equation}\label{vaphidelta1}
\left\{
\begin{array}{l}
 { \displaystyle \ddt{}\varphi^{(\delta)} = \log \frac{( \omega_{t}^{(\delta)}  + \ddbar \varphi^{(\delta)} )^n }{ \Omega} }, ~~~~on ~[0, T]\times X\setminus \tilde{E}
\\
\\
 \varphi^{(\delta)} (0, \cdot)= (1-\delta) \varphi_0,
\end{array} \right.
\end{equation}
where $\omega_t^{(\delta)} = \omega_{t,0}^{(\delta)}= (1-\delta)\omega_0 + t\chi.$


\begin{lemma} \label{lip1}  For any $T\in(0, T_0)$, there exist  $C$ and $\delta_0>0$ such that on $[0, T]\times X$,   for all $s \in (0,1]$ and $\delta_1, \delta_2 \in [-\delta_0,\delta_0]$,
\begin{equation}
 | \varphi_{s}^{(\delta_1)} - \varphi_s^{(\delta_2)}  | \leq  C |\delta_1 - \delta_2|( 1- \log |S_{\tilde{E}}|_{h_{\tilde{E}}}^{2} ),
\end{equation} and so
\begin{equation}
 | \varphi^{(\delta_1)} - \varphi^{(\delta_2)}  | \leq  C |\delta_1 - \delta_2|( 1- \log |S_{\tilde{E}}|_{h_{\tilde{E}}}^{2} )
\end{equation}

\end{lemma}

\begin{proof} This is an immediate result of Lemma \ref{osc1} by letting $w, r \rightarrow 0$ and then $s\rightarrow 0$.

\qed
\end{proof}

\begin{corollary} For any $T\in(0, T_0)$ and $K\subset\subset X\setminus \tilde{E}$, $\varphi^{(\delta)}$ converges to $\varphi$ uniformly in $L^\infty([0, T]\times K)$ as $\delta \rightarrow 0.$

\end{corollary}

\begin{proof} By Lemma \ref{lip1}, $\varphi_s^{(\delta)}$ is uniformly Lipschitz in $\delta$ on $K$ and $s\in [0,1)$. The corollary follows easily by letting $\delta\rightarrow 0$.

\qed\end{proof}


Now we are able to prove our main result for the existence and uniqueness of the Monge-Amp\`ere solution.

\medskip

\noindent{\bf Proof of Theorem \ref{firstthm}}

For any $T\in [0, T_0)$, Corollary \ref{corexist} gives the existence of the solution for the Monge-Amp\`ere flow (\ref{degmaflowlim}). By Lemma \ref{timebound1}, $\ddt{\varphi}$ is integrable and so $\ddt{ \varphi} = \log \frac{ (\omega_t + \ddbar\varphi)^n} { \Omega}$ as distributions. Now it suffices to prove the uniqueness for the solution on $[0, T] \times X$ for any $0<T<T_0$.

Suppose there is another solution $\varphi'$ satisfying the Monge-Amp\`ere flow (\ref{degmaflowlim}) such that $\varphi' \in  C^{\infty}([0, T])\times X \setminus {\tilde{E}} )$ and  $\varphi(t, \cdot) \in  PSH(X, \omega_{t}) \cap L^{\infty}(X)$ for each $t\in [0,T]$.

First, we show that $$\varphi' \leq \varphi.$$
Let $\psi_{s, \epsilon} = \varphi_{s} -\varphi' - \epsilon s \log|S_{\tilde{E}}|_{h_{\tilde{E}}}^2$ for sufficiently small $\epsilon>0$. Then $\psi_{\epsilon} \in C^\infty( [0, T]\times (X\setminus \tilde{E}))$ and

$$\ddt{}\psi_{s,\epsilon}  = \log \frac{ (\omega_t + \ddbar \varphi' + s(\vartheta -\epsilon Ric(h_{\tilde{E}})) + \ddbar\psi_{s,\epsilon})^n }   { (\omega_t + \ddbar\varphi')^n  } .$$ Suppose $\psi_{s,\epsilon} (t, z_{min})= \min_X  \psi_{s,\epsilon} (t, \cdot)$. Then $z_{min} \in X\setminus \tilde{E}$ since both $\varphi_s$ and $\varphi'$  $\in L^\infty(X)$.  If we choose $\epsilon$ sufficiently small, then by the maximum principle,
$$ \ddt{}  \psi_{s,\epsilon} (t, z_{min}) \geq  \log \frac{ (\omega_t (t, z_{min})+ \ddbar \varphi'(t, z_{min})  + \ddbar\psi_{s,\epsilon} (t, z_{min}) )^n }   { (\omega_t (t, z_{min})+ \ddbar\varphi'(t, z_{min}))^n  } \geq 0  .$$

Note that $\psi_{s,\epsilon} (0, \cdot) = -\epsilon s\log |S_{\tilde{E}}|_{h_{\tilde{E}}}^2\geq 0 $ and so
$$\psi_{s,\epsilon} \geq 0$$
for any $\epsilon$ sufficiently small. Therefore by letting $\epsilon \rightarrow 0$, we have

$$\varphi' \leq \varphi_{s}$$ and so $$\varphi' \leq \varphi$$  by letting $s\rightarrow 0$.

In order to prove $\varphi\leq \varphi'$,
we let  $$v_{\delta} = \varphi' -  \varphi^{( \delta)} - \delta^2  \log|S_{\tilde{E}}|_{h_{\tilde{E}}}^2.$$
At $t=0$, $v_\delta = \delta \varphi_0 - \delta^2 \log |S_{\tilde{E}}|_{h_{\tilde{E}}}^2$.

Suppose $v_{\delta} (t, z_{min})= \min_X  v_{\delta} (t, \cdot)$. Then $z_{min} \in X\setminus \tilde{E}$ since both $\varphi^{( \delta)}$ and $\varphi'$ $\in L^\infty(X)$.  By the maximum principle, if we choose $\delta$ sufficiently small, then  at $(t, z_{min})$,

\begin{eqnarray*}
\ddt{}  v_{\delta} &=& \log \frac{  ( \omega_{t} ^{(\delta)}   + \ddbar \varphi^{(\delta)}   + \delta (\omega_0 -\delta Ric (h_{\tilde{E}}) )  + \ddbar v_\delta )^n    }  {   ( \omega_{t}^{(\delta)} + \ddbar \varphi^{(\delta)}   )^n  } \\
&\geq&  \log \frac{  ( \omega_{t}^{(\delta)}     + \ddbar \varphi^{(\delta)}  + \ddbar v_\delta )^n    }  {   ( \omega_{t}^{(\delta)}   + \ddbar \varphi^{(\delta)} )^n  } \\
&\geq&0.
\end{eqnarray*}
Therefore $$v_{\delta}\geq \inf_X v_{\delta}(0, \cdot) \geq \delta  \inf_X( \varphi_0 - \delta \log |S_{\tilde{E}}|^2_{h_{\tilde{E}}}),$$
 and so we have for $t\in [0,T]$, $$       \varphi^{(\delta)}  + \delta  \inf_X( \varphi_0 - \delta \log |S_{\tilde{E}}|^2_{h_{\tilde{E}}})  + \delta^2  \log|S_{\tilde{E}}|_{h_{\tilde{E}}}^2           \leq   \varphi' \leq \varphi.$$
For any $K\subset \subset X\setminus {\tilde{E}}$, there exists a constant $C_K>0$ such that $$| \varphi^{(\delta)}   - \varphi |   \leq C_K \delta$$ for sufficiently small $\delta>0$ by Lemma \ref{lip1}  .

Therefore on $K$,   $$   \varphi' \geq \varphi   + \delta \inf_X( \varphi_0 - \delta \log |S_{\tilde{E}}|^2_{h_{\tilde{E}}}) - C_K \delta.$$



Letting $\delta\rightarrow 0$ and then $K\rightarrow X\setminus {\tilde{E}}$, we have on $X\setminus {\tilde{E}}$,

$$\varphi'\geq \varphi.$$

Therefore we have proved the uniqueness of the solution on $[0, T] \times X$. The theorem is proved by letting $T \rightarrow T_0$.

\qed


\subsection{Monge-Amp\`ere flows with rough  and degenerate initial data}  \label{3.3}

In this section, we will generalize Theorem \ref{firstthm} for Monge-Amp\`ere flows with rough initial data. The main result will be applied to the K\"ahler-Ricci flow on singular projective varieties with surgery.

Let $X$ be an $n$-dimensional projective manifold. Let $L_1$ and $L_2$ be two holomorphic line bundles on $X$ satisfying {\bf Condition A} along with $\omega_0\in c_1(L_1)$ and $\chi\in c_1(L_2)$ being smooth closed $(1,1)$-forms. Let $\Omega$ be a non-negative $(n,n)$-form on $X$ satisfying {\bf Condition B}. Let $$PSH_p(X, \omega_0, \Omega) = \{ \varphi\in PSH(X, \omega_0) \cap L^\infty(X) ~|~  \frac{ (\omega_0 + \ddbar\varphi)^n}{\Omega}\in L^p(X, \Omega)   \}$$ for $p >0$ and $$T_0 =\sup \{ t>0~|~ L_1 + tL_2 ~\textnormal{ is~semi-ample}\}. $$

Since $L_1$ is big and semi-ample, we denote $Exc(L_1)$ be the exceptional locus for the linear system $|mL_1| $ for sufficiently large $m$. Without loss of generality, we can assume that $\tilde{E}$ as defined in Section \ref{3.2} contains $Exc(L_1)$.

For any $\varphi_0 \in PSH_p (X, \omega_0, \Omega)$ for some $p>1$. We define the non-negative function $F$ by the following Monge-Amp\`ere equation
\begin{equation}
(\omega_0 + \ddbar \varphi_0)^n = F e^{\varphi_0} \Theta.
\end{equation}

\begin{lemma} Let $\Theta$ be a smooth volume form on $X$ and $F = \frac{ ( \omega_0 + \ddbar\varphi_0 )^n}{\Omega}$. Then there exists $p'>1$ such that \begin{equation} F\in L^{p'}(X, \Theta). \end{equation}

\end{lemma}

\begin{proof} The lemma is easily proved by H\"older's inequality and the fact that $\varphi_0 \in L^{\infty}(X)$.

\qed\end{proof}

There exist a family of positive functions $\{ F_s\}_{s\in(0,1]}$ such that $F_s \in C^\infty(X)$ and $$\lim_{s \rightarrow 0 } || F_s - F||_{L^{p'}(X, \Omega)} = 0.$$ We let $F_0 = F$ and then consider the following Monge-Amp\`ere equations \begin{equation}
(\omega_0 + s\vartheta + \ddbar \varphi_{(0,s)} )^n = F_s e^{\varphi_{(0,s)} } \Theta
\end{equation}
and
\begin{equation}
(\omega_0  + s\vartheta+ \ddbar \hat{\varphi}_{(s,\gamma)} )^n = F_{s+\gamma} e^{\hat{\varphi}_{(s,\gamma)} } \Theta
\end{equation}

Obviously, $\varphi_{(0,s)} \in C^{\infty}(X)$ and $\hat{\varphi}_{(s,r)} \in C^\infty(X)$ by Yau's theorem \cite{Y2} for $s>0$ and $\gamma>0$. Furthermore, both $||\varphi_{0,s}||_{L^\infty(X)}$ and $||\hat{\varphi}_{0,s}||_{L^\infty(X)}$ are uniformly bounded for $s\in [0,1]$.

\begin{lemma} There exists a decreasing function $\mu(s) \geq 0$ for $s  \geq 0$  with $\lim_{s\rightarrow 0} \mu(s) =0$ such that
\begin{equation}
 ||\hat{\varphi}_{(0,s)} - \varphi_0 ||_{L^\infty(X)} \leq \mu(s).
\end{equation}
In particular,
\begin{equation}
\lim_{s\rightarrow 0} ||\hat{\varphi}_{(0,s)} - \varphi_0 ||_{L^\infty(X)} = 0.
\end{equation}

\end{lemma}

\begin{proof} Notice $F_s$ converges to $F$ uniformly in $L^{p'}(X, \Omega)$. The lemma follows by combining the proof of the stability theorems for degenerate Monge-Amp\`ere equations in \cite{Ko2} and \cite{DZ} (see Theorem \ref{stability}).

\qed
\end{proof}


\begin{lemma}\label{dapu}  There exists $C>0$ such that  \begin{equation}
\hat{\varphi}_{(0,s)}  \leq \varphi_{(0,s)} \leq \hat{\varphi}_{(0,s)} + Cs (1 -  \log |S_{ \tilde{E} }|_{h_{\tilde{E}}}^2) .
\end{equation}

\end{lemma}

\begin{proof}

First let $\psi^+ = \varphi_{(0,s)} - \hat{\varphi}_{(0,s)} - s^m \log |S_{ \tilde{E} }|_{h_{\tilde{E}}}^2$ for sufficiently large $m>0$. Notice that $\hat{\varphi}_{0 , s} \in C^\infty(X \setminus \tilde{E})$ for $s>0$.
Then
$$\frac{(\omega_0 + \ddbar \hat{\varphi}_{(0,s)} + (s\vartheta -s^m Ric(h_{\tilde{E}})) + \ddbar \psi^+)^n } {(\omega_0 + \ddbar \hat{\varphi}_{(0,s)})^n} = e^{\psi^+ + s^m \log |S_{ \tilde{E} }|_{h_{\tilde{E}}}^2} \leq e^{\psi^+} .$$
Note that $s\vartheta - s^m Ric(h_{\tilde{E}})>0$ for $m>>1$. By the maximum principle, $\psi^+ \geq 0$ and so $\hat{\varphi}_{(0,s)} +  s^m \log |S_{ \tilde{E} }|_{h_{\tilde{E}}}^2 \leq \varphi_{(0,s)}$. Then by letting $m \rightarrow \infty$, we have
$$\hat{\varphi}_{0, s} \leq \varphi_{0,s}.$$

Now we will bound $\varphi_{(0,s)}$ from above. Let $\psi^- = \frac{1}{1+A^2 s}\varphi_{(0,s)} - \hat{\varphi}_{(0,s)} + As \log |S_{ \tilde{E} }|_{h_{\tilde{E}}}^2$. Then

\begin{eqnarray*}
&& \frac{(1+ A^2s)^n (\omega_0 + \ddbar \hat{\varphi}_{(0,s)} - (\frac{A^2s}{1+A^2 s} \omega_0 - \frac{s}{1+A^2s} \vartheta - A s Ric(h_{\tilde{E}})) + \ddbar \psi^-)^n } {(\omega_0 + \ddbar \hat{\varphi}_{(0,s)})^n} \\
&=& e^{\psi^- + \frac{A^2 s}{1+A^2s} \varphi_{(0,s)}  - As  \log |S_{ \tilde{E} }|_{h_{\tilde{E}}}^2} \\
&\geq& e^{\psi^- + \frac{A^2 s}{1+A^2s} \varphi_{(0,s)} }.
\end{eqnarray*}

Note that $\frac{A^2s}{1+A^2 s} \omega_0 - \frac{s}{1+A^2s} \vartheta - A s Ric(h_{\tilde{E}}) >0$ for $A>>1$ and $s>0$ sufficiently small. By the maximum principle, $$\psi^- \leq  n \log (1+A^2 s) + \frac{A^2 s}{1+A^2 s} || \varphi_{(0,s)} ||_{L^{\infty}(X)} . $$ Therefore $\varphi_{(0,s)} \leq \hat{\varphi}_{(0,s)} + Cs - As \log |S_{ \tilde{E} }|_{h_{\tilde{E}}}^2$ for some $C>0$.

\qed
\end{proof}


\begin{corollary} Let $\vartheta$ be a K\"ahler form on $X$. Then for $s>0$, there exists $ \varphi_{0, s} \in PSH(X, \omega_0+ s\vartheta) \cap C^\infty( X)$ such that  for any $K \subset \subset X\setminus \tilde{E} $

\begin{equation}
\lim_{s\rightarrow 0} ||\varphi_{(0, s)} - \varphi_0||_{L^\infty ( K  )} =0.
\end{equation}

\end{corollary}


\begin{lemma}There exist $\gamma_0>0$ and $C>0$ such that for $0<\gamma <\gamma_0$,
\begin{equation}
| \varphi_{(0,s+\gamma)}  - \varphi_{(0,s)} | \leq  C \gamma ( 1 -  \log |S_{ \tilde{E} }|_{h_{\tilde{E}}}^2 ) + C \mu(s+\gamma) .
\end{equation}

\end{lemma}

\begin{proof} Notice that
$$
|\varphi_{(0, s+\gamma)} - \varphi_{(0,s)}|
= |\varphi_{(0, s+\gamma)} - \hat{\varphi}_{(s,\gamma)}|+|\varphi_{(0, s)} - \hat{\varphi}_{(s,\gamma)}|.
$$
By the same argument as that in the proof of Lemma \ref{dapu}, there exists $C>0$ such that
$$|\varphi_{(0, s+\gamma)} - \hat{\varphi}_{(s,\gamma)}| \leq Cr (1 -  \log |S_{ \tilde{E} }|_{h_{\tilde{E}}}^2).$$ Applying the family version of the stability theorem, there exists $C>0$ such that

$$ ||\varphi_{(0, s)} - \hat{\varphi}_{(s,\gamma)}||_{L^\infty(X)} \leq C ||F_{s+\gamma} - F_s||_{L^1(X, \Omega)}^{\frac{1}{n+3}}.$$
The Lemma then follows from the above estimates.

\qed
\end{proof}


Consider the following family of Monge-Amp\`ere equations.
\begin{equation}\label{degmaflow5}
\left\{
\begin{array}{cl}
&{ \displaystyle \ddt{\varphi^{(\delta)}_{s,w,r} } = \log
 \frac{ ( \omega^{(\delta)}_{t,s} + \ddbar \varphi^{(\delta)}_{s,w, r}  ) ^n } {\Omega_{w,r}} },\\
&\\
& \varphi^{(\delta)}_{s, w, r} (0,\cdot)=(1-\delta) \varphi_{(0,s)} ,
\end{array} \right.
\end{equation}
where $\omega_{t,s}^{(\delta)} = (1-\delta)\omega_0 + t\chi + s\vartheta$ is defined as in Section \ref{3.2}.

For any $T \in [0, T_0)$, there exists $\delta_0>0$ such that for $\delta\in [-\delta_0, \delta_0]$, the equation (\ref{degmaflow5}) admits a smooth solution on $[0, T]\times X$ as shown in Section \ref{3.2}. We will then fix such $T$ and  $\delta_0$.

\begin{lemma} There exists $C>0$ such that for $s$, $w$, $r\in (0,1]$ and $\delta\in [-\delta_0, \delta_0]$,
\begin{equation}
||\varphi^{(\delta)}_{s,w,r}(t, \cdot)||_{L^\infty([0, T]\times X)} \leq C.
\end{equation}

\end{lemma}

\begin{proof} It can be proved by the same argument as that in the proof of Lemma \ref{c0estimate1}.

\qed\end{proof}

\begin{lemma} There exists $C>0$ such that on $[0, T]\times X$, for all $s$, $w$,$r\in (0,1]$ and $\delta \in [-\delta_0, \delta_0]$,
\begin{equation}
-C  \leq t \dot{\varphi}^{(\delta)}_{s,w,r} \leq C. \end{equation}

\end{lemma}

\begin{proof} The upper bound can be proved using the same argument as that in Lemma \ref{volsm} by applying the maximum principle on $$H^+ = t\dot{\varphi}^{(\delta)}_{s,w,r} - \varphi^{(\delta)}_{s,w,r}.$$

In order to prove the lower bound. We consider the following family of Monge-Amp\`ere equations

$$ (\omega_{0,s} + \ddbar \phi_{s, w,r} )^n = A_{s,w,r} \Omega_{w,r}$$
where $A_{s,w,r} = \frac{[\omega_{0,s}]^n}{\int_X \Omega_{w,r}}$ and  $\sup_X \phi_{s,w,r} = 0.$
Then  $A_{s,w,r}$ is uniformly bounded from above and below for $s$, $w$ and $r\in (0,1]$. As $A_{s,w,r} \Omega_{w,r}$ is uniformly bounded in $L^p(X, \Theta)$ for some $p>1$,  $\phi_{s,w,r} $  uniformly bounded in $L^\infty(X)$ for  $s,w,r\in (0,1]$.

Let
 $$H^- = t\dot{\varphi}^{(\delta)}_{s,w,r} + A^2 \varphi^{(\delta)}_{s,w,r} -  A \phi_{s^2,w,r} .$$

Let $\Delta^{(\delta)}_{s,w,r}$ be the Laplace operator associated to $\omega^{(\delta)}_{s,w,r}$. Then there exist $C_1$, $C_2$ and $C_3>0$ such that
\begin{eqnarray*}
&& (\ddt{} - \Delta^{(\delta)}_{s,w,r}) H^- \\
&=&  tr_{\omega^{(\delta)}_{s,w,r}} ( A^2 \omega^{(\delta)}_{t,s} + t \chi + A \ddbar \phi_{s^2,w,r}) + (A^2 +1)\dot{\varphi}^{(\delta)}_{s,w,r} - A^2n \\
&\geq & tr_{\omega^{(\delta)}_{s,w,r}} ( A \omega_{0,s^2} + A \ddbar \phi_{s^2,w,r}) + (A^2 +1)\dot{\varphi}^{(\delta)}_{s,w,r} - A^2n \\
&\geq& C_1 \left( \frac{ (\omega_{0, s^2} + \ddbar \phi_{s^2, w,r})^n } {(\omega^{(\delta)}_{s,w,r})^n } \right)^{\frac{1}{n}} + (A^2 + 1)\log \frac{(\omega^{(\delta)}_{s,w,r})^n}{\Omega_{w,r}} -A^2n\\
&\geq& C_2 \left( \frac{ \Omega_{w,r} } {(\omega^{(\delta)}_{s,w,r})^n } \right)^{\frac{1}{n}} - C_3\\
&\geq& - C_3
\end{eqnarray*}

 Applying the maximum principle, $H^-$ is uniformly bounded from below since both $\varphi_{s,w,r}^{(\delta)}$ and $\phi_{s,w,r}$ are uniformly bounded in $L^\infty(X)$. We are done.

\qed
\end{proof}

We have the following volume estimate.

\begin{corollary}
There exists $C>0$ such that on $[0, T]\times X$, for all $s$, $w$,$r\in (0,1]$ and $\delta \in [-\delta_0, \delta_0]$,
\begin{equation}\label{volest0}
e^{-\frac{C}{t}}  \leq \frac{(\omega^{(\delta)}_{s,w,r})^n}{\Omega_{w,r}}  \leq e^{\frac{C}{t}}.
\end{equation}

\end{corollary}

\begin{lemma} There exist $\alpha>0$ and $C>0$ such that on $[0, T]\times X$, for all $s$, $w$,$r\in (0,1]$ and $\delta \in [-\delta_0, \delta_0]$,
\begin{equation} tr_{\vartheta}( \omega^{(\delta)}_{s,w,r} )\leq  C |S_{\tilde{E}} |_{h_{\tilde{E}}}^{-\frac{2\alpha} {t} } . \end{equation}

\end{lemma}
\begin{proof} 
Define $$H= t \log tr_{\vartheta} (\omega^{(\delta)}_{s,w,r}) - A^2 \varphi^{(\delta)}_{s,w,r} + A \log |S_{\tilde{E} }|_{h_{\tilde{E}}}^2.$$
By  applying $(\ddt{} - \Delta^{(\delta)}_{s,w,r})$ to $H$, the lemma follows from the maximum principle applied to $(\ddt{} - \Delta_{s,w,r}^{\delta}) H$ and   the similar argument in the proof of Lemma \ref{c2smoothing}.

\qed
\end{proof}

 \begin{lemma}  For $s\in (0,1]$ and $\delta\in [-\delta_0, \delta_0]$, $\varphi^{ (\delta)}_{s} \in C^\infty( (0, T] \times X\setminus \tilde{E})$.

 \end{lemma}

 \begin{proof} The lemma follows from the standard argument combined with the $C^2$-estimate.

 \qed\end{proof}

Obviously for $s,w, r \in (0, 1]$, $\varphi_{s,w, r}^{(\delta)} $ is a smooth family in $\delta$ and   $$\lim_{\delta\rightarrow 0 } \varphi_{s,w, r}^{(\delta)} = \varphi_{s,w, r} $$ in $C^\infty$-topology.

 For each $s\in (0, 1]$, let

 \begin{equation}
 \varphi^{(\delta)}_s = \lim_{w\rightarrow 0 }(  \limsup_{r\rightarrow 0} \varphi^{(\delta)}_{s,w,r} )^*
 \end{equation}
 and
 \begin{equation}
 \varphi_s = \lim_{w, r\rightarrow 0 } \varphi_{s,w,r}= \lim_{w, r\rightarrow 0 } \varphi^{(0)}_{s,w,r}
. %
 \end{equation}

\begin{lemma}\label{lip2} There exist constants $C$ and $\delta_0>0$ such that   for $s, w, r  \in (0,1]$, $t\in [0, T]$  and $\delta \in [-\delta_0,\delta_0]$,

\begin{equation}
  C \log |S_{\tilde{E}}|_{h_{\tilde{E}}}^{2} - C \leq \frac{\partial}{\partial \delta } \varphi_{s,w, r}^{(\delta)}   \leq  C .
\end{equation}

 \end{lemma}

\begin{proof} The lemma can be proved by the same argument in the proof of Lemma \ref{osc1}.

\qed
\end{proof}


\begin{lemma} \label{difference}  There exist $\delta_0>0$ and  $C>0$ such that for $s\in [0,1]$ and  $0<\delta< \delta_0$

\begin{equation}
\varphi^{(\delta)}_{s} \leq \varphi_{s+ \delta^3} -  \delta^2 \log |S_{\tilde{E}}|_{h_{\tilde{E}}}^{2} + C (\mu(s+\delta)+\delta).
\end{equation}

\end{lemma}

\begin{proof}   Let $\psi_{\delta} = \varphi_{s+ \delta^3} - \varphi^{(\delta)}_{s}  - \delta^2 \log |S_{\tilde{E}}|_{h_{\tilde{E}}}^{2} + A( \mu(s+\delta)+ \delta)$.

Then for each $t\in (0, T]$, at the maximal point of $\psi_\delta$,
\begin{eqnarray*}
\ddt{} \psi_{\delta} &=& \log \frac{ ( \omega^{(\delta)}_s + \ddbar \varphi^{(\delta)}_s + \delta (\omega_0 + \delta^2 \vartheta - \delta Ric( h_{\tilde{E}})) + \ddbar \psi_{\delta})^n} {(  \omega^{(\delta)}_s + \ddbar \varphi^{(\delta)}_s)^n  }\\
&\geq& \log \frac{ ( \omega^{(\delta)}_s + \ddbar \varphi^{(\delta)}_s  + \ddbar \psi_{\delta})^n} {(  \omega^{(\delta)}_s + \ddbar \varphi^{(\delta)}_s)^n  }\\
&\geq& 0,
\end{eqnarray*}
for sufficiently small $\delta>0$.

For sufficiently large $A>0$, $$\psi_{\delta} |_{t=0} = \varphi_{(0, s+ \delta^3)} - (1-\delta)\varphi_{(0, s)} -  \delta^2 \log |S_{\tilde{E}}|_{h_{\tilde{E}}}^{2} + A \mu(s+\gamma) + A\delta \geq 0 .$$

Therefore  $\psi_{\delta} \leq 0 $ on $[0, T] \times X$ by the maximum principle.

\qed
\end{proof}


\begin{lemma} There exist $\delta_0>0$ and  $C>0$ such that on $[0, T] \times X$, for $s\in [0,1]$ and  $0<\delta< \delta_0$,

 \begin{equation}
 \varphi^{(\delta)}_{s+ \delta^3} \leq \varphi_{s} - \delta^2 \log |S_{\tilde{E}}|_{h_{\tilde{E}}}^{2} + C(\mu(s+\delta) + \delta).
 \end{equation}

\end{lemma}

\begin{proof} Let $\psi_{\delta} =  \varphi^{(\delta)}_{s+ \delta^3}- \varphi_{s} + \delta^2 \log |S_{\tilde{E}}|_{h_{\tilde{E}}}^{2} - A (\mu(s+\delta) +  \delta) $.

 Then for each $t\in (0, T]$, at the minimal point of $\psi_\delta$,

\begin{eqnarray*}
\ddt{} \psi_{\delta} &=& \frac{ ( \omega_s + \ddbar \varphi_s - (\delta \omega_0 -  \delta^3  \vartheta - \delta^2 Ric( h_{\tilde{E}}) ) + \ddbar \psi_{\delta})^n} {(  \omega_s + \ddbar \varphi_s)^n  } \\
& \leq &  \frac{ ( \omega_s + \ddbar \varphi_s + \ddbar \psi_{\delta})^n} {(  \omega_s + \ddbar \varphi_s)^n  }
\end{eqnarray*}
for sufficiently small $\delta>0$.

For sufficiently large $A >0$, $$\psi_{\delta} |_{t=0} = (1- \delta) \varphi_{(0, s+ \delta^3)} - \varphi_{(0, s)}  +  \delta^2 \log |S_{\tilde{E}}|_{h_{\tilde{E}}}^{2}  - A(\mu(s+\delta) +  \delta) \leq 0 .$$
Therefore  $\psi_{\delta} \leq 0 $ on $[0, T\times X]$ by the maximum principle if $\delta>0$ is sufficiently small.

\qed\end{proof}



Then we can show that $\{ \varphi_s \}_{s\in (0,1]}$ is a Cauchy family in $L^\infty([0,T]\times K)$ for any compact subset in $X\setminus \tilde{E}$.

\begin{lemma}  On any $K \subset \subset X \setminus \tilde{E}$,

 \begin{equation}
    \lim_{s_1, s_2 \rightarrow 0 } || \varphi_{s_1} - \varphi_{s_2}||_{L^\infty([0, T]\times K)} = 0.
 \end{equation}

\end{lemma}

\begin{proof} Assume $\delta = s_2 - s_1 \geq 0$.  Then on $[0, T]\times K$, by Lemma \ref{lip2} and Lemma \ref{difference}, there exist $C$ and $C'>0$ such that

$$\varphi_{s_1} \leq \varphi_{s_1}^{(\delta^{1/3})} + C\delta^{1/3} \leq  \varphi_{s_2} -  \delta^{2/3} \log |S_{\tilde{E}}|_{h_{\tilde{E}}}^{2} + C(\mu(s + \delta^{1/3}) + \delta^{1/3}) $$
and
\begin{eqnarray*}
\varphi_{s_1} & \geq & \varphi^{(\delta^{1/3})}_{s_1 + \delta} + \delta^{2/3} \log |S_{\tilde{E}}|_{h_{\tilde{E}}}^{2}  - C (\mu(s+\delta^{1/3}) + \delta^{1/3}) \\
&\geq& \varphi_{s_2} +  \delta^{2/3} \log |S_{\tilde{E}}|_{h_{\tilde{E}}}^{2}  -  C'  (\mu(s+\delta^{1/3}) + \delta^{1/3}) .
\end{eqnarray*}

 The lemma follows immediately by letting $s_1$ and $s_2\rightarrow 0$, $\delta \rightarrow 0$.

\qed\end{proof}



\begin{proposition}  For any $\varphi_0\in PSH_p(X, \omega_0, \Omega)$ for some $p>1$, there exists a unique $\varphi \in C^\infty( (0, T_0)\times (X\setminus \tilde{E})) $ with $\varphi(t, \cdot) \in PSH(X, \omega_0 + t\chi) \cap L^\infty(X)$ for each $t\in (0, T_0)$ such that

\begin{enumerate}

\item $\ddt{\varphi} = \log \frac{(\omega_0 + t\chi +\ddbar \varphi)^n}{\Omega}$ on $(0, T_0)\times X\setminus\tilde{E} $.

\item  for any $K \subset\subset X\setminus \tilde{E}$, $\lim_{t\rightarrow 0^-} ||\varphi(t, \cdot) - \varphi_0(\cdot)||_{L^\infty(K)} = 0$.

\item $||\varphi||_{L^\infty( (0, T]\times X))} $ is bounded for each $T< T_0$.

\end{enumerate}
Furthermore, for any $T\in (0, T_0)$, there exists $C>0$ such that on $[0, T]\times X$,
\begin{equation}\label{volest}
e^{-\frac{C}{t}}  \leq \frac{ (\omega_0+t\chi + \ddbar\varphi)^n}{\Omega}  \leq e^{\frac{C}{t}}.
\end{equation}

\end{proposition}

\begin{proof} It suffices to show $(2)$ and the volume estimate.  On any $K \subset\subset X \setminus \tilde{E}$,
\begin{eqnarray*}
&&||\varphi ( t, \cdot) - \varphi_0(\cdot)  ||_{L^\infty(K)} \\
&\leq& ||\varphi ( t, \cdot) - \varphi_s(\cdot)  ||_{L^\infty(K)} + ||\varphi_s ( t, \cdot) - \varphi_{(0,s)} (\cdot)  ||_{L^\infty(K)}+ ||\varphi_{(0,s)} ( \cdot) - \varphi_0(\cdot)  ||_{L^\infty(K)}\\
\end{eqnarray*}
For any $\epsilon>0$, let $s$ be sufficiently small such that $$||\varphi ( t, \cdot) - \varphi_s(\cdot)  ||_{L^\infty([0, T]\times K)} < \epsilon$$ and $$ ||\varphi_{(0,s)} ( \cdot) - \varphi_0(\cdot)  ||_{L^\infty(K)}  < \epsilon .$$
Fix such $s$. There exists $t_0>0$ such that  $$  ||\varphi_s ( t, \cdot) - \varphi_{(0,s)} (\cdot)  ||_{L^\infty([0, t_0] \times K)} < \epsilon.  $$
Therefore $\lim_{t\rightarrow 0^-} ||\varphi(t, \cdot) - \varphi_0(\cdot)||_{L^\infty(K)} = 0$.

The volume estimate (\ref{volest}) follows from (\ref{volest0}) by letting $s$, $w$, $r$ and $\delta \rightarrow 0$.

\qed\end{proof}


\begin{theorem}\label{allinone}
Let $X$ be an $n$-dimensional algebraic manifold. Let $L_1$ and $L_2$ be two holomorphic line bundles on $X$ satisfying {\bf Condition A} and $\omega_0\in c_1(L_1)$ and $\chi\in c_1(L_2)$ are smooth closed $(1,1)$-forms. Let $\Omega$ be an $(n,n)$-form on $X$ satisfying {\bf Condition B}. Let $$T_0 =\sup \{ t>0~|~ L_1 + tL_2 >0\}. $$

Then for any $\varphi_0\in PSH_p(X, \omega_0, \Omega)$ for some $p>1$, there exists a unique $\varphi \in C^0( [0, T_0) \times X\setminus \tilde{E}) \cap C^\infty( (0, T_0)\times (X\setminus \tilde{E})) $ with $\varphi(t, \cdot) \in PSH(X, \omega_0 + t\chi) \cap L^\infty(X)$ for each $t\in [0, T_0)$ such that

\begin{enumerate}

\item $\ddt{\varphi} = \log \frac{(\omega_0 + t\chi +\ddbar \varphi)^n}{\Omega}$ on $(0, T_0)\times X\setminus \tilde{E}$,

\item $\varphi |_{t=0} = \varphi_0$ on $X$.


\end{enumerate}

Furthermore, for any $T\in (0, T_0)$, there exists $C>0$ such that on $[0, T]\times X$,
\begin{equation}\label{volestt}
e^{-\frac{C}{t}}  \leq \frac{ (\omega_0+t\chi + \ddbar\varphi)^n}{\Omega}  \leq e^{\frac{C}{t}}.
\end{equation}
\end{theorem}

\begin{proof} It suffices to show the uniqueness of the solution $\varphi$ and it can be proved by the similar argument for the proof of Theorem \ref{firstthm}.

\qed\end{proof}



%
%
%





\section{K\"ahler-Ricci flow on varieties with log terminal singularities} \label{4}

\subsection{Notations} \label{4.1}

Let $X$ be a $\mathbf{Q}$-factorial projective variety with at worst log terminal singularities. We denote the singular set of $X$ by $X_{sing}$ and let $X_{reg} =  X\setminus X_{sing}$. Let  $\pi: \tilde{X} \rightarrow X$ be  the resolution of singularities and $K_{\tilde{X}} = \pi^* K_X + \sum a_i E_i$ where $E_i$ is the irreducible component of the exceptional locus $Exc(\pi)$ of $\pi$.  Since $X$ is log terminial,  $a_i >-1$.

Since $K_X$ is a $\mathbf{Q}$-Cartier divisor,  there exists a positive $m\in \mathbf{Z}$ such that $mK_X$ is Cartier.

\begin{definition} $\Omega$ is said to be a smooth volume form on $X$ if $\Omega$ is a smooth $(n,n)$-form on $X$ such that for any $z\in X$, there exists an open neighborhood $U$ of $z$ such that $$\Omega = f_U ( \alpha \wedge \overline{\alpha})^{\frac{1}{m}},$$ where $f_U$ is a smooth positive function on $U$ and $\alpha$ is a local generator of $mK_X$ on $U$.
\end{definition}

On each $U$, $\ddbar \log (\alpha\wedge\overline{\alpha})=0$ on $U$ and $\ddbar \log f_U$ is a well-defined smooth closed $(1,1)$-form on $U$ if we extend $f_U$ in the ambient space of $U$.  Then $\chi = \ddbar \log \Omega $ is a well-defined smooth closed $(1,1)$-form on $X$, furthermore, $\chi \in [K_X]$. We can then define $Ric(\Omega)$ to be $\chi = \ddbar\log h_\Omega$, where  $h_\Omega = \Omega^{-1}$ defines a smooth hermitian metric on $K_X$.

After pulling back $\Omega$ by the resolution $\pi$, $\pi^* \Omega$ is then a non-negative $(n,n)$-form on $\tilde{X}$. In particular,  $\pi^*\Omega$ has zeros or poles along the exceptional divisor $E_i$ of order $|a_i|$ and $$\pi^*\chi = Ric( \pi^* h_\Omega).$$


 Let $D$ be an ample divisor on $\tilde{X}$ such that $$\omega_D  = Ric(h_D) = - \ddbar\log h_D >0$$ where  $h_D$ is a hermitian metric equipped on the line bundle associated to $D$.

Let  $\iota : X \rightarrow \mathbf{CP}^N$ be any imbedding of $X$ into a projective space and
$\omega_0 $ be the pullback of a smooth K\"ahler metric from $ \mathbf{CP}^N$ in a multiple of the K\"ahler class $\mathcal{O}(1)$.
  Then $\omega_0$ is a smooth K\"ahler metric on $X$.  Since $[ \pi^*\omega_0]$ is the pullback of an ample class on $\mathbf{CP}^N$, it is a big and semi-ample divisor on $\tilde{X}$. By the Kodaira's lemma,  there exists an effective divisor ${\tilde{E}}$ on $\tilde{X}$ such that $$[\pi^* \omega_0] -\epsilon [\tilde{E}]$$ is ample for any $\epsilon>0$ sufficiently small. Furthermore, since $X$ is $\mathbf{Q}$-factorial, we can assume by Proposition \ref{qfac} that the support of $\tilde{E}$ is contained in the exceptional locus of $\pi$.
 There exists  a hermitian metric $h_{\tilde{E}}$ equipped on the line bundle associated to ${\tilde{E}}$ such that for sufficiently small $\epsilon>0$, $$\pi^* \omega_0 - \epsilon Ric(h_{\tilde{E}}) >0.$$  Let $S_D$ and $S_{\tilde{E}}$ be the defining section of $D$ and ${\tilde{E}}$.


\begin{definition}\label{khp}   Let $X$ be a $\mathbf{Q}$-factorial projective variety  with log terminal singularities and $H$ be a big and semi-ample $\mathbf{Q}$-divisor on $X$. Let $\omega_0\in [H]$ be a smooth closed $(1,1)$-form and $\Omega$ a smooth volume form on $X$. We define for $p\in(0, \infty]$,

\begin{equation}
PSH_{p}(X,\omega_0, \Omega) = \{ \varphi\in PSH(X,\omega_0) \cap L^\infty(X) ~|~\frac{(\omega_0 +\ddbar\varphi)^n}{\Omega}\in L^p(X, \Omega) \}.
\end{equation}
and
\begin{equation}
\mathcal{K}_{H, p}(X) = \{ \omega_0 +\ddbar\varphi ~|~ \varphi\in PSH_p (X, \omega_0, \Omega) \}.
\end{equation}

\end{definition}
The definition of $\mathcal{K}_{H, p}(X)    $ does not depend on the choice of the smooth closed $(1,1)$-form $\omega_0 \in [H]$ and the smooth volume form $\Omega$.


We define the following notion of the weak K\"ahler-Ricci flow on projective varieties with singularities.
\begin{definition}\label{weakdef}  Let $X$ be a $\mathbf{Q}$-factorial projective variety  with log terminal singularities and $\omega_0 \in [H]$ be a closed semi-positive $(1,1)$-current on $X$ associated to a big and semi-ample $\mathbf{Q}$-divisor $H$ on $X$. Suppose that $$T_0 = \sup\{ t>0 ~|~ H + t K_X ~is~ nef~\}>0.$$ A family of closed positive $(1,1)$-current $\omega (t, \cdot)$ on $X$ for $t\in [0, T_0)$ is called a solution of the unnormalized weak K\"ahler-Ricci flow if  the following conditions hold.

\begin{enumerate}

\item $\omega \in C^{\infty} ((0, T_0 )\times X\setminus D ) $, where $D$ is a subvariety of $X$. Let $\hat{\omega}_t \in [H+ tK_X]$ be a smooth family of smooth closed $(1,1)$-forms on $X$ for $t\in [0, T_0)$. Then $\omega= \hat{\omega}_t + \ddbar \varphi$ for some $\varphi \in C^0([0, T_0)\times X\setminus D )\cap C^\infty((0, T_0)\times X\setminus D)$  and $\varphi(t, \cdot) \in PSH(X, \hat{\omega}_t)\cap L^{\infty}(X)$ for all $t\in [0, T_0)$.

\item \begin{equation}
\left\{
\begin{array}{l}
{ \displaystyle \frac{\partial \omega } {\partial t} = - Ric(\omega)}, ~~~~~~~{on } ~~ (0, T_0) \times X\setminus D,\\
\\
\omega(0,\cdot)=\omega_0, ~~~~~~~~~~~~~{on}~~ X .
\end{array} \right.
\end{equation}

\end{enumerate}

\end{definition}

In particular, when $H$ is ample, $T_0$ is always positive and $X\setminus D = X_{reg}$.

We would like to prove the existence and uniqueness of the weak K\"ahler-Ricci flow on singular varieties if the initial metric satisfies certain regularity conditions.

The following theorems are well-known as the rationality theorem and base-point-free theorem in the Minimal Model Program (see \cite{KMM}, \cite{D}).
\begin{theorem}\label{rationality} Let $X$ be a projective manifold such that $K_X$ is not nef. Let $H$ be an ample $\mathbf{Q}$-divisor and let

\begin{equation}
\lambda = \max \{ t\in \mathbf{R}~|~ H + t K_X ~ is ~ nef~\}.
\end{equation}
Then $\lambda\in \mathbf{Q}$.

\end{theorem}

\begin{theorem}\label{basepointfree}

Let $X$ be a projective manifold. Let $D$ be a nef $\mathbf{Q}$-divisor such that $aD - K_X$ is nef and big for some $a>0$. Then $D$ is semi-ample.

\end{theorem}


\subsection{Existence and uniqueness of the weak K\"ahler-Ricci flow}  \label{4.2}

Let $X$ be a $\mathbf{Q}$-factorial projective variety  with log terminal singularities. Let $H$ be an ample $\mathbf{Q}$-divisor on $X$,  $\omega_0 \in [H] $ be a smooth K\"ahler metric and  $\Omega$ be a smooth volume form on $X$ and $\chi = \ddbar \log \Omega$.

Consider the ordinary differential equation for the K\"ahler class defined by the unnormalized Ricci flow on $X$

\begin{equation}\label{classflow}
\left\{
\begin{array}{ll}
&{ \displaystyle \ddt{ [\omega]} = [\chi] = [K_X] , }\\
&\\
& [\omega(t, \cdot)] = [\omega_0] = [H].
\end{array} \right.
\end{equation}
Then $$[\omega] = [\omega_0] + t [\chi].$$   Heuristically, if the K\"ahler-Ricci flow exists for $t\in [0, T)$, the unnormalized K\"ahler-Ricci flow should be equivalently to the following Monge-Amp\`ere flow

\begin{equation}\label{krflow2.1}
\left\{
\begin{array}{rcl}
&&{ \displaystyle \ddt{\varphi} = \log
 \frac{ ( \omega_0 + t\chi + \ddbar \varphi ) ^n } {\Omega} }\\
&&\\
&& \varphi(0,\cdot)=0 .
\end{array} \right.
\end{equation}

Let $\pi: \tilde{X} \rightarrow X$ be the resolution of singularities as defined in Section \ref{4.1}.
In order to define the Monge-Amp\`ere flow on $X$, one might want to lift the flow to the nonsingular model $\tilde{X}$ of $X$. However, $\omega_0$ is not necessarily K\"ahler on $\tilde{X}$ and $\Omega$ in general vanishes or blows up along the exceptional divisor of $\pi$ on $\tilde{X}$ unless the resolution $\pi$ is crepant, hence the lifted flow is degenerate near the exceptional locus. So we have to perturb the Monge-Amp\`ere flow (\ref{krflow2.1}) and obtain uniform estimates so that the flow might be allowed to be pushed down on $X$.

Let $$T_0 = \sup \{ t \geq 0 ~|~ [\omega_t] ~is ~ nef ~{\rm on }~ X  \} = \sup \{ t \geq 0 ~|~ H+ t K_X ~is ~ nef ~{\rm on }~ X  \}.$$
Then for any $t\in [0, T_0)$, $[\omega_t]$ is ample and $T_0>0$ is a rational number or $T_0= \infty$ by the rationality theorem \ref{rationality}.  The base-point-free theorem \ref{basepointfree} implies the following important proposition.

\begin{proposition} When $T_0<\infty$,
$H+ T_0 K_X$  is semi-ample.
\end{proposition}

\begin{theorem}\label{liftflow} Let $\varphi_0 \in PSH_p(X, \omega_0, \Omega)$ for some $p>1$. Then the Monge-Amp\`ere flow on $\tilde{X}\setminus {\tilde{E}}$

\begin{equation}\label{krflowlimit}
\left\{
\begin{array}{rcl}
&&{ \displaystyle \frac{\partial \tilde{\varphi} (t, \cdot) } {\partial t} = \log
 \frac{ ( \pi^* \omega_t + \ddbar \tilde{\varphi} ) ^n } { \pi^* \Omega} }\\
&&\\
&& \tilde{\varphi}(0,\cdot)=\pi^*\varphi_0 
\end{array} \right.
\end{equation}
has a unique solution $\tilde{\varphi} \in
 C^{\infty} ((0, T_0)\times \tilde{X} \setminus {\tilde{E}} ) \cap C^0( [0, T_0) \times \tilde{X} \setminus \tilde{E}) $ such that for all $t\in [0, T_0)$, $\varphi (t, \cdot) \in L^{\infty} (\tilde{X}) \cap PSH(\tilde{X}, \pi^* \omega_t )$. Furthermore, $\tilde{\varphi}$ is constant along each connected fibre of $\pi$, hence $\tilde{\varphi}$ descends to a unique solution $\varphi \in C^{\infty}((0, T_0) \times X_{reg}) \cap C^0([0, T_0)\times X_{reg})$ of the Monge-Amp\`ere flow (\ref{krflow2.1}) such that  for each $t\in [0, T_0)$, $\varphi \in PSH(X, \omega_t) \cap L^\infty(X)$.

\end{theorem}

\begin{proof} Since $[\pi^*\omega_0]$ corresponds to  a big and semi-ample divisor on $\tilde{X}$ and $[\pi^* \omega_0] - \epsilon [\pi^* \chi]$ is also big and semi-ample for sufficiently small $\epsilon>0$. The adjunction formula gives $K_{\tilde{X}} =\pi^* K_X + \sum_{i} a_i E_i + \sum_{j} F_j$, where $E_i$ and $F_j$ are irreducible components of the exceptional locus with $a_i\geq 0$ and $b_j>-1$.
Note that $\pi^* \Omega$ vanishes only on each $E_i$ to order $a_i$ and $\pi^*\Omega$ has poles along those $F_i$ with $b_i$.  Then $\pi^*\omega_0$, $\pi^*\chi$ and $\pi^*\Omega$ satisfy {\bf Condition A} and {\bf Condition B}. Furthermore, $\pi^*\varphi_0 \in PSH_p(\tilde{X}, \pi^* \omega_0, \pi^*\Omega)$ and so the assumptions in Theorem \ref{allinone} are satisfied.   The first part of the theorem is then an immediate corollary of Theorem \ref{allinone}.

 The singular set $\tilde{E}$ can be chosen to  be contained in the exceptional locus $Exc(\pi)$ of $\pi$, since $X$ is Q-factorial. Also $\tilde{\varphi}$ must be constant along each component of
$Exc(\pi)$ as $[\pi^*\omega_t]$ is trivial along each component of the exceptional divisors. So it descends to a function in $PSH(X, \omega_t)$ on $X$.

\qed
\end{proof}



\begin{theorem} \label{kunpao} Let $X$ be a $\mathbf{Q}$-factorial projective variety  with log terminal singularities and $H$ be an ample $\mathbf{Q}$-divisor on $X$. Let $$T_0 = \sup\{ t>0 ~|~ H + t K_X ~is~ nef~\}.$$
If $\omega_0 \in \mathcal{K}_{H,p}(X)$ for some $p>1$, then there exists a unique solution $\omega$ of the unnormalized weak K\"ahler-Ricci flow for $t\in [0, T_0)$.

Furthermore, if $\Omega$ is a smooth volume form on $X$, then for any $T\in (0, T_0)$, there exists a constant $C>0$ such that on $[0, T]\times X$,

\begin{equation}
e^{-\frac{C}{t} } \Omega \leq \omega^n \leq e^{\frac{C}{t} }\Omega.
\end{equation}

\end{theorem}

\begin{proof} It suffices to prove the uniqueness as the existence and the volume estimate follow easily from Theorem \ref{liftflow} and Theorem \ref{allinone}.

Let $\omega_t = \omega_0 + t\chi$ and then $\omega = \omega_t + \ddbar \varphi$ with $\varphi\in L^{\infty}(\tilde{X})\cap C^{\infty}(\tilde{X}\setminus \tilde{E} )$.
Then the K\"ahler-Ricci flow is equivalent to the following equation

\begin{equation}\label{krflow4}
\left\{
\begin{array}{rcl}
&&{ \displaystyle \ddbar \left( \ddt{}\varphi - \log \frac{\omega^n}{\Omega}  \right)= 0}, ~~~~{\textnormal on }~ \tilde{X}\setminus {\tilde{E}} \\
&&\\
&& \varphi(0,\cdot)=0 .
\end{array} \right.
\end{equation}

Let $\displaystyle{F=  \ddt{}\varphi - \log \frac{\omega^n}{\Omega} }$. Then $F\in C^\infty( \tilde{X}\setminus \tilde{E})$ and $\ddbar F= 0 $ on $\tilde{X}\setminus \tilde{E}$.

Since $X$ is $\mathbf{Q}$-factorial, $\pi^* [\omega_0] - \epsilon [Exc(\pi)]$ is ample for $\epsilon>0$ sufficiently small.  So we can choose ${\tilde{E}}$ to be contained in $ Exc(\pi)$. Hence  $F$ descends to $X_{reg}$ and $\ddbar F = 0$ on $X_{reg}$. For each $t\in (0, T_0)$, $F$ is smooth on $X_{reg}$, therefore $F$ is constant on each curve in $X$ which does not intersect $X_{sing}$.  On the other hand, for any two generic points $z$ and $w$ on $X$, there exists a curve joining $z$ and $w$ without intersecting $X_{sing}$ since $codim(X_{sing})\geq 2$. So $F(z)=F(w)$ as $F$ is constant on $C$. Then $F$ is constant on $X_{reg}$ since $F$ is continuous on $X_{reg}$.

By modifying $\varphi$ by a function only in $t$, $\varphi$ would satisfy the Monge-Amp\`ere flow (\ref{krflowlimit}). The theorem follows from the uniqueness of the solution $\varphi$.

The volume estimate also follows from Theorem \ref{allinone}.

\qed
\end{proof}

We immediately have the following long time existence result generalizing the case  for nonsingular minimal models due to Tian-Zhang \cite{TiZha}.

\begin{corollary}\label{kunpao1} Let $X$ be a minimal model  with log terminal singularities and $H$ be an ample $\mathbf{Q}$-divisor on $X$. Then $$T_0 = \sup\{ t>0 ~|~ H + t K_X ~is~ nef~\}=\infty$$
and the unnormalized weak K\"ahler-Ricci flow starting with $\omega_0 \in \mathcal{K}_{H,p}(X)$ for some $p>1$ exists for  $t\in [0, \infty)$.

\end{corollary}


\subsection{K\"ahler-Ricci flow on projective varieties with orbiforld or crepant singularities}  \label{4.3}

Given a normal projective variety $X$, very little is known how to construct "good" K\"ahler metrics on $X$ with reasonable curvature conditions. In general, the restriction of Fubini-Study metrics $\omega_{FS}$ on $X$ from ambient projective spaces behaves badly near the singularities of $X$. Even the scalar curvature of $\omega_{FS}$ would have to blow up. In particular, $\omega^n$ is not necessarily equivalent to a smooth volume form on $X$. For example, let $X$ be a surface containing a curve $C$ with self-intersection number $-2$ and $Y$ be the surface obtained from $X$ by contracting $C$. Then $Y$ has an isolated orbifold singularity. Let $\omega$ be a smooth K\"ahler metric and $\Omega$ a smooth volume on $Y $. Then $\frac{\omega^n}{\Omega}=0$ at the orbifold singularity. It tells that one should look at the category of smooth orbifold K\"ahler metrics on $Y$ instead of smooth K\"ahler metrics from ambient spaces.

As it turns out, the K\"ahler-Ricci flow produces K\"ahler currents whose Monge-Amp\`ere mass is equivalent to a smooth volume form on singular varieties by Theorem \ref{kunpao}. It is desirable that the K\"ahler-Ricci flow indeed improves the regularity of the initial data. In the case when $X$ has orbifold or crepant singularities, we show that at least the scalar curvature of the K\"ahler currents are bounded. In particular if $X$ has only orbifold singularities, the K\"ahler-Ricci flow immediately smoothes out the initial K\"ahler current.

\begin{theorem}\label{love}
Let $X$ be a $\mathbf{Q}$-factorial projective normal variety with orbifold singularities.  Let $H$ be an ample $\mathbf{Q}$-divisor on $X$ and $$T_0 = \sup\{ t>0 ~|~ H + t K_X ~is~ nef~\}.$$
If $\omega_0 \in \mathcal{K}_{H,p}(X)$ for some $p>1$, then there exists a unique solution $\omega$ of the unnormalized weak K\"ahler-Ricci flow for $t\in [0, T_0)$.

Furthermore, $\omega(t, \cdot)$ is a smooth orbifold K\"ahler-metric on $X$ for all $t>0$ and so the weak K\"aher-Ricci flow becomes the smooth K\"ahler-Ricci flow on $X$ immediately when $t>0$.

\end{theorem}

\begin{proof} $X$ is automatically log terminal under the assumptions in the theorem if it only admits orbifold singularities. The theorem can be proved by the same argument as in Theorem \ref{smoothing1}. We leave the details for the readers as an exercise.

\qed
\end{proof}

Theorem \ref{love} can also be applied to the K\"ahler-Ricci flow on projective manifolds whose initial class is not K\"ahler.

\begin{theorem}\label{smoothman}
Let $X$ be a smooth projective variety. Let $H$ be a big and semi-ample $\mathbf{Q}$-divisor on $X$. Suppose that $$T_0 = \sup\{ t>0 ~|~ H + t K_X ~\textnormal{is~ semi-ample~}\}>0.$$
If $\omega_0 \in \mathcal{K}_{H,p}(X)$ for some $p>1$, then there exists a unique solution $\omega$ of the unnormalized weak K\"ahler-Ricci flow for $t\in [0, T_0)$.

Furthermore, for any $t\in(0, T_0)$, there exists $C(t)>0$ such that the scalar curvature $S(\omega(t, \cdot))$ is bounded by $C(t)$

\begin{equation}
||S(\omega(t, \cdot))||_{L^\infty(X)}\leq C(t).
\end{equation}

\end{theorem}

\begin{proof} Let $\Omega$ be a smooth volume form on $X$ and $\chi = \ddbar\log \Omega$. Let $\vartheta$ be a K\"ahler form on $X$. Suppose that $\omega_0 = \omega_0' + \ddbar\phi$, where $\omega'_0\in [H]$ is a smooth K\"ahler current and $\phi\in PSH_p(X, \omega'_0)$ for some $p>1$. We consider the special case of the Monge-Ampere flow (\ref{degmaflow5}) by letting $\delta=w=r=0$ and $\omega'_{t,s} = \omega_0' + s \vartheta + t\chi $.

\begin{equation}\label{alsoricci}
\ddt{\varphi_s} = \log \frac{(\omega'_{t,s}+\ddbar\varphi_s)^n}{\Omega}, ~~~~~~~~~~\varphi_s |_{t=0}= \varphi_{(0,s)}.
\end{equation}

In fact, equation (\ref{alsoricci}) is equivalent to the unnormalized K\"ahler-Ricci flow on $X$ starting with $\omega_0$. Furthermore, $\varphi_s(t, \cdot)$ is smooth for $t \in (0, T_0)$. Let $\tilde{\omega}_s(t, \cdot) = \omega_{t,s}' + \ddbar \varphi_s$. Then for $t>0$, $$\ddt{\tilde{\omega}_s} = -Ric(\tilde{\omega}_s)$$ and so $$ \ddt{}S(\tilde{\omega}_s) = \Delta_s S(\tilde{\omega}_s) + |Ric(\tilde{\omega}_s)|^2,$$
where $\Delta_s$ is the Laplace operator associated to $\tilde{\omega}_s$.

Since $$(\ddt{}-\Delta_s) t S(\tilde{\omega}_s) = S(\tilde{\omega}_s) + t|Ric(\tilde{\omega}_s)|^2 \geq S(\tilde{\omega}_s) + \frac{t}{n} S(\tilde{\omega}_s)^2 .$$
The maximum principle immediately implies that $t S(\tilde{\omega}_s)$ is bounded from below on $[0, T_0) \times X$ uniformly in $s \in (0, 1]$. By letting $s\rightarrow 0$, $tS(\omega(t, \cdot))$ is uniformly bounded from below on $[0, T_0)\times X$.

Now we will prove the upper bound for $S(\omega)$.

\begin{claim}  For any $0 < t_0< T<  T_0$, there exist $A$ and $B>0$ such that for all $s\in (0, 1]$ and  on $[t_0, T]\times X$,

\begin{equation}
(\ddt{} - \Delta_s) tr_{\tilde{\omega}_s}(\omega'_{0}) \leq A (tr_{\tilde{\omega}_s}(\omega'_{0}))^2 - B |\nabla_s tr_{\tilde{\omega}_s}(\omega'_{0})|^2,
\end{equation}
\begin{equation}
(\ddt{} - \Delta_s) tr_{\tilde{\omega}_s}(\omega'_{0}+T\chi) \leq A (tr_{\tilde{\omega}_s}(\omega'_{0}+T\chi))^2 - B |\nabla_s tr_{\tilde{\omega}_s}(\omega'_{0}+T\chi)|^2,
\end{equation}
where $\nabla_s$ is the gradient operator associated to $\tilde{\omega}_s$.

\end{claim}

\begin{claim} For any $0< t_0 <T< T_0$, There exists $C>0$ such that for all $s\in (0, 1]$ and on $[t_0, T)\times X$,

\begin{equation}
0\leq tr_{\tilde{\omega}_s}(\omega'_{t,0}) <C.
\end{equation}

\end{claim}

In particular, there exists $C>0$ such that \begin{equation} 0\leq tr_{\tilde{\omega}_s}(\omega'_0) <C, ~~~~-C <tr_{\tilde{\omega}_s}( \chi) <C.\end{equation}

Straightforward calculations show that

\begin{eqnarray}
&&(\ddt{} - \Delta_s) (|\nabla_s \ddt{\varphi_s}|^2)  \nonumber \\
&=& - |\nabla_s\nabla_s \ddt{\varphi_s}|^2 - |\overline{\nabla}_s \nabla_s \ddt{\varphi_s}|^2 + (\nabla_s tr_{\tilde{\omega}_s}(\chi) \cdot \overline{\nabla}_s \ddt{\varphi_s} + \nabla_s \ddt{\varphi_s} \cdot \overline{\nabla}_s tr_{\tilde{\omega}_s}(\chi))
\end{eqnarray}
and
\begin{equation}
(\ddt{} - \Delta_s) \Delta_s \ddt{\varphi_s} = - |\overline{\nabla}_s \nabla_s \ddt{\varphi_s}|^2 - g^{i\bar{l}}g^{k\bar{j}}\chi_{i\bar{j}} (\ddt{\varphi_s})_{k\bar{l}}
\end{equation}

Notice that $\ddt{\varphi_s}$ is uniformly bounded on $[t_0, T]$ for any $0< t_0 < T < T_0$ and $s\in (0,1]$.

Then similar argument in the proof of Theorem 5.1 can be applied. Namely, one can apply the maximum principle for $(t-t_0)  \mathcal{H}$ and $(t_0)  \mathcal{K}$, where

$$\mathcal{H} = \frac{|\nabla_s \ddt{\varphi_s}|^2} {A - \ddt{\varphi_s}} + tr_{\tilde{\omega}_s}(\omega'_0)+ tr_{\tilde{\omega}_s}(\omega'_0+ T\chi)$$

and $$ \mathcal{K} = - \frac{\Delta_s \ddt{\varphi_s}}{A-u} + B  \mathcal{H}.$$

If we choose $A>0$ sufficiently large,
\begin{eqnarray*}
&&(\ddt{} - \Delta_s)( t-t_0) \mathcal{H} \\
&\leq& - \epsilon(t-t_0) \frac{| \ddt{\varphi_s} |^4}{(A- \ddt{\varphi_s})^3} - \frac{2(1-\epsilon)(t-t_0)}{A-u} Re (\nabla_s  \mathcal{ H} \cdot \overline{\nabla}_s \ddt{\varphi_s}) +  C_1\mathcal{H} + C_1\\
&\leq& -\epsilon C_2 (t-t_0)  \mathcal{H}^2 +  C_1 \mathcal{H} -\frac{2(1-\epsilon)(t-t_0)}{A-u} Re (\nabla_s  \mathcal{H} \cdot \overline{\nabla}_s \ddt{\varphi_s}) +C_3.
\end{eqnarray*}
Hence $(t-t_0) \mathcal{H}$ is uniformly bounded on $[t_0, T]\times X$ for any $s\in (0,1]$.

If we choose $A$ and $B>0$ sufficiently large,
\begin{eqnarray*}
&&(\ddt{} - \Delta_s)( t-t_0)  \mathcal{K} \\
&\leq& - C_4(t-t_0) \frac{| \nabla_s\overline{\nabla}_s\ddt{\varphi_s} |^2}{A- \ddt{\varphi_s}} - \frac{2(t-t_0)}{A-u} Re (\nabla_s  \mathcal{K} \cdot \overline{\nabla}_s \ddt{\varphi_s}) +  C_1\mathcal{K} + C_1\\
&\leq& - C5 (t-t_0)  \mathcal{K}^2 +  C_1 \mathcal{K} - \frac{2(t-t_0)}{A-u} Re (\nabla_s  \mathcal{K} \cdot \overline{\nabla}_s \ddt{\varphi_s}) +C_6 .
\end{eqnarray*}
Hence $(t-t_0)K$ is uniformly bounded on $[t_0, T]\times X$ for any $s\in (0,1]$.  Here we make use of Claim 1 that $$T~ tr_{\tilde{\omega}_s}(\chi) = tr_{\tilde{\omega}_s} (\omega'_0 + T\chi) - tr_{\tilde{\omega}_s}(\omega'_0)$$ is uniformly bounded on $[t_0, T]\times X$ uniformly for $s\in (0, 1]$. Also the term $$T^2 |\nabla_s tr_{\tilde{\omega}_s}(\chi)|^2 \leq |\nabla_s tr_{\tilde{\omega}_s}(\omega'_0 + T\chi)|^2 + |\nabla_s tr_{\tilde{\omega}_s}(\omega'_0)|^2$$ can be controlled by $(\ddt{} - \Delta_s) ( tr_{\tilde{\omega}_s}(\omega'_0)+ tr_{\tilde{\omega}_s}(\omega'_0+ T\chi))$.

Therefore there exists $C>0$ such that on $[t_0, T]\times X$
$$ S(\tilde{\omega}_s ) = - \Delta_s \ddt{\varphi_s} - tr_{\tilde{\omega}_s}(\chi) \leq C $$ uniformly in $s\in (0,1]$.

The theorem is then proved by letting $s\rightarrow 0$.

\qed
\end{proof}

Now we shall prove the two claims in the proof of Theorem \ref{smoothman}.

\medskip

\noindent{\bf Proof of Claim 1} Without loss of generality, we let $\pi: X \rightarrow \mathbf{CP}^{N_m}$ be the morphism induced by $mH$ and $m \omega'_0$ is the pullback of the Fubini-Study metric on $\mathbf{CP}^{N_m}$ if $m$ is sufficiently large. Notice that for $t\in [t_0, T]$, $H+ TK_X$ is still semi-ample and big, so $m'(H+TK)$ induces a morphism $\pi': X \rightarrow \mathbf{CP}^{N_{m'}}$. We can again assume that $\omega_0 + T\chi$ is the pullback of the Fubini-Study metric on $\mathbf{CP}^{N_{m'}}$.  The curvature of $\omega'_0$ on $\mathbf{CP}^{N_m}$ and the curvature of $\omega'_0 + T\chi$ on $\mathbf{CP}^{N_{m'}}$ are both bounded. Then it becomes a straightforward calculation from \cite{SoT1}.
\qed

\medskip

\noindent{\bf Proof of Claim 2} This can be proved by the parabolic Schwarz lemma from \cite{SoT1}. We apply the maximum principle for $t\log tr_{\tilde{\omega}_s}(\omega'_0) -A \varphi_s$ and $t\log tr_{\tilde{\omega}_s}(\omega'_0 + T\chi) -A \varphi_s$ for sufficiently large $A$ so that both terms are uniformly bounded on $[0,T]\times X$ uniformly for $s\in (0,1]$. The claim then easily follows.

\qed

\medskip

Theorem \ref{smoothman} shows that the K\"ahler-Ricci flow can be defined even if the initial smooth K\"ahler current is in the semi-ample cone of divisors.

The following theorem is an immediate corollary of Theorem \ref{smoothman}.
\begin{theorem}
Let $X$ be a $\mathbf{Q}$-factorial projective variety  with crepant singularities.  Let $H$ be an ample $\mathbf{Q}$-divisor on $X$ and $$T_0 = \sup\{ t>0 ~|~ H + t K_X ~is~ nef~\}.$$
If $\omega_0 \in \mathcal{K}_{H,p}(X)$ for some $p>1$, then there exists a unique solution $\omega$ of the unnormalized weak K\"ahler-Ricci flow for $t\in [0, T_0)$.

Furthermore, for any $t\in(0, T_0)$, there exists $C(t)>0$ such that the scalar curvature $S(\omega(t, \cdot))$ is bounded by $C(t)$

\begin{equation}
||S(\omega(t, \cdot))||_{L^\infty(X)}\leq C(t).
\end{equation}

\end{theorem}

\begin{proof} Let $\Omega$ be a smooth volume form on $X$. Let $\pi: \tilde{X} \rightarrow X$ be a crepant resolution of $X$. Then $\pi^* \Omega$ is again a smooth volume form on $\tilde{X}$. Then we can apply Theorem \ref{smoothman}.

We remark that it might be interesting to remove the dependence on the $t$ for the scalar curvature bound if $T_0=\infty$.

\qed
\end{proof}

It is shown in \cite{Z2} that the scalar curvature is uniformly bounded along the normalized K\"ahler-Ricci flow on smooth manifolds of general type. On the other hand, the scalar will in general blow up if the K\"ahler-Ricci flow develops finite time singularities (see \cite{Z3}).


\section{K\"ahler-Ricci flow with surgery} \label{5}

\subsection{Minimal Model Program with Scaling} \label{5.1}

\begin{definition} Let $X$ be a projective variety and $N_1(X)_{\mathbf{Z}}$ the group of numerically equivalent $1$-cycles (two $1$-cycles are numerically equivalent if they have the same intersection number with every Cartier divisor). Let $N_1(X)_{\mathbf{R}} = N_1(X)_{\mathbf{Z}} \otimes _{\mathbf{Z}} \mathbf{R}$. We denote by $NE(X)$ the set of classes of effective $1$-cycles. $NE(X)$ is convex and we let $\overline{NE}(X)$ be the closure of $NE(X)$ in the Euclidean topology.
\end{definition}

A special case of the Minimal Model Program is proposed in \cite{BCHM} and plays an important role for the termination of flips. We briefly explain the Minimal Model Program with Scaling below.

\begin{definition}[MMP with scaling] \label{mmps}~~

\begin{enumerate}

\item We start with a pair $(X, H)$, where $X$ is  a projective $\mathbf{Q}$-factorial variety $X$ with log terminal singularities  and  $H$ is a big and semi-ample $\mathbf{Q}$-divisor on $X$.

\item Let $\lambda_0 = \inf \{ \lambda>0~|~ \lambda H+K_X~is ~nef \}$ be the nef threshold. If $\lambda_0 =0$, then we stop since $K_X$ is already nef.

\item Otherwise, there is an extremal ray $R$ of the cone  of curves $\overline{NE}(X)$ on which $K_X$ is negative and $\lambda_0 H + K_X $ is zero. So there exists a contraction $\pi: X \rightarrow Y$ of $R$.

\begin{itemize}
\item If $\pi$ is a divisorial contraction, we replace $X$ by $Y$ and let $H_Y$ be the strict transformation of $\lambda_0 H + K_X$ by $\pi$. Then we return to $1.$ with $(Y, H_Y)$.

\item If $\pi$ is a small contraction, we replace $X$ by  its flip $X^+$ and let $H_{X^+}$ be the strict transformation of $\lambda_0 H+ K_X$ by the flip. Then we return to $1.$ with $(X^+, H_{X^+})$.

\item If $\dim Y < \dim X$, then $X$ is a Mori fibre space, i.e., the fibers of $\pi$ are Fano. Then we stop.

\end{itemize}

\end{enumerate}

\end{definition}

The following theorem is proved in \cite{BCHM}.
\begin{theorem} \label{finiteterm} If $X$ is of general type, the Minimal Model Program with Scaling terminates in finite steps.

\end{theorem}

In general, the contraction of the extremal ray might not be the same as the contraction induced by the semi-ample divisor $\lambda_0 H + K_X$. We define the following special ample divisors so that at each step, there is only one extremal ray contracted by the morphism induced by $\lambda_0 H + K_X$.
\begin{definition}\label{goodinitial} Let $X$ be a projective $\mathbf{Q}$-factorial variety with log terminal singularities. An ample  $\mathbf{Q}$-divisor $H$ on $X$ is called a good initial divisor $H$ if the following conditions are satisified.

\begin{enumerate}
\item
Let $X_0=X$ and $H_0=H$. The MMP with scaling terminates in finite steps by replacing $(X_0, H_0)$ by  $(X_1, H_1)$, ..., $(X_m, H_m)$ until $X_{m+1}$ is a minimal model or $X_m$ is a Mori fibre space.

\item Let $\lambda_i$ be the nef threshold for each pair $(X_i, H_i)$ for $i=1, ..., m$. Then the contraction induced by the semi-ample divisor $\lambda_i H_i + K_{X_i}$ contracts exactly one extremal ray.

\end{enumerate}

\end{definition}

It might be possible that good initial divisors are generic if MMP with scaling holds for any pair $(X,H)$.  It will be seen in the future that good initial divisors simplify the analysis for surgery along the K\"ahler-Ricci flow, though such an assumption is not necessary. We will explain it in detail in Section \ref{5.5}.

Now we relate the K\"ahler-Ricci flow to MMP with scaling. Consider the unnormalized K\"ahler-Ricci flow $ \displaystyle{ \ddt{\omega} = - Ric(\omega)}$ on $X$ with the initial K\"ahler current $\omega_0\in [H]$ for an ample divisor $H$ on $X$.
Let $T_0= \sup\{ t\geq 0 ~|~  H + t K_X >0 \}$. By the rationality theorem \ref{rationality}, $T_0 = \infty$ or $T_0$ is a positive rational number. In particular, if $X$ is a minimal model, then $T_0= \infty$. In fact, $T_0 = \frac{1}{\lambda_0}$ is the inverse of the nef threshold. The following theorem is a natural generalization for the long time existence theorem of Tian and Zhang \cite{TiZha} for the K\"ahler-Ricci flow on smooth minimal models.

\begin{theorem}\label{longtime}
Let $X$ be an $n$-dimensional $\mathbf{Q}$-factorial projective variety with log terminal singularities with nef $K_X$. For any ample $\mathbf{Q}$-divisor $H$ on $X$ and $\omega_0 \in \mathcal{K}_{H, p}(X)$ with $p>1$, there unnormalized weak K\"ahler-Ricci flow starting with $\omega_0$, exists for $t\in [0, \infty)$.

\end{theorem}

Suppose that $X$ is not minimal and so $T_0 <\infty$. Then $H+T_0 K_X$ is nef and the weak K\"ahler-Ricci exists uniquely for $t\in [0, T_0)$. By Kawamata's base point free theorem, $H+ T_0 K_X$ is semi-ample and hence the ring $R(X, H+ T_0 K_X) = \oplus_{m=0}^\infty H^0(X, m (H+ T_0 K_X)$ is finitely generated.

\begin{itemize}

\item

If $H+ T_0 K_X$ is big and hence $R(X, H+T_0 K_X)$ induces a birational morphism $\pi:  X \rightarrow Y$. For a generic ample divisor $H$, the morphism $\pi$ contracts exactly one extremal ray of $\overline{NE}(X)$.   We discuss the following two cases according to the size of the exceptional locus of $\pi$.
\begin{enumerate}

\item $\pi$ is a divisorial contraction, that is, the exceptional locus $Exc(\pi)$ is a divisor whose image of $\pi$ has codimension at least two. In this case, $Y$ is still $\mathbf{Q}$-factorial and has at worst log terminal singularities.

\item $\pi$ is a small contraction, that is,  the exceptional locus $Exc(\pi)$  has codimension at least two. In this case, $Y$ have rather bad singularities and $K_Y$ is no longer a Cartier $\mathbf{Q}$-divisor. The solution to such a small contraction is to replace $X$ by a birationally equivalent variety with singularities milder than those of $Y$.

    \begin{definition}\label{flip}( see \cite{KMM}) Let $\pi: X\rightarrow Y$ be a small contraction such that $-K_X$ is $\pi$-ample. A variety $X^+$ together with a proper birational morphism $\pi^+: X^+ \rightarrow Y$ is called a flip of $\pi$ if $\pi^+$ is also a small contraction and $K_{X^+}$ is $\pi$-ample.
 \begin{equation}
\begin{diagram}\label{diag2}
\node{X} \arrow{se,b,}{\pi}  \arrow[2]{e,t,..}{(\pi^+)^{-1}\circ\pi }     \node[2]{X^+} \arrow{sw,r}{\pi^+} \\
\node[2]{Y}
\end{diagram}
\end{equation}

    \end{definition}

Here $X^+$ is again $\mathbf{Q}$-factorial and has at worst log terminal singularities.
\end{enumerate}

\item If $H+T_0 K_X$ is not big, then the Kodaira dimension $0\leq \kappa= \kod(H+T_0 K_X) <n$ and $X$ is a Mori fibre space admitting a Fano fibration over a normal variety $Y$ of dimension $\kappa$. In particular, $Y$ is $\mathbf{Q}$-factorial and has log terminal singularities.

\end{itemize}

We will discuss in the following sections the behavior of the K\"ahler-Ricci flow at the singular time $T_0$ according to the above situations.


\subsection{Estimates}  \label{5.2}

In this section, we assume that $T_0 < \infty$ and $H+T_0K_X$ is big.

Let $\Omega$ be a smooth volume form on $ X$ and $\chi= \ddbar\log \Omega  \in [K_X]$. Consider the Monge-Amp\`ere flow associated to the unnormalized K\"ahler-Ricci flow on $X$ with the initial K\"ahler form $\omega_0$,

\begin{equation}\label{maflow_deglim2}
\left\{
\begin{array}{rcl}
&&{ \displaystyle \ddt{\varphi} = \log
 \frac{ ( \omega_t + \ddbar \varphi ) ^n } {\Omega} },  ~~~~~[0, T_0)\times X\\
&&\\
&& \varphi(0,\cdot)=\varphi_0 \in PSH_p(X, \omega_0, \Omega)
\end{array} \right.
\end{equation}
where $p>1$, $\omega_t = \omega_0 + t \chi$ and $\omega = \omega_t + \ddbar\varphi$.

Since $H+ T_0 K_X$ is big and semi-ample, the linear system $|m(H+T_0K_X)|$ for sufficiently large $m$ induces a morphism

$$\pi: X \rightarrow Y \subset \mathbf{CP}^{N_m}.$$

Let $\omega_Y$ be the pullback of a multiple of the Fubini-Study metric form on $\mathbf{CP}^{N_m}$ with $\omega_Y \in [H+ T_0K_X]$.
There exists a resolution of singularities and the exceptional locus of $\pi$

$$\mu: \tilde{X} \rightarrow Y$$ satisfying the following conditions.

\begin{enumerate}

\item $\tilde{X}$ is smooth.

\item  There exists
an effective divisor $E_Y$ on $\tilde{X}$ such that $\mu^*[H+ T_0 K_X] - \epsilon [E_Y]$ is ample for any sufficiently small $\epsilon >0$ and the support of $E_Y$ coincides with the exceptional locus of $\mu$.

\end{enumerate}

Let $S_{E_Y}$ be the defining section for the line bundle associated to $[E_Y]$ and $h_{E_Y}$ the hermitian metric such that  for any $\epsilon>0$, $$\mu^* \omega_Y - \epsilon Ric(h_{E_Y})>0.$$
Let $Exc(\pi)$ be the exceptional locus of $\pi$. Then we have the following uniform estimates.

\begin{theorem}\label{tianzhang1} Let $\varphi\in C^0([0,T_0)\times X_{reg})\cap C^\infty((0,T_0)\times X_{reg})$ with $\varphi(t,\cdot)\in L^\infty(X)$ for all $t\in [0, T_0)$,  be the solution solving the Monge-Amp\`ere flow (\ref{maflow_deglim2}). There exists a constant $C>0$ such that
\begin{equation}
||\varphi||_{L^{\infty} ([0, T_0) \times X)} \leq C.
\end{equation}

Furthermore, for any $K \subset\subset X \setminus Exc(\pi) $ and $k\geq 0$, there exists $C_{K, k}>0$ such that

\begin{equation}
||\varphi||_{C^k([0, T_0)\times K } \leq C_{K,k}.
\end{equation}

\end{theorem}

\begin{proof} We lift the Monge-Amp\`ere flow (\ref{maflow_deglim2}) on $\tilde{X}$. The proof of the $L^\infty$-estimate proceeds in the same way as in the proof of Lemma \ref{c0estimate1} since $[\omega_t]$ is big and semi-ample for all $t\in [0, T_0]$. The $C^2$-estimate on $\tilde{X}$ follows the same argument as in Lemma \ref{c2estimate1}, which is valid on $\tilde{X}\setminus E_Y$. Since the support of $\mu(E_Y)$ is contained in $Exc(\pi)$, the $C^2$-estimate holds on $X_{reg}\setminus Exc(\pi)$.  We leave the details for the readers as an exercise.

\qed
\end{proof}

\begin{theorem} There exists $C>0$ such that on $[0, T_0)\times X$,

\begin{equation}
      \frac{\omega^n}{\Omega}\leq e^{\frac{C}{t}}.
      \end{equation}

\end{theorem}

\begin{proof} We consider $H= t\dot{\varphi} - \varphi + \epsilon \log|S_{E_Y}|^2_{h_Y}$ on $\tilde{X}$. Let $\Delta$ be the Laplace operator with respect to the pullback of $\omega$. Then $H$ is smooth outside $E_Y$ and
$$(\ddt{}- \Delta) H = - tr_{\omega} (\omega_0 - \epsilon Ric(h_{E_Y})) \leq 0.$$
As $H|_{t=0} = -\varphi_0 + \epsilon\log|S_{E_Y}|^2_{h_Y}$ is bounded from above and for each $t\in (0, T_0)$, the maximum of $H$ can only be achieved on $\tilde{X}\setminus E_Y$, $H\leq H|_{t=0}$ is uniformly bounded above and by letting $\epsilon \rightarrow 0$, there exists $C>0$ such that  $$t\dot{\varphi} \leq C.$$ We are done.

\qed\end{proof}

\begin{corollary}\label{descend} We consider the unique solution $\omega$ for the unnormalized weak K\"ahler-Ricci flow on $X$ starting with an initial current in $\mathcal{K}_{H, p}(X)$ for some $p>1$. If $H+ T_0 K_X$ is big, then $\omega(t, \cdot)$ converges to a K\"ahler current $\tilde{\omega}_{T_0} \in \mathcal{K}_{H+T_0 K_X, \infty}(X)$ in $C^\infty(X_{reg}\setminus Exc(\pi))$-topology. That is, there exists $C>0$ such that
\begin{equation}
\tilde{\omega}_{T_0}^n \leq C \Omega
\end{equation}
for a fixed smooth volume form $\Omega$ on $X$.

\end{corollary}

By Corollary \ref{descend},  $H$ is trivial over the fibres and so $\omega_{Y}$ is trivial restricted on each fibre. Let $\tilde{\omega}_{T_0} =\omega_{Y} + \ddbar\varphi_{T_0}$, then $\varphi_{T_0}$ must be constant on each fibre as any fibre of $\pi$ are connected. Therefore $\varphi_{T_0}$ can descend onto $Y$ and $\varphi_{T_0}\in PSH(Y, \omega_{Y})\cap L^\infty(Y)$. So the limiting K\"ahler current $\tilde{\omega}_{T_0}$ descends onto $Y$ as a semi-positive closed $(1,1)$-current.


\subsection{Extending K\"ahler-Ricci flow through singularities by divisorial contractions } \label{5.3}

In this section, we will prove that the weak K\"ahler-Ricci flow can be continued through divisorial contractions.

We assume that $\pi: X \rightarrow Y$ is a divisorial contraction and the fibres of $\pi$ are connected. It is well-known that $Y$ is again a $\mathbf{Q}$-factorial projective variety with at worst log terminal singularities if $X$ is.

\begin{proposition}\label{push1} Let $\Omega_Y$ be a smooth volume form on $Y$ and $H_Y=\pi_* (H+T_0K_X)$. Then for some $p>1$,

\begin{equation}
(\pi^{-1})^*\tilde{\omega}_{T_0} \in \mathcal{K}_{H_Y, p}(Y).
\end{equation}

\end{proposition}

\begin{proof} Obviously, $\tilde{\omega}_{T_0}$ has bounded local potentia and the restriction of $\tilde{\omega}$ has constant local potential along each fibre of $\pi$. So  $\tilde{\omega}_{T_0}$ descends to $Y$ and $(\pi^{-1})^*\tilde{\omega}_{T_0}$ is well-defined and admits bounded local potential on $Y$. Let $F  = \frac{ ( \tilde{\omega}_{T_0})^n}{\Omega_Y }$. It suffices to show $F \in L^p(Y, \Omega_Y) $ for some $p>1$. There exists $C>0$ such that

$$ \int_{Y} F^{p} \Omega_Y = \int_Y \left( \frac{\tilde{\omega}_{T_0}^n }{ \Omega_Y } \right)^{p-1} \tilde{\omega}_{T_0}^n = \int_X \left( \frac{\tilde{\omega}_{T_0}^n }{ \pi^* \Omega_Y } \right)^{p-1} \tilde{\omega}_{T_0}^n \leq C \int_X \left( \frac{\Omega }{ \Omega_Y } \right)^{p-1} \Omega.$$

Since  $\frac{\Omega }{ \Omega_Y } $ has at worst poles, $ \int_{Y} F^{p} \Omega_Y < \infty$ for $p-1>0$ sufficiently small.

\qed\end{proof}

\begin{theorem} \label{thrudivcon}Let $X$ be a $\mathbf{Q}$-factorial projective variety  with log terminal singularities and $H$ be an ample $\mathbf{Q}$-divisor on $X$. Let $$T_0 = \sup\{ t>0 ~|~ H + t K_X ~is~ nef~\}$$ be the first singular time.  Suppose that the semi-ample divisor $H+T_0K_X$ induces a divisorial contraction $\pi: X \rightarrow Y$.

Let $\omega$ be the unique solution  of the unnormalized weak K\"ahler-Ricci flow for $t\in [0, T_0)$ starting with $\omega_0 \in \mathcal{K}_{H, p}(X)$ for some $p>1$.
Then there exists $\omega_{Y,0} \in \mathcal{K}_{\pi_*H_Y, p'}(Y) \cap C^\infty(Y_{reg} \setminus \pi(Exc(\pi)) $ for some $p'>1$ such that $\omega(t, \cdot)$ converges to $\pi^* \omega_{Y,0}$ in $C^\infty(X_{reg}\setminus Exc(\pi))$-topology as $t\rightarrow T_0$.

Furthermore, the unnormalized weak K\"ahler-Ricci flow can be continued on $Y$ with the initial K\"ahler current $\omega_{Y,0}$.

\end{theorem}

\begin{proof}

Since $H_Y$ is the strict transformation of $H$ by $\pi$ and $\omega_{Y, 0}$ admits bounded local potential, $\omega_{Y, 0}\in \mathcal{K}_{H_Y, p'}(X)$ for some $p'>1$ by Proposition \ref{push1}. Then the K\"ahler-Ricci flow can start with $\omega_{Y, 0}$ on $Y$ uniquely as $H_Y$ is ample.

\qed
\end{proof}


\subsection{Extending K\"ahler-Ricci flow through singularities by flips } \label{5.4}

In this section, we will prove that the weak K\"ahler-Ricci can be continued through flips.

We assume that $\pi: X \rightarrow Y$ is a small contraction and there exists a flip $$\check{\pi} = \pi^+ \circ \pi^{-1} : X^+ \dashrightarrow X . $$
Then $X^+$ is $\mathbf{Q}$-factorial and it has at worst log terminal singularities.
The limiting K\"ahler current $\tilde{\omega}_{T_0}$ descends on $Y$ and it can be then pulled back on $X^+$ by $\pi^+$. Furthermore, there exists $C>0$ such that

\begin{equation}
\frac{ (\check{\pi}^* \tilde{\omega}_{T_0})^n }{\check{\pi} ^* \Omega} \leq C.
\end{equation}

\begin{proposition}\label{push2} Let $\Omega_{X^+}$ be a smooth volume form on $X^+$ and $H_{X^+}$ be the strict transformation of $H+T_0K_X$ by $\check{\pi}$. Then for some $p>1$,

\begin{equation}
\check{\pi}^* \tilde{\omega}_{T_0} \in \mathcal{K}_{H_{X^+}, p}(X^+).
\end{equation}

\end{proposition}

\begin{proof} $(\pi^{-1})^* \tilde{\omega}_{T_0}$ is well-defined on $Y$ with bounded local potential. Then $\check{\pi}^*\tilde{\omega}_{T_0} = ((\pi\circ\pi^+)^{-1})^* \tilde{\omega}_{T_0}$ is semi-positive closed $(1,1)$ current on $X^+$ with bounded local potential as well.   Let $F  = \frac{ ( \check{\pi}^*\tilde{\omega}_{T_0})^n}{\Omega_{X^+} }$. It suffices to show $F \in L^p(X^+, \Omega_{X^+}) $ for some $p>1$. There exists $C>0$ such that
\begin{eqnarray*}
 \int_{X^+} F^{p} \Omega_{X^+}& =& \int_{X^+} \left( \frac{(\check{\pi}^* \tilde{\omega}_{T_0})^n }{ \Omega_{X^+} } \right)^{p-1} (\check{\pi}^*\tilde{\omega}_{T_0})^n \\
 &=& \int_X \left( \frac{\tilde{\omega}_{T_0}^n }{ (\check{\pi}^{-1}) ^* \Omega_{X^+} } \right)^{p-1} \tilde{\omega}_{T_0}^n \\
 &\leq&  C \int_X \left( \frac{\Omega }{ (\check{\pi}^{-1}) ^*\Omega_{X^+} } \right)^{p-1} \Omega.
 \end{eqnarray*}

Since  $\frac{\Omega }{ (\check{\pi}^{-1}) ^*\Omega_{X^+} } $ has at worst poles, $ \int_{X^+} F^{p} \Omega_{X^+} < \infty$ for $p-1>0$ sufficiently small.

\qed\end{proof}

\begin{theorem} \label{thruflip} Let $X$ be a $\mathbf{Q}$-factorial projective variety  with log terminal singularities and $H$ be an ample $\mathbf{Q}$-divisor on $X$. Let $$T_0 = \sup\{ t>0 ~|~ H + t K_X ~is~ nef~\}$$ be the first singular time.   Suppose that the semi-ample divisor $H+T_0K_X$ induces a small contraction $\pi: X \rightarrow Y$ and there exists a flip \begin{equation}
\begin{diagram}\label{diag3}
\node{X} \arrow{se,b,}{\pi}  \arrow[2]{e,t,..}{\check{\pi}^{-1} }     \node[2]{X^+} \arrow{sw,r}{\pi^+} \\
\node[2]{Y}
\end{diagram} .
\end{equation}

Let $\omega$ be the unique solution  of the unnormalized weak K\"ahler-Ricci flow for $t\in [0, T_0)$ starting with $\omega_0 \in \mathcal{K}_{H, p}(X)$ for some $p>1$.
Then there exists $\omega_{X^+,0} \in \mathcal{K}_{H_{X^+}, p'}(X^+)$ such that $\omega(t, \cdot)$ converges to $(\check{\pi}^{-1})^* \omega_{X^+,0}$ in $C^\infty(X_{reg}\setminus Exc(\pi))$-topology, where $H_{X^+}$ is the strict transformation of $H$ by $\check{\pi}$.

Furthermore, $\omega_{X^+, 0}$ is smooth outside the singularities of $X^+$ and where the flip is performed, and the unnormalized weak K\"ahler-Ricci flow can be continued on $X^+$ with the initial K\"ahler current $\omega_{X^+,0}$.

\end{theorem}

\begin{proof} Since $H^+$ is the strict transformation of $H$ by $\check{\pi}$ and $\omega_{X^+, 0}$ admits bounded local potential, $\omega_{X^+, 0}\in \mathcal{K}_{H^+, p'}$ for some $p>1$ by Proposition \ref{push2}. Then the K\"ahler-Ricci flow can start with $\omega_{X^+, 0}$ on $X^+$ uniquely as $H^+$ is big and semi-ample, and $H^+ + \epsilon K_{X^+}$ is ample for sufficiently small $\epsilon>0$.

\qed\end{proof}


\subsection{Long time existence assuming MMP} \label{5.5}

As proved in Section \ref{5.3} and \ref{5.4}, the K\"ahler-Ricci flow can flow through divisorial contractions and flips.  If the exceptional loci of the contracted extremal rays do not meet each other, Theorem \ref{thrudivcon} and \ref{thruflip} still hold.  However, at the singular time $T_0$, the morphism $\pi: X \rightarrow Y$ induced by the semi-ample divisor $H+T_0K_X$ might contract more than one extremal ray. It simplifies the analysis to assume the existence of a good initial divisor as in Definition \ref{goodinitial} as to avoid complicated contractions.

\begin{theorem}\label{longtime1} Let $X$ be a $\mathbf{Q}$-factorial projective variety  with log terminal singularities. If there exists a good initial divisor $H$ on $X$, then either $X$ does not admit a minimal model or the unnormalized weak K\"ahler-Ricci flow has long time existence for any K\"ahler current $\omega_0\in \mathcal{K}_{H, p}(X)$ with $p>1$, after finitely many surgeries through divisorial contractions and flips.

\end{theorem}

\begin{proof} Assume $X$ admits a minimal model and let $X_0=X$ and $H_0=H$.  Since $H$ is a good initial divisor, by MMP with scaling, at each singular time, the morphism induced by the semi-ample divisor is always a contractional contraction or flipping contraction.

More precisely, suppose the K\"ahler-Ricci flow performs surgeries and replaces $(X_0, H_0)$ by a finite sequence of $(X_i, H_i)$ at each singular time $T_i$, $i = 1, ..., m$ , and $X_{m+1}$ is a minimal model of $X$. If $\lambda_i$ is the nef threshold for $(X_i, H_i)$ as in Definition \ref{mmps}, $i=1, ..., m$, $\lambda_i >0$ and $$ T_i = T_{i-1}+ \frac{1}{\lambda_i} .$$
At $T_i$, the morphism induced by the semi-ample divisor $H_i + T_i K_{X_i}$ contracts exactly one extremal ray and so it must be a divisorial contraction or a flip. By Theorem \ref{thrudivcon} and Theorem \ref{thruflip}, the K\"ahler-Ricci flow with the pair $(X_i, H_i)$ is replaced by the one with the pair $(X_{i+1}, H_{i+1})$ with $H_{i+1}$ being the strict transform of $H_i + T_i K_{X_i}$ until $X$ is finally replaced by its minimal model $X_{m+1}$ and the K\"ahler-Ricci flow exists for all time afterwards by Theorem \ref{kunpao} as $K_{X_{m+1}}$ is nef.

\qed
\end{proof}

If $H$ is not a good initial divisor, the surgery at the finite singular time could be complicated and a detailed speculation is given in Section \ref{6.2}.


\subsection{Convergence on projective varieties of general type} \label{5.6}

Let $X$ be a minimal model of general type with log terminal singularities and so $K_X$ is big and nef.  Let $H$ be an ample $\mathbf{Q}$-divisor on $X$ and $\omega_0 \in \mathcal{K}_{H, p}$. We consider the normalized K\"ahler-Ricci flow on $X$.

\begin{equation}\label{normalflow}
\left\{
\begin{array}{rcl}
&&{ \displaystyle \ddt{ \omega} = -Ric(\omega) - \omega }\\
&&\\
&& \omega|_{t=0} = \omega_0 .
\end{array} \right.
\end{equation}
Let $\Omega$ be a smooth volume form on $X$ and $\chi \in \ddbar \log \Omega \in c_1(K_X)$.
Then the K\"ahler-Ricci flow (\ref{normalflow}) is equivalent to the following Monge-Amp\`ere flow.

\begin{equation}\label{normalmaflow}
\left\{
\begin{array}{rcl}
&&{ \displaystyle \ddt{\varphi} = \log \frac{ (\omega_t + \ddbar\varphi)^n}{\Omega} - \varphi }\\
&&\\
&& \varphi |_{t=0} = \varphi_0,
\end{array} \right.
\end{equation}
where $\omega_t = e^{-t} \omega_0 + (1 - e^{-t}) \chi.$

Now that $K_X$ is semi-ample as $X$ is a minimal model, the abundance conjecture holds for general type. The linear system $|mK_X|$ for sufficiently large $m>0$ induces a morhpism $$\pi: X  \rightarrow X_{can},$$  where $X_{can}$ is the canonical model of $X$. Without loss of generality, we can always assume that $\chi\geq 0 $ and $\chi$ is big. Furthermore, we can assume $\omega_0 \geq \epsilon \chi $ for sufficiently small $\epsilon>0$ since $H$ is ample.

The long time existence is guaranteed by Theorem \ref{kunpao} since $K_X$ is nef and $T_0=\infty$.
\begin{proposition} The weak normalized K\"ahler-Ricci flow (\ref{normalflow}) exists on $[0, \infty)\times X$ for any initial K\"ahler current $\omega\in\mathcal{K}_{H,p}(X)$ with $p>1$.

\end{proposition}

\begin{lemma} There exists $C>0$ such that

\begin{equation}
||\varphi(t, \cdot)||_{L^\infty([0, \infty)\times X) }\leq C.
\end{equation}

\end{lemma}

\begin{proof} Let $\tilde{X}$ be a nonsingular model of $X$. Without loss of generality, we can consider the Monge-Amp\`ere flow (\ref{normalmaflow}) on $\tilde{X}$ by pullback and the following smooth approximation for the Monge-Amp\`ere flow as discussed in Section \ref{3.2}.

\begin{equation}\label{normalmaflowapp}
\left\{
\begin{array}{rcl}
&&{ \displaystyle \ddt{\varphi_{s,w,r}} = \log \frac{ (\omega_{t,s} + \ddbar\varphi_{s,w,r})^n}{\Omega_{w,r}} - \varphi_{s,w,r} }\\
&&\\
&& \varphi_{s,w,r} |_{t=0} = \varphi_{(0,s)},
\end{array} \right.
\end{equation}
where $\omega_{t,s}=\omega_t+ s\vartheta$ and $\Omega_{w,r}$ are defined as in Section \ref{3.2}.

Let $\phi_{s,w,r}\in C^\infty(\tilde{X})$ be the solution of the following Monge-Amp\`ere equation

\begin{equation}
(\chi+s\vartheta + \ddbar \phi_{s,w,r})^n = e^{\phi_{s,w,r}}\Omega_{w,r}.
\end{equation}
There exists $C>0$ such that for all $s, w, r \in (0,1]$, $$||\phi_{s,w,r}||_{L^\infty(\tilde{X})}\leq C.$$

Let $\psi_\epsilon = \varphi_{s,w,r} - \phi_{s,w,r} -\epsilon \log |S_{\tilde{E}}|_{h_{\tilde{E}}}$, where $\tilde{E}$ is a divisor whose support contains the exceptional locus of the resolution of $\tilde{X}$ over $X$ and $\chi - \epsilon Ric(h_{\tilde{E}})>0$ for sufficiently small $\epsilon>0$. Then similar argument by the maximum principle as in Section \ref{3.3} shows that $\psi_{\epsilon}$ is uniformly bounded from below for all $t\in [0, \infty)$ and for all sufficiently small $\epsilon>0$. Then by letting $\epsilon \rightarrow 0$, there exists $C>0$ such that for $t\in [0, \infty)$, $s$, $w$ and $r\in (0,1]$,  $$\varphi_{s,w,r}\geq -C.$$
Therefore $\varphi$ is uniformly bounded from below for all $t\in [0, \infty)$ by its definition.
The uniform upper bound of $\varphi$ can be obtained by similar argument.

\qed\end{proof}

\begin{lemma} Let $X^{\circ} = X_{reg}\setminus Exc(\pi)$. For any $K \subset\subset X^{\circ}$, $t_0$ and $k >0$ , there exists $C_{K, k, t_0}>0$ such that for $t\in [t_0, \infty)$,

\begin{equation}
||\varphi(t,\cdot)||_{ C_{\omega_0} ^k(X) } \leq C_{K,k,t_0}.
\end{equation}

\end{lemma}

\begin{proof} We can assume that $X$ is nonsingular with the cost of $\omega_0$ and $\Omega$ being degenerate. We first have to show that $tr_{\omega_0}(\omega)$ is uniformly bounded on $K$. This is achieved by similar arguments for Lemma \ref{c2estimate1}. Then the higher order estimates follow by standard argument.

\qed
\end{proof}

\begin{lemma} For any $t_0>0$, there exists $C>0$ such that on $[t_0, \infty)\times X$,

\begin{equation}
\ddt{\varphi} \leq C te^{-t}.
\end{equation}

\end{lemma}

\begin{proof} Notice that

$$(\ddt{} - \Delta) \ddt{\varphi} = - e^{-t} tr_{\omega} (\omega_0 -\chi ) - \ddt{\varphi} .$$

Let $H=e^t \ddt{\varphi} - A\varphi + \epsilon \log |S_{\tilde{E}}|_{h_{\tilde{E}}}^2- An t$, where $A>0$ is sufficiently large such that $A\omega_0 \geq \chi$ and $\epsilon>0$ is chosen to be sufficiently small.  Then there exists $C>0$ for all sufficiently small $\epsilon>0$ such that

\begin{eqnarray*}
(\ddt{} - \Delta)   H
&=& - tr_{\omega}( A\omega_t +\omega_0- \chi -\epsilon Rich(h_{\tilde{E}}) ) -A \ddt{\varphi}\\
&\leq& - A \ddt{\varphi}\\
&\leq& - A e^{-t} H -A^2\varphi + A\epsilon \log |S_{\tilde{E}}|_{h_{\tilde{E}}}^2 - Ant\\
&\leq& - A e^{-t} H - C.
\end{eqnarray*}

Since the maximum can only be achieved on $X^\circ$ and $H|_{t_0=0}$ is bounded from above, by the maximum principle, there exists $C>0$ such that on $[t_0, \infty)\times X$, $$ H \leq C (t+1).$$

Therefore there exists $C>0$ independent of $\epsilon$ such that on $[t_0, \infty) \times X$,

$$ \ddt{\varphi} \leq C e^{-t} ( 1 + t - \epsilon \log |S_{\tilde{E}}|_{h_{\tilde{E}}}^2).$$

The lemma is proved by letting $\epsilon \rightarrow 0$.

\qed\end{proof}

\begin{corollary}\label{volcon}

\begin{equation}
\lim_{t\rightarrow \infty} || \ddt{\varphi}||_{L^1(X)} = 0.
\end{equation}

\end{corollary}

\begin{proof} There exists $T>0$, such that $\ddt{\varphi} \leq e^{-t/2}$ for all $t\geq T$. Notice that $\int_{t_1}^{t_2}\ddt{\varphi} (t,z) dt = \varphi(t_2, z) - \varphi(t_1, z)$ is uniformly bounded for all $t_1, t_2 \geq 0$ and $\ddt{\varphi} - e^{-t/2} \leq 0 $ for all $t\geq T$. Then

\begin{equation}\label{l1con}
\int_T^\infty || \ddt{\varphi}||_{L^1(X)} dt < \infty.
\end{equation}

On the other hand, $\ddt{}(\varphi + e^{-t/2})\leq 0$ and so $\varphi+ e^{-t/2}$ is decreasing in time. Since $\varphi$ is uniformly bounded, there exists $\varphi_\infty \in PSH(X, \chi)\cap L^\infty(X)$ such that $\varphi$ converges to $\varphi_\infty$ in $L^1(X)$ and $C^\infty(X^{\circ})$ as $t\rightarrow \infty$. Hence $\ddt{\varphi}$ converges to a function $F\in L^\infty(X)\cap C^\infty(X^\circ)$ in $C^\infty(X^\circ)\cap L^1(X)$. Combined with (\ref{l1con}), $F=0$. Otherwise, $\int_T^\infty || \ddt{\varphi}||_{L^1(X)} dt  = \infty$. The corollary is then proved.

\qed
\end{proof}

\begin{proposition}\label{baoshan} Let $\varphi_\infty \in PSH(X, \chi)\cap L^\infty(X)$ be the unique solution of the Monge-Amp\`ere equation
\begin{equation}\label{ke1}
(\chi + \ddbar\varphi_\infty)^n = e^{\varphi_{\infty}} \Omega.
\end{equation}
Then   $\varphi$ converges to $\varphi_\infty$  in $L^1(X)\cap C^\infty(X^\circ) $ as $t\rightarrow \infty$.

\end{proposition}

\begin{proof} Let $\varphi_\infty$ be the limit of $\varphi$ as $t\rightarrow \infty$. By Corollary \ref{volcon}, $\ddt{\varphi}$ converges to $0$, and so $\varphi_\infty$ must satisfy equation (\ref{ke1}). The uniqueness of $\varphi_\infty$ follows from the uniqueness of the solution to the equation (\ref{ke1}) as $\varphi_\infty\in PSH(X, \chi)\cap L^\infty(X)$.

\qed\end{proof}

The K\"ahler current $\chi+\ddbar \varphi_{\infty}$ is exactly the pullback of the unique K\"ahler-Einstein metric  $\omega_{KE}$ on the canonical model $X_{can}$ of $X$ in Theorem \ref{canmetric}. The following theorem then follows from Proposition \ref{baoshan}.

\begin{theorem}\label{gtcon} Let $X$ be a minimal model of general type with log terminal singularities. For any $\mathbf{Q}$-ample divisor $H$ on $X$, the normalized weak K\"ahler-Ricci flow converges to the unique K\"ahler-Eintein metric $\omega_{KE}$ on the canonical model $X_{can}$ for any initial K\"ahler current in $\mathcal{K}_{H, p}(X)$ with $p>1$.

\end{theorem}

We have the following general theorem by combining Theorem \ref{longtime1} and Theorem \ref{gtcon} if the general type variety is not minimal.

\begin{theorem} Let $X$ be a projective $\mathbf{Q}$-factorial variety of general type with log terminal singularities. If there exists  a good initial divisor $H$ on $X$, then the normalized weak K\"ahler-Ricci flow starting with any initial K\"ahler current in $\mathcal{K}_{H, p}(X)$ with $p>1$ exists for $t\in [0,\infty)$ and   replaces $X$ by its minimal model $X_{min}$ after finitely many surgeries. Furthermore, the normalized K\"ahler-Ricci flow converges to the unique K\"ahler-Eintein metric $\omega_{KE}$ on its canonical model $X_{can}$.

\end{theorem}


\section{Analytic Minimal Model Program with Ricci Flow} \label{6}

In this section, we lay out the program relating the K\"ahler-Ricci flow and the classification of projective varieties following \cite{SoT1} and \cite{T3}. The new insight is that the Ricci flow is very likely to deform a given projective variety to its minimal model and eventually to its canonical model coupled with a canonical metric of Einstein type, in the sense of Gromov-Hausdorff. We will start discussions with the case of projective surfaces. 


\subsection{Results on surfaces} \label{6.1}

  A smooth projective surface is minimal if it does not contain any $(-1)$-curve. Let $X_0$ be an projective surface of non-negative Kodaira dimension. If $X_0$ is not minimal, then the unnormalized K\"ahler-Ricci flow starting with any K\"aher metric $\omega_0$ in the class of an ample divisor  $H_0$ has a smooth solution until the first singular time $T_0 = \sup \{ t>0~|~H_0 + T_0 K_{X_0} ~\textnormal{is ~nef}\}$. The limiting semi-ample divisor $H_0 + T_0 K_{X_0}$ induces a morphism$$\pi_0: X_0 \rightarrow X_1$$ by contracting finitely many $(-1)$-curves. $X_1$ is smooth and there exists an $\mathbf{Q}$-ample divisor $H_1$ on $X_1$ such that $H_0 + T_0 K_{X_0} = \pi_0^* H_1$. Then by Theorem \ref{thrudivcon}, the unnormalized K\"ahler-Ricci flow can be continued  through the contraction $\pi_0$ at time $T_0$. Since there are finitely many $(-1)$-curves on $X$, the unnormalized K\"ahler-Ricci flow will arrive at a minimal surface $X_{min}$ or it collapses a $\mathbf{CP}^1$ fibration in finite time after repeating the same surgery for finitely many times.

It is still a largely open question if the K\"ahler-Ricci flow converges to the new surface in the sense of Gromov-Hausdorff at each surgery. The only confirmed case is the K\"ahler-Ricci flow on $\mathbf{CP}^2 $ blow-up at one point. More precisely, it is shown in \cite{SW} that the unnormalized K\"ahler-Ricci flow on $\mathbf{CP}^2 $ blow-up at one point converges to $\mathbf{CP}^2$ in the sense of Gromov-Hausdorff  if the initial K\"ahler class is appropriately chosen and the initial K\"ahler metric satisfies the Calabi symmetry. Then the flow can be continued on $\mathbf{CP}^2 $ and  eventually will be contracted to a point in finite time. This shows that the K\"ahler-Ricci flow deforms the non-minimal surface to a minimal surface in the sense of Gromov-Hausdorff. Similar behavior is also shown in \cite{SW} for higher-dimensional analgues of the Hirzebruch surfaces. This leads us to propose a conjectural program in the following section for general projective varieties.

After getting rid of all $(-1)$-curves,  we can focus on  the minimal surfaces divided into ten classes by the Enriques-Kodaira classification.

If $\kod(X_{min})=2$, $X_{min}$ is a minimal surface of general type and its canonical model $X_{can}$ is an orbifold surface achieved by contracting all the $(-2)$-curves on $X_{min}$. It is shown in \cite{TiZha} that the normalized K\"ahler-Ricci flow $\ddt{\omega}= -Ric(\omega) - \omega$ converges in the sense of distributions to the pullback of the orbiford K\"ahler-Einstein metric on the canonical model $X_{can}$.

If $\kod(X_{min})=1$, $X_{min}$ is a minimal elliptic fibration over its canonical model $X_{can}$.  It is shown in \cite{SoT1} that the normalized K\"ahler-Ricci flow $\ddt{\omega}= -Ric(\omega) - \omega$ converges in the sense of distributions to the pullback of the generalized K\"ahler-Einstein metric on the canonical model $X_{can}$.

If $\kod(X_{min})=0$, $K_X$ is numerically trivial. Yau's solution to the Calabi conjecture shows that there always exists a Ricci-flat K\"ahler metric in any given K\"ahler class on $X_{min}$. In particular, it is shown in \cite{C} that the unnormalized K\"ahler-Ricci converges in the sense of distributions to the unique Ricci-flat metric in the initial K\"ahler class.

If $X$ is Fano, then it is proved in \cite{Pe2} and \cite{TiZhu} that the normalized K\"ahler-Ricci flow $\ddt{\omega}= -Ric(\omega) + \omega$ with an appropriate initial K\"ahler metric will converge in the sense of Gromov-Hausdorff to a K\"ahler-Ricci soliton after normalization.

In general, the understanding of the K\"ahler-Ricci flow is still not completely understood for surfaces of $-\infty$ Kodaira dimension as the flow might collapse in finite time.


\subsection{Conjectures} \label{6.2}

In this section, we discuss our program in higher dimensions. Our proposal gives new understanding of the Minimal Model Program from the viewpoint of differential geometry. We refer it as the analytic Minimal Model Program.

\bigskip

\noindent{\bf Minimal Model Program with Ricci Flow}

\begin{enumerate}

\item[ {\bf  1.} ] We start with a triple $(X, H, \omega)$, where $X$ is  a $\mathbf{Q}$-factorial projective variety with log terminal singularities,   $H$ is a big semi-ample $\mathbf{Q}$-divisor on $X$ such that $H+\epsilon K_X$ is ample for sufficiently small $\epsilon>0$, and $\omega \in \mathcal{K}_{H, p}(X)$ for some $p>1$.  Let $$T_0 = \inf \{ t>0~|~  H + t K_X~is ~nef \}.$$

Let $\omega(t,\cdot)$ be the unique solution of the unnormalized weak K\"ahler-Ricci flow for $t\in [0, T_0)$. 

\begin{conjecture}  For each $t\in (0, T_0)$, the metric completion of  $X_{reg}$ by $ \omega(t, \cdot)$ is homeomorphic to $X$.
\end{conjecture}

We also conjecture that the Ricci curvature of $\omega(t, \cdot)$ is bounded from below. 

\item[{\bf 2.}]  If $T_0 =\infty $, then $X$ is a minimal model and the K\"ahler-Ricci flow  has long time existence. The abundance conjecture predicts that $K_X$ is semi-ample and $\kod(X) \geq 0$.

\begin{enumerate}

\item[2.1.]  $\kod(X)=\dim X$, i.e.,  $X$ is a minimal model of general type.

\begin{conjecture}
The normalized K\"ahler-Ricci flow $$\frac{\partial \tilde{\omega} } {\partial s}= - Ric(\tilde{\omega}) -\tilde{\omega}.
$$ starting with $\omega$ converges to the unique K\"ahler-Einstein metric $\omega_{KE}$ on $X_{can}$ in the sense of Gromov-Hausdorff as $s\rightarrow \infty$. \end{conjecture}

The weak convergence in distribution and smooth convergence outside the exceptional locus is obtained in \cite{Ts} and \cite{TiZha}  if $X$ is nonsingular. If $X$ is nonsingular and $K_X$ is ample, it is the classical result in \cite{C} that the flow converges in $C^\infty(X)$-topology.

\item[2.2.] $0<\kod(X)<\dim X$.  
\begin{conjecture} The normalized K\"ahler-Ricci flow $$\frac{\partial \tilde{\omega} } {\partial s}= - Ric(\tilde{\omega}) -\tilde{\omega}.
$$ starting with $\omega$ converges to the unique generalized K\"ahler-Einstein metric $\omega_{can}$ on $X_{can}$ (as in Theorem \ref{canmetric}) in the sense of Gromov-Hausdorff  as $s\rightarrow \infty$. \end{conjecture}

If $K_X$ is semi-ample, $X$ admits a Calabi-Yau fiberation over its canonical model $X_{can}$.
The weak convergence in distribution is obtained in \cite{SoT1} and \cite{SoT2} if $X$ is nonsingular and $K_X$ is semi-ample.

\item[2.3.] $\kod(X)=0$. $K_X$ is numerically trivial.  \begin{conjecture} The unnormalized K\"ahler-Ricci flow $$\ddt{\omega} = - Ric(\omega) $$ converges to the unique Ricci-flat K\"ahler metric in $[H]$ in the sense of Gromov-Hausdorff as $t\rightarrow \infty$.
\end{conjecture}

It is shown in \cite{C} that the flow converges in $C^\infty(X)$-topology if $X$ is smooth. The weak convergence is obtained in \cite{SY} if $X$ has log terminal singularities.

\end{enumerate}

\item[{\bf  3.}] If $T_0<\infty$,  then the semi-ample divisor  $ H + T_0 K_X $ induces a contraction $$\pi: X \rightarrow Y. $$

\begin{enumerate}

\item[3.1] $\dim Y=\dim X$.

\begin{conjecture}   As $t\rightarrow T_0$, $ (X, \omega(t, \cdot))$ converges to a metric space $(X^+, \omega_{X^+})$ along the unnormalized K\"ahler-Ricci flow $$\ddt{\omega} = - Ric(\omega) $$ in the sense of 
Gromov-Hausdorff. Furthermore, $\omega(t, \cdot)$ converges in $C^\infty$ outside a subvariety $S$ of $X$, and $(X^+, \omega_{X^+})$ is the metric completion of the smooth limit of $(X\setminus S, \omega(T_0, \cdot))$ .  Here $X^+$ is a normal projective variety satisfying the following diagram
 \begin{equation}
\begin{diagram}\label{diag4}
\node{X} \arrow{se,b,}{\pi}  \arrow[2]{e,t,..}{(\pi^+)^{-1}\circ\pi }     \node[2]{X^+} \arrow{sw,r}{\pi^+} \\
\node[2]{Y}
\end{diagram}
\end{equation}
and  $X^+$ is $\pi^+$-ample. $ \pi^+: X^+ \rightarrow Y$ is a general flip of $X$.

Let $H_{X^+}$ be the strict transformation of $H+T_0K_X$ by the general flip. Then $K_{X^+}$ and $H_{X^+}$ are both $\mathbf{Q}$-Cartier with $H_{X^+}+\epsilon K_{X^+}$ being ample for sufficiently small $\epsilon>0$, and $\omega_{X^+}\in \mathcal{K}_{H_{X^+}, p}(X^+)$ for some $p'>1$.

\end{conjecture}

We then repeat {\bf Step 1} by replacing $(X,H,\omega)$ with $(X^+, H_{X^+}, \omega_{X^+})$ even though $X^+$ is not necessarily $\mathbf{Q}$-factorial. Note that a divisorial contraction is also a general flip if we choose $\pi^+$ to be the identity map.

\item[3.2] $0<\dim Y < \dim X$. $X$ then admits a Fano fibration over $Y$.
\begin{conjecture} As $t\rightarrow T_0$, $(X, \omega(t, \cdot))$ converges to a metric space $(Y', \omega_Y')$ along the unnormalized K\"ahler-Ricci flow $$\ddt{\omega} = - Ric(\omega) $$  in the sense of  Gromov-Hausdorff. Let $H_{Y'}$ be the divisor where $\omega_{Y'}$ lies. Then both $K_{Y'}$ and $H_{Y'}$ are $\mathbf{Q}$-Cartier, and $\omega_{Y'} \in \mathcal{K}_{H_{Y'}, p'} (Y')$ for some $p'>1$.

\end{conjecture}

We then repeat {\bf Step 1} by replacing $(X,H,\omega)$ by $(Y', H_{Y'}, \omega_{Y'})$.

\item [3.3] If $\dim Y=0$,  $X$ is Fano and $\omega \in -T_0 [K_X]$.  %

Then we have the following generalized Hamilton-Tian conjecture.
\begin{conjecture}\label{htcon} Then the normalized K\"ahler-Ricci flow $$\frac{\partial \tilde{\omega} } {\partial s}= - Ric(\tilde{\omega}) + \frac{1}{T_0} \tilde{\omega}.
$$ starting with $\omega$ converges to a K\"ahler-Ricci soliton $(X_\infty, \omega_{KR})$ in the sense of Gromov-Hausdorff as $s\rightarrow \infty$.

\end{conjecture}

 Perelman \cite{Pe2} announced a proof for this conjecture for K\"ahler-Einstein manifolds. A proof is given for Fano manifolds with  a K\"ahler-Ricci soliton by Tian-Zhu \cite{TiZhu}.

\medskip

It is conjectured by Yau \cite{Y2} that the existence of a K\"ahler-Einstein metric on a Fano manifold is equivalent to suitable stability in the sense of geometric invariant theory. The condition of $K$-stability is proposed by Tian \cite{T1} and is refined by Donaldson \cite{Do}. The Yau-Tian-Donaldson conjecture claims that the existence of K\"ahler metrics with constant scalar curvature is equivalent to the $K$-stability (possibly with some additional milder conditions on holomorphic vector fields). Since the K\"ahler-Ricci flow provides an approach to such a conjecture for K\"ahler-Einstein metrics and it has attracted considerable current interest. We refer the readers to an incomplete list of literatures \cite{PS1}, \cite{PS2}, \cite{TiZhu}, \cite{PSSW1}, \cite{PSSW2}, \cite{Sz} and \cite{To}  for some recent development.  
\end{enumerate}

\end{enumerate}

\bigskip
\noindent
{\bf Acknowledgements.} \  The authors would particularly like to thank Chenyang Xu for many inspiring discussions and for bringing MMP with scaling to the authors' attention. They would also like to thank Valentino Tosatti for a number of  helpful suggestions on a previous draft of the paper. The first named author is  grateful to Professor D.H. Phong for his advice, encouragement and support. He also wants to thank Yuan Yuan for some helpful discussions and comments.


\end{document}